\documentclass{amsart}
\usepackage[main=english,french]{babel}
\usepackage[utf8]{inputenc}
\usepackage{amsmath}
\usepackage{amsfonts}
\usepackage{amssymb}
\usepackage{amsthm}
\usepackage{amscd}
\usepackage{float}
\usepackage{tikz}
\usepackage{graphicx}
\usepackage[colorlinks=true]{hyperref}
\hypersetup{urlcolor=blue, citecolor=red}
\usepackage{hyperref}

\newtheorem{theorem}{Theorem}[section]
\newtheorem{corollary}[theorem]{Corollary}

\newtheorem{lemma}[theorem]{Lemma}
\newtheorem{proposition}[theorem]{Proposition}

\theoremstyle{definition}
\newtheorem{definition}[theorem]{Definition}
\newtheorem{remark}[theorem]{Remark}

\newcommand{\R}{\mathbb{R}}
\newcommand{\N}{\mathbb{N}}
\newcommand{\C}{\mathbb{C}}
\newcommand{\T}{\mathbb{T}}
\newcommand{\Z}{\mathbb{Z}}

\def\XXint#1#2#3{{\setbox0=\hbox{$#1{#2#3}{\int}$ }
\vcenter{\hbox{$#2#3$ }}\kern-.6\wd0}}

\makeatletter
\@namedef{subjclassname@2020}{%
  \textup{2020} Mathematics Subject Classification}
\makeatother

\begin{document}

\title[The Cauchy problem for the periodic KP-II equation below $L^2$]{The Cauchy problem for the periodic Kadomtsev--Petviashvili--II equation below $L^2$}

\author[S.~Herr]{Sebastian Herr}
\address[S.~Herr]{Fakultat f\"ur Mathematik, Universit\"at Bielefeld, Postfach 10 01 31, 33501 Bielefeld, Germany}
\email{herr@math.uni-bielefeld.de}

\author[R.~Schippa]{Robert Schippa}
\address[R.~Schippa]{UC Berkeley, Department of Mathematics, 847 Evans Hall
Berkeley, CA 94720-3840}
\email{rschippa@berkeley.edu}
\author[N.~Tzvetkov]{Nikolay Tzvetkov}
\address[N.~Tzvetkov]{Ecole Normale Sup\'erieure de Lyon and IUF, Unit\'e de Math\'ematiques Pures et Appliqu\'es,
UMR CNRS 5669, Lyon, France}
\email{nikolay.tzvetkov@ens-lyon.fr}

\maketitle

\begin{abstract}
We extend Bourgain's $L^2$-well-posedness result for the KP-II equation on $\T^2$ to initial data with negative Sobolev regularity. The key ingredient is a new linear $L^4$-Strichartz estimate which is effective on frequency-dependent time scales. The $L^4$-Strichartz estimates follow from combining an $\ell^2$-decoupling inequality recently proved by Guth--Maldague--Oh with semiclassical Strichartz estimates.
Moreover, we rely on a variant of Bourgain's bilinear Strichartz estimate on frequency-dependent times, which is proved via a biorthogonality argument akin to the proof of the Córdoba--Fefferman square function estimate.
\end{abstract}

\selectlanguage{french}

\medskip

\begin{center}
\textbf{
Titre : Le problème de Cauchy pour l’équation KP-II périodique dans des espaces de Sobolev d'indice négatif}
\end{center}

\begin{abstract}
On généralise le résultat de Bourgain démontrant le caractère bien posé dans $L^2$ de l'équation KP-II périodique à données initiales initiales distributions de Schwartz appartenant à des espaces de Sobolev d'indices négatifs. Un élément clé de la preuve est une nouvelle estimation de Strichartz sur des intervalles de temps dépendant de la localisation en fréquence de la condition initiale. Cette nouvelle estimation de Strichartz est le résultat de la combinaison d'une inégalité récente de presque orthogonalité obtenue par Guth--Maldague--Oh et des estimées de Strichartz semi-classiques dans l'esprit des travaux de Staffilani-Tatatru et Burq-Gérard-Tzvetkov. Nous utilisons également une variante d'une estimation de Strichartz bilinéaire due à Bourgain, sur des intervalles de temps dépendant de la fréquence spatiale. Cette estimation utilise des arguments de bi-orthogonalité dans l'esprit de la preuve de l'estimée de la fonction carré de  Córdoba--Fefferman.

\end{abstract}

\selectlanguage{english}

\section{Introduction}

In this article we are concerned with the local well-posedness of the KP-II equation on the torus $\T^2$, where $\T$ is the circle $\R / (2 \pi \Z)$:
\begin{equation}
\label{eq:KPII}
\left\{ \begin{array}{cl}
\partial_t u + \partial_x^3 u + \partial_{y}^2 \partial_{x}^{-1} u &= u \partial_x u, \quad (t,x,y) \in \R \times \T^2, \\
u(0) &= u_0 \in H^{s,0}(\T^2).
\end{array} \right.
\end{equation}
Above $u_0$ denotes a real-valued distribution on $\T^2$ with $\hat{u}_0(0,\eta) = 0$ in the anisotropic Sobolev space $H^{s,0}(\T^2)$ with norm
\begin{equation}
\label{eq:AnisotropicSobolevSpace}
\| f \|^2_{H^{s,0}(\T^2)} = \sum_{(\xi,\eta) \in \Z^2} \langle \xi \rangle^{2s} |\hat{f}(\xi,\eta)|^2.
\end{equation}
With the vanishing condition on the Fourier coefficients, we can explain $\partial_x^{-1}$ as a Fourier multiplier
\begin{equation*}
\widehat{\partial_x^{-1} f} (\xi,\eta) = (i \xi)^{-1} \hat{f}(\xi,\eta).
\end{equation*}
The vanishing mean condition is propagated by the nonlinear evolution.

\smallskip

Bourgain proved analytic local well-posedness of \eqref{eq:KPII} for $u_0 \in L^2(\T^2)$ in \cite{Bourgain1993}. Since the $L^2$-norm is conserved, global well-posedness in $L^2$ follows for real-valued initial data. Here we seek to improve the result on $\T^2$ in \cite{Bourgain1993} by showing local well-posedness with regularity \emph{below $L^2$}, i.e., $s<0$ in \eqref{eq:KPII}. 

\smallskip

By local well-posedness we refer to existence, uniqueness, and continuity of the data-to-solution mapping $S_T: H^{s,0}(\T^2) \to C([0,T],H^{s,0}(\T^2))$, $ S_T(u_0) = u $ with $T=T(s,\| u_0 \|_{H^{s,0}(\T^2))})$. We obtain the data-to-solution mapping as unique continuous extension of the data-to-solution mapping on $L^2(\T^2)$, which allows us to work with smooth solutions. The analytic local well-posedness in \cite{Bourgain1993} is a consequence of the contraction mapping principle. Presently, we combine the Fourier restriction analysis \cite{Bourgain1993} with the energy method on frequency-dependent time scales \cite{KochTzvetkov2003,IonescuKenigTataru2008}. We implement this argument via short-time Fourier restriction norms as introduced by Ionescu--Kenig--Tataru. Notably, Koch--Tzvetkov \cite{KochTzvetkov2003} previously combined energy arguments with Strichartz estimates on frequency-dependent time scales to show low regularity well-posedness of the Benjamin--Ono equation.

\smallskip

 We remark that on the Euclidean space the KP-II equation has the scaling symmetry
\begin{equation*}
u \to u_\lambda(t,x,y) = \lambda^2 u(\lambda^3 t, \lambda x, \lambda^2 y).
\end{equation*}
On $\R^2$ this distinguishes the scaling-critical Sobolev space $\dot{H}^{-\frac{1}{2},0}(\R^2)$. Hadac--Herr--Koch \cite{HadacHerrKoch2009} showed global well-posedness and scattering for initial data in $\dot{H}^{-\frac{1}{2},0}(\R^2)$ under suitable smallness assumption. 

A few more remarks on the model are in order:
The KP-II equation was proposed by Kadomtsev and Petviashvili \cite{Kadomtsev1970} as model for two-dimensional water waves (see also Ablowitz-Segur \cite{AblowitzSegur1979}). Solutions to \eqref{eq:KPII} which do not depend on the $y$-variable solve the KdV equation:
\begin{equation*}
\partial_t u + \partial_{xxx} u = u \partial_x u.
\end{equation*}
Understanding the stability of solitons to the KdV equation under transverse perturbations was a central question upon the proposal of KP-equations.

\medskip

Like the KdV equation, the KP-equations are known to be completely integrable. A Lax pair formulation was obtained by Dryuma \cite{Dryuma1974} and as a consequence an infinite number of conserved quantities. Two conserved quantities for \eqref{eq:KPII} read
\begin{equation*}
\begin{split}
M &= \frac{1}{2} \int_{\T^2} u^2 dx dy, \\
E &= \frac{1}{2} \int_{\T^2} \big((\partial_x u)^2 - \frac{1}{3} u^3 - (\partial_x^{-1} \partial_y u)^2\big) dx dy.
\end{split}
\end{equation*}

Due to a lack of coercivity, which is evident for $E$, the mass $M$ seems to be the only conservation law, which can be used to show global well-posedness. Recently, complete integrability techniques have been successfully employed to establish low regularity well-posedness in Sobolev spaces for various equations: Firstly, a priori estimates up to the scaling critical regularity \cite{KochTataru2018,KillipVisanZhang2018} were proved and later even global well-posedness \cite{KillipVisan2019,HarropGriffithsKillipVisan2024} via the method of commuting flows introduced by Killip-Vi\c{s}an. Whether the arguments extend to two-dimensional models remains open. Our approach does not rely on the complete integrability; the symmetries it depends on are the real-valuedness of solutions and a conservative derivative nonlinearity.

\smallskip

In \cite{Bourgain1993} Fourier restriction spaces were used (previously introduced in \cite{Bourgain1993A,Bourgain1993B}) to show analytic well-posedness in $L^2$. We use short-time Fourier restriction spaces introduced in \cite{IonescuKenigTataru2008}. A second key ingredient is an $L^4_{t,x,y}$-Strichartz estimate with sharp (up to endpoints) derivative loss. This is based on the very recent work by Guth--Maldague--Oh \cite{GuthMaldagueOh2024} on $\ell^2$-Fourier decoupling. It is combined with the dispersive inequality on a semiclassical time scale.
This novel approach to low-regularity well-posedness of quasilinear dispersive equations does apply to other models as well, see e.g. the recent paper of the second author \cite{Schippa2023Refinements}.
 
\smallskip
In the non-periodic setting the Strichartz estimate
\[
\|e^{-t(\partial_x^3+\partial_y^2\partial_x^{-1})}u_0\|_{L^4(\R\times\R^2)}\lesssim \|u_0\|_{L^2(\R^2)},
\]
which was obtained by Saut in \cite{Saut93}, plays a fundamental role \cite{HadacHerrKoch2009}. The example in Subsection \ref{subsec:sharp} shows that in the periodic setting the estimate
\[
\|e^{-t(\partial_x^3+\partial_y^2\partial_x^{-1})}u_0\|_{L^4([0,1]\times\T^2)}\lesssim \|u_0\|_{H^{s,0}(\T^2)}
\]
fails for $s<1/8$, while in Theorem \ref{thm:ImprovedL4Strichartz} we prove that this estimate holds true for $s>1/8$ if the transversal frequencies are not too large.
\smallskip

In the following we survey the key ingredients of our analysis. Denote the dispersion relation by
\begin{equation*}
\omega(\xi,\eta) = \xi^3 - \frac{\eta^2}{\xi}.
\end{equation*}
A key quantity is the resonance function:
\begin{equation*}
\begin{split}
\Omega(\xi_1,\eta_1,\xi_2,\eta_2) &= \omega(\xi_1 + \xi_2, \eta_1 + \eta_2) - \omega(\xi_1,\eta_1) - \omega(\xi_2,\eta_2) \\
&= 3 \xi_1 \xi_2 (\xi_1+\xi_2) + \frac{(\eta_1 \xi_2 - \eta_2 \xi_1)^2}{\xi_1 \xi_2 (\xi_1+\xi_2)}.
\end{split}
\end{equation*}
Since we impose a vanishing-mean condition on the initial data, which is preserved under the nonlinear evolution, the frequencies satisfy $\xi_i \neq 0$, $\xi_1 + \xi_2 \neq 0$.

\smallskip

We let $\Omega_{\mathrm{KdV}}(\xi_1,\xi_2)= 3 \xi_1 \xi_2 (\xi_1 + \xi_2)$. Suppose that $|\xi_i| \sim N_i$ and $|\xi_1+\xi_2| \sim N_3$ for dyadic numbers $N_i \in 2^{\N_0}$. We note that we always have
\begin{equation}
\label{eq:ResonanceLowerBound}
|\Omega| \gtrsim N_{\max}^2 N_{\min}, \quad N_{\max} = \max_{i=1,2,3} N_i, \; N_{\min} = \min_{i=1,2,3} N_i,
\end{equation}
which was one key aspect in \cite{Bourgain1993}. For comparison consider the KP-I equation
\begin{equation*}
    \partial_t u +\partial_x^3 u - \partial_x^{-1} \partial_y^2 u = u \partial_x u
\end{equation*}
with dispersion relation given by
\begin{equation*}
\bar{\omega}(\xi,\eta) = \xi^3 + \frac{\eta^2}{\xi},
\end{equation*}
in which case the resonance relation reads
\begin{equation*}
\bar{\Omega}(\xi,\eta)= 3 \xi_1 \xi_2 (\xi_1 + \xi_2) - \frac{(\eta_1 \xi_2 - \eta_2 \xi_1)^2}{\xi_1 \xi_2 (\xi_1 + \xi_2)}.
\end{equation*}
In conclusion, in the KP-I case the resonance function can become much smaller than the KdV resonance relation. The lower bound \eqref{eq:ResonanceLowerBound} manifests a defocusing effect in the nonlinear evolution of the KP-II equation, which makes the nonlinear evolution better behaved than the related KP-I equation. It is known that the latter cannot be solved via the contraction mapping principle on $\R^2$ (see \cite{MolinetSautTzvetkov2002Illposedness}). Indeed, the global well-posedness of the KP-I equation on $\R^2$ in the energy space was established by Ionescu--Kenig--Tataru \cite{IonescuKenigTataru2008} combining Fourier restriction analysis with the energy method on frequency-dependent time scales. In addition, it is noteworthy that  the characteristic surface given by the dispersion relation for the KP-II equation --contrary to the  KP-I case-- is non-elliptic, which makes the proof of the periodic Strichartz estimates more difficult.

\smallskip

For the following discussion let us focus on the resonant case $|\Omega| \sim N_{\max}^2 N_{\min}$ and suppose by symmetry that $N_1 \sim N_3$, $N_2 \lesssim N_1$. A simplified version of the bilinear Strichartz estimate due to Bourgain (see Proposition \ref{prop:SimplifiedBilinearStrichartz} for the precise statement) reads as a convolution estimate in Fourier variables as follows:
\begin{equation}
\label{eq:BilinearStrichartzIntro}
\| f_{1,N_1,L_1} * f_{2,N_2,L_2} \|_{L^2_{\tau,\xi,\eta}} \lesssim N_2^{\frac{1}{2}} L_{\min}^{\frac{1}{2}} \langle L_{\max} / D \rangle^{\frac{1}{2}} \prod_{i=1}^2 \| f_i \|_{L^2_{\tau,\xi,\eta}}.
\end{equation}
Above $f_{i,N_i,L_i} \in L^2(\R \times \Z^2)$ with
\begin{equation*}
\operatorname{supp} (f_{i,N_i,L_i}) \subseteq \{ (\tau,\xi,\eta) \in \R^3 : |\xi| \sim N_i, \; |\tau - \omega(\xi,\eta)| \lesssim L_i \}.
\end{equation*}
We denote $L_{\max} =\max( L_1 ,L_2)$, $L_{\min} = \min( L_1 ,L_2)$, and  $D$ denotes transversality in the $\eta$-frequencies within the supports of $f_{1,N_1,L_1}$ and $f_{2,N_2,L_2}$:
\begin{equation}\label{eq:trans}
\big| \frac{\eta_1}{\xi_1} - \frac{\eta - \eta_1}{\xi - \xi_1} \big| \gtrsim D.
\end{equation}
More precisely,
\begin{equation*}
\partial_{\eta_1} \Omega(\xi_1,\eta_1,\xi-\xi_1,\eta-\eta_1) = 2 \big( \frac{\eta_1}{\xi_1} - \frac{\eta - \eta_1}{\xi - \xi_1} \big),
\end{equation*}
which leads to estimates on the number of contributing $\eta$-frequencies for fixed $\xi$-frequencies. Using endpoint versions of the Fourier restriction norm spaces yields always a gain of $(N_1^2 N_2)^{-\frac{1}{2}}$ by the lower bound on the resonance function. This compensates for the derivative loss giving rise to a factor of $N_1$. Together with \eqref{eq:BilinearStrichartzIntro} this gives an (oversimplified) explanation to $L^2$-well-posedness.

\medskip

To break the $L^2$-barrier set in \cite{Bourgain1993}, we add the following ingredients:
\begin{itemize}
\item[(i)] Frequency-dependent time localization $T=T(N_1)=N_1^{-\alpha}$, which improves the bilinear Strichartz estimates in case of a High-Low interaction $N_2 \ll N_1$ with large transversality $D \gtrsim N_1$,
\item[(ii)] Bilinear Strichartz estimates, which improve \eqref{eq:BilinearStrichartzIntro} in case of High-Low-\-inter\-action $N_2 \ll N_1$ with small transversality $D \ll N_1$ by using a Córdoba--Fefferman \cite{Cordoba1979,Cordoba1982,Fefferman1973} square function estimate (in \cite[Section~4]{Bourgain1993} more direct counting arguments are used),
\item[(iii)] Novel linear $L^4_{t,x,y}$-Strichartz estimates, which we show using $\ell^2$-decoupling. These are most useful in case of resonant High-High-interaction $N_1 \sim N_2 \sim N_3$.
\end{itemize}

Following Ionescu--Kenig--Tataru \cite{IonescuKenigTataru2008}, we need to establish estimates for the solutions in short-time Fourier restriction spaces $F^s(T)$. The ``dual'' spaces to estimate the nonlinearity are denoted by $\mathcal{N}^s(T)$. Since frequency-dependent time localization erases the dependence of the solution on the initial data for high frequencies, we need to establish energy estimates in spaces $E^s(T)$, which estimate the $L^\infty$-norm in time of every dyadic frequency component. For solutions the estimates read as follows:
\begin{equation*}
\left\{ \begin{array}{cl}
\| u \|_{F^s(T)} &\lesssim \| u \|_{E^s(T)} + \| \partial_x (u^2) \|_{\mathcal{N}^s(T)}, \\
\| \partial_x(u^2) \|_{\mathcal{N}^s(T)} &\lesssim T^\delta \| u \|^2_{F^s(T)}, \\
\| u \|^2_{E^s(T)} &\lesssim \| u_0 \|_{H^{s,0}(\T^2)}^2 + T^\delta \| u \|^3_{F^s(T)}.
\end{array} \right.
\end{equation*}
For a time localization parameter $\alpha \in (0,\alpha^*)$, with threshold $\alpha^* >0$ to be determined below, it turns out that the above estimates hold true for $s \geq - \frac{\alpha}{2} + \varepsilon$.
This gives a priori estimates
\begin{equation*}
\sup_{t \in [0,T]} \| u(t) \|_{H^{s,0}(\T^2)} \lesssim \| u \|_{F^s(T)} \lesssim \| u_0 \|_{H^{s,0}(\T^2)}.
\end{equation*}

\smallskip

Secondly, we show estimates for differences of solutions $v = u_1 - u_2$, which solve the difference equation
\begin{equation*}
\partial_t v + \partial_x^3 v + \partial_{x}^{-1} \partial_{y}^2 v = \partial_x(v(u_1+u_2))
\end{equation*}
 at lower regularity $s' = 20s$:
\begin{equation*}
\left\{ \begin{array}{cl}
\| v \|_{F^{s'}(T)} &\lesssim \| v \|_{E^{s'}(T)} + \| \partial_x (v(u_1+u_2)) \|_{N^{s'}(T)}, \\
 \| \partial_x (v(u_1+u_2)) \|_{N^{s'}(T)} &\lesssim T^\delta \| v \|_{F^{s'}(T)} ( \| u_1 \|_{F^s(T)} + \| u_2 \|_{F^s(T)} ), \\
 \| v \|^2_{E^{s'}(T)} &\lesssim \| v(0) \|_{H^{s',0}}^2 + T^\delta \| v \|_{F^{s'}(T)}^2 ( \| u_1 \|_{F^s(T)} + \| u_2 \|_{F^s(T)}).
\end{array} \right.
\end{equation*}
With a priori estimates for $\| u_i \|_{F^s(T)} \lesssim \| u_i(0) \|_{H^{s,0}(\T^2)}$ at hand, we find for times $T=T(\| u_i(0) \|_{H^{s,0}(\T^2)})$:
\begin{equation*}
\| v \|_{F^{s'}(T)} \lesssim \| v(0) \|_{H^{s',0}(\T^2)}.
\end{equation*}
This shows Lipschitz continuity of the data-to-solution mapping in the $H^{s'}$-topology with more regular solutions $u_i(0) \in H^{s,0}(\T)$. To infer continuous dependence, but no uniform continuous dependence, we invoke the frequency envelope argument by Tao \cite{Tao2001}. For details we refer to the recent treatise by Ifrim--Tataru \cite{IfrimTataru2023}. 

\smallskip

A simplified version of our local well-posedness result reads as follows:
\begin{theorem}[Local well-posedness below $L^2$]
\label{thm:LWPKPII}
Let $s > - \frac{1}{90}$. Then \eqref{eq:KPII} is locally well-posed.
\end{theorem}
A more detailed version is provided in Section \ref{section:ImprovedLWP}, where the proof is concluded with the short-time estimates at hand. The argument does not yield analyticity of the flow map, and the precise regularity of the flow map in Sobolev spaces of negative order remains open (see also Remark \ref{rem:quasi}). The argument in \cite{Gruen09} shows that for
$s<-1/4$ one has $C^3$-illposedness in $H^{s,0}(\T^2)$, see Section \ref{sect:ex}.

\medskip

Finally, we remark that our result does not depend on the periods of the torus. Let $\gamma \in (\frac{1}{2},1]$, $\T_\gamma = \R / (2 \pi \gamma \Z)$, and $\T^2_\gamma = \T \times \T_\gamma$. We can as well consider \eqref{eq:KPII} on the possibly irrational torus:
\begin{equation}
\label{eq:KPIIIrr}
\left\{ \begin{array}{cl}
\partial_t u + \partial_x^3 u + \partial^2_{y} \partial_x^{-1} u &= u \partial_x u, \quad (t,x,y) \in \R \times \T^2_\gamma, \\
u(0) &= u_0 \in H^{s,0}(\T^2_\gamma).
\end{array} \right.
\end{equation}
Since the novel $L^4$-Strichartz estimates are proved  using $\ell^2$-decoupling, these hold as well on the irrational torus. It is easy to see that the bilinear Strichartz estimates from Section \ref{section:BilinearStrichartz} hold as well on $\T^2_\gamma$. Hence, we obtain the same regularity results for \eqref{eq:KPIIIrr} like for \eqref{eq:KPII}. In the follow-up work \cite{HerrSchippaTzvetkov2025} we have extended the arguments to cylinders $\R \times \T$. Here one additionally has to take into account the possibility of small frequencies $|\xi| \ll 1$ as already outlined by Bourgain \cite[Section~8]{Bourgain1993} when he extended his result to $\R^2$.

\medskip

\emph{Outline of the paper.} In Section \ref{section:Notations} we introduce the notation and define the short-time Fourier restriction function spaces. We recall basic properties of the short-time Fourier restriction norms. In Section \ref{section:StrichartzEstimates} we show new linear Strichartz estimates using $\ell^2$-decoupling. Bilinear Strichartz estimates are recalled in Section \ref{section:BilinearStrichartz}. In Section \ref{section:Outline} we explain how the novel $L^4$-Strichartz estimates is combined with bilinear estimates and frequency-dependent time localization to prove the new well-posedness result.
 In Section \ref{section:ShorttimeNonlinear} we show the short-time bilinear estimates for solutions and differences of solutions.
 In Section \ref{section:Trilinear} we show trilinear convolution estimates for functions localized in frequency and modulation, which prepares the proof of short-time energy estimates carried out in Section \ref{section:ShorttimeEnergy}.
 In Section \ref{section:ImprovedLWP} we give an expanded version of Theorem \ref{thm:LWPKPII} and finish the proof with the short-time estimates at hand. In Section \ref{sect:ex} we provide an example showing sharpness of the bilinear estimate and an ill-posedness result, which is based on \cite{Gruen09}.

\medskip

\textbf{Basic notation:}
\begin{itemize}
\item  For $A \in \R$ we denote $\langle A \rangle = 1 + |A|$; for $A,B \in \R$ we let $A \vee B = \max(A,B)$ and $A \wedge B = \min(A,B)$.
\item For $A,B \in \R_{>0}$ the notation $A \lesssim B$ means that there is $C \geq 1$, only depending on inessential parameters, such that $A \leq C B$. We write $A \sim B$ provided that $A \lesssim B$ and $B \lesssim A$. Similarly, $A \ll B$ means that there is $0< c < 1$ such that $A \leq c B$.
\item The variables in position space are denoted by $(t,x,y) \in \R \times \T \times \T$. The dual variables in frequency space are denoted by $(\tau,\xi,\eta) \in \R \times \Z \times \Z$.
\item By $B_d(x_0,R) = \{ x \in \R^d : |x-x_0| \leq R \}$ we denote closed balls with radius $R > 0$.
\item The circle is denoted by $\T = \R / (2\pi \Z)$ and for $\gamma \in (1/2,1]$ we let $\T_\gamma = \R / (2 \pi \gamma \Z)$.
\item The $L^p$-norm on $\T^k$, $k \geq 1$, $p \in (1,\infty)$ is given by 
\begin{equation*}
\| f \|_{L^p(\T^k)}^p = \int_{[0,2\pi]^k} |f(x)|^p dx
\end{equation*}
with the usual modification for $p=\infty$.
\item Anisotropic Sobolev spaces for $s \in \R$ are comprised of the distributions with finite norm:
\begin{equation*}
H^{s,0}(\T^2) = \{ f \in \mathcal{S}'(\T^2) : \| f \|^2_{H^{s,0}(\T^2)} = \sum_{(\xi,\eta) \in \Z^2} \langle \xi \rangle^{2s} |\hat{f}(\xi,\eta)|^2 < \infty \}.
\end{equation*}
\end{itemize}

\section{Notation and Function spaces}
\label{section:Notations}

\subsection{Fourier transform and Littlewood-Paley decomposition}

The space-time Fourier transform is given by
\begin{equation*}
(\mathcal{F}_{t,x,y} u)(\tau,\xi,\eta) = \hat{u}(\tau,\xi,\eta) = \int_{\R \times \T^2} e^{-i t \tau } e^{-i x \xi} e^{-i y \eta} u(t,x,y) \, dx \, dy \, dt.
\end{equation*}
Let $f: \R \times \Z^2 \to \C$, $f \in L^1_{\tau} \ell^1_{\xi,\eta}$. The inverse Fourier transform is defined by
\begin{equation*}
u(t,x,y) = \mathcal{F}_{t,x,y}^{-1} [f] = \frac{1}{(2\pi)^3} \int_{\R } \sum_{(\xi,\eta) \in \Z^2} e^{i t \tau} e^{i x \xi} e^{i y \eta} f(\tau,\xi,\eta).
\end{equation*}

Let $\eta_0 \in C^\infty_c(B(0,2))$ with $\eta_0 \equiv 1$ on $[0,1]$ and $\eta_0$ be radially decreasing.
We define 
\begin{equation*}
\eta_1(\tau) = \eta_0(\tau/2) - \eta_0(\tau).
\end{equation*}
For $L \in 2^{\mathbb{N}}$ we set
\begin{equation*}
\eta_{L}(\tau) = \eta(\tau/L).
\end{equation*}

Let $N \in 2^{\N_0}$ denote a dyadic number. We define frequency projections (with respect to the $\xi$-frequencies) $P_N:L^2(\R \times \T^2) \to L^2(\R \times \T^2)$ by
\begin{equation*}
 \widehat {P_N u} (\tau,\xi,\eta) = 1_{[N,2N)}(\xi) \hat{f}(\tau,\xi,\eta).
\end{equation*}

\subsection{Definition of function spaces}
In this subsection, we introduce all function spaces which are used in the subsequent analysis and some of their elementary and well-known properties, we also refer to \cite{Bourgain1993,IonescuKenigTataru2008}.

The dispersion relation is given by
\begin{equation*}
\omega(\xi,\eta) = \xi^3 - \frac{\eta^2}{\xi}.
\end{equation*}

Let $u_0 \in L^2(\T^2;\C)$ and suppose in the following that $\hat{u}_0(0,\eta) = 0$ for all $\eta \in \Z$. This mean-zero condition is clearly preserved by the nonlinear evolution.
The linear propagation on $\T^2$ is defined by:
\begin{equation*}
S_{\mathrm{KP}}(t) u_0 = \sum_{\substack{(\xi,\eta) \in \Z^2 : \\ \xi \neq 0 }} e^{i(x \xi + y \eta + t \omega(\xi,\eta))} \hat{u}_0(\xi,\eta).
\end{equation*}

We define modulation projections for $L \in 2^{\N_0} \cup \{ 0 \}$ on functions $L^2(\R \times \T^2)$ by
\begin{equation*}
\widehat{Q_L u} (\tau,\xi,\eta)= \eta_L(\tau - \omega(\xi,\eta)) \hat{u}(\tau,\xi,\eta) \text{ and } Q_{\leq L} u = \sum_{K \leq L} Q_K u.
\end{equation*}
The sets describing the Fourier support are denoted by 
\begin{equation*}
    D_{N,L} = \{ (\tau,\xi,\eta) \in \R^3 : \, |\xi| \sim N, \; |\tau- \omega(\xi,\eta)| \sim L \}.
\end{equation*}
The sets $D_{N,\leq L}$ are defined correspondingly.

We introduce a modulation weight.
For $g:\R \times \Z^2 \to \C$ with $g(\tau,\xi,\eta) = 0$ unless $|\xi| \in [N,2N)$ we define
\begin{equation*}
\| g \|_{X_N} = \sum_{L \geq 1} L^{\frac{1}{2}} (1+L/N^3)^{\frac{1}{4}} \| \eta_L(\tau - \omega(\xi,\eta)) g \|_{L^2}.
\end{equation*}
The weight $(1+L/N^3)^{\frac{1}{4}}$ is supposed to improve estimates in the High-Low-inter\-action when the low frequency carries the high modulation and has already been used in \cite{Bourgain1993}. We define the norm without this weight by
\begin{equation}
\label{eq:FourierRestrictionBesov}
\| g \|_{\bar{X}_N} = \sum_{L \geq 1} L^{\frac{1}{2}} \| \eta_L(\tau - \omega(\xi,\eta)) g \|_{L^2}.
\end{equation}

For $\alpha > 0$ the short-time Fourier restriction spaces at frequencies $N \in 2^{\N_0}$ are defined by
\begin{equation*}
    F_N = \{ u \in C(\R;L^2) : P_N u = u, \; \| u \|_{F_N} < \infty \}
\end{equation*}
with
\begin{equation*}
\| u \|_{F_N} = \sup_{t_k \in \R} \| \mathcal{F}_{t,x,y} [ \eta_0(N^\alpha (t-t_k)) P_N u] \|_{X_N}.
\end{equation*}
This can be localized in time $T \in (0,1]$ in the usual way
\begin{equation*}
    F_N(T) = \{ u \in C([0,T];L^2) : P_N u = u, \; \| u \|_{F_N(T)} < \infty \}
\end{equation*}
with
\begin{equation*}
\| u \|_{F_N(T)} = \inf_{\tilde{u}\vert_{[0,T]} = u\vert_{[0,T]}} \| \tilde{u} \|_{F_N}.
\end{equation*}
We take the infimum over all frequency-localized functions $\tilde{u} \in C(\R;L^2)$ which coincide with $u$ on the interval $[0,T]$. 

The dual norm dictated by the Duhamel formula is given by
\begin{equation*}
\| u \|_{\mathcal{N}_N} = \sup_{t_k \in \R} \| (\tau - \omega(\xi,\eta) + i N^\alpha )^{-1} \mathcal{F}_{t,x,y} [\eta_0(N^{\alpha}(t-t_k)) P_N u ] \|_{X_N}
\end{equation*}
The time-localized space $\mathcal{N}_N(T)$ is defined like above.

\smallskip

The following lemma allows us to localize to shorter time intervals than required by the short-time norms:
\begin{lemma}
\label{lem:KernelWeightedNorms}
For $M \in 2^{\N_0}$, and $f_N \in \bar{X}_N$, the following estimate holds:
\begin{equation*}
\begin{split}
&\sum_{J \geq M} J^{\frac{1}{2}} \| \eta_J(\tau - \omega(\xi,\eta)) \int_{\R} \big| f_N(\tau,\xi,\eta) \big| M^{-1} (1+M^{-1} |\tau - \tau'|)^{-4} d \tau' \|_{L^2_{\tau,\xi,\eta}} \\
&\quad + M^{\frac{1}{2}} \| \eta_{\leq M}(\tau-\omega(\xi,\eta)) \int_{\R} |f_N(\tau',\xi,\eta) | M^{-1} (1 + M^{-1} |\tau - \tau'|)^{-4} d\tau' \|_{L^2_{\tau,\xi,\eta}} \\
&\lesssim \| f_N \|_{\bar{X}_N}
\end{split}
\end{equation*}
with implicit constant independent of $K$ and $L$.
\end{lemma}

\smallskip

The following is immediate:
\begin{lemma}
\label{lem:UnweightedXNEstimate}
For $\gamma \in \mathcal{S}(\R)$ and $M \in 2^{\N_0}$, and $t_0 \in \R$, we have
\begin{equation*}
\| \mathcal{F}_{t,x,y}[\gamma(M(t-t_0)) \mathcal{F}^{-1}_{t,x,y}(f_N)] \|_{\bar{X}_N}  \lesssim \| f_N \|_{\bar{X}_N}.
\end{equation*}
The implicit constant is independent of $M$, $N$, and $t_0 \in \R$.
\end{lemma}

To interpret the above, suppose we are estimating an expression $\gamma(T'^{-1} (t-t_0)) u_N$ which is localized at frequencies $N$ and on time scales $T' \ll T(N)=N^{-\alpha}$. When we do not use the weight $(1+L/N^3)^{\frac{1}{4}}$, we can estimate the expressions 
\begin{equation*}
\sum_{L \geq (T')^{-1}} L^{\frac{1}{2}} \| f_{N,L} \|_{L^2_{\tau,\xi,\eta}} \lesssim \| u_N \|_{F_N},
\end{equation*}
where 
\begin{equation*}
f_{N,L} = \begin{cases} 1_{D_{N,\leq T'}} \mathcal{F}_{t,x,y} [\gamma(T'^{-1} (t-t_0) u_N], &\quad L = T', \\
1_{D_{N,L}} \mathcal{F}_{t,x,y} [\gamma(T'^{-1} (t-t_0) u_N], &\quad L > T'.
\end{cases}
\end{equation*}

Suppose we want to take advantage of the low modulation weight. In comparison with the above, we can only dispose the higher time localization if it does not exceed the size of the modulation weight. We have the following variant of Lemma \ref{lem:KernelWeightedNorms} with straightforward proof, see e.g. \cite{HerrSchippaTzvetkov2025}.
\begin{lemma}
\label{lem:WeightedXNEstimate}
Under the assumptions of Lemma \ref{lem:KernelWeightedNorms}, for $M \lesssim N^3 \in 2^{\N_0}$, and $f_N \in X_N$, the following estimate holds:
\begin{equation*}
\begin{split}
&\sum_{J \geq M} J^{\frac{1}{2}} (1+J / N^3)^{\frac{1}{4}} \| \eta_J(\tau - \omega(\xi,\eta)) \int_{\R} \big| f_N(\tau,\xi,\eta) \big| M^{-1} (1+M^{-1} |\tau - \tau'|)^{-4} d \tau' \|_{L^2_{\tau,\xi,\eta}} \\
&\quad + M^{\frac{1}{2}} \| \eta_{\leq M}(\tau-\omega(\xi,\eta)) \int_{\R} |f_N(\tau',\xi,\eta) | M^{-1} (1 + M^{-1} |\tau - \tau'|)^{-4} d\tau' \|_{L^2_{\tau,\xi,\eta}} \\
&\lesssim \| f_N \|_{X_N}
\end{split}
\end{equation*}
with implicit constant independent of $K$ and $L$.
\end{lemma}

Correspondingly,
\begin{lemma}
For $\gamma \in \mathcal{S}(\R)$ and $M \lesssim N^3 \in 2^{\N_0}$, and $t_0 \in \R$, we have
\begin{equation*}
\| \mathcal{F}_{t,x,y}[\gamma(M(t-t_0)) \mathcal{F}^{-1}_{t,x,y}(f_N)] \|_{X_N}  \lesssim \| f_N \|_{X_N}.
\end{equation*}
The implicit constant is independent of $M$, $N$, and $t_0 \in \R$.
\end{lemma}
Suppose we are estimating an expression $\gamma(T'^{-1} (t-t_0) u_N )$ which is localized at frequencies $N$ and on times $T' \ll T(N)=N^{-\alpha}$. When we use the weight $(1+L/N^3)^{\frac{1}{4}}$, we can estimate the expressions 
\begin{equation*}
\sum_{L \geq (T')^{-1}} L^{\frac{1}{2}} (1+L/N^3)^{\frac{1}{4}} \| f_{N,L} \|_{L^2_{\tau,\xi,\eta}} \lesssim \| u_N \|_{F_N},
\end{equation*}
provided that $T'\lesssim N^3$.

\smallskip

We recall the embedding $F^s(T) \hookrightarrow C([0,T],H^{s,0}(\T^2))$.
\begin{lemma}[{\cite[Lemma~3.1]{IonescuKenigTataru2008}}]
\label{lem:ShorttimeEmbedding}
Let $s \in \R$, and $T \in (0,1]$. Then we have
\begin{equation*}
\sup_{t \in [0,T]} \| u(t) \|_{H^{s,0}(\T^2)} \lesssim \| u \|_{F^s(T)}.
\end{equation*}
\end{lemma}

\medskip

We assemble the time-localized function spaces by Littlewood-Paley theory:
\begin{equation*}
\begin{split}
\| u \|^2_{F^s(T)} &= \sum_{N \geq 1} N^{2s} \| P_N u \|_{F_N(T)}^2, \\
\| u \|^2_{\mathcal{N}^s(T)} &= \sum_{N \geq 1} N^{2s} \| P_N u \|_{\mathcal{N}_N(T)}^2.
\end{split}
\end{equation*}
We shall see (Section \ref{section:Outline} for details) that $\alpha = 0+$ suffices for a short-time bilinear estimate at negative regularity. For the energy estimate large $\alpha$ is less favorable because we have to divide the time-interval $[0,T]$ into more and more subintervals of length $N^{-\alpha}$ to estimate frequency-localized components of solutions $P_N u$.

\medskip

The ``energy space" is given by
\begin{equation*}
\begin{split}
E^s(T) &= \{ u \in C([0,T] ; L^2(\T^2)) : \|u \|_{E^s(T)} < \infty \} \text{ with } \\
\| u \|^2_{E^s(T)} &= \| P_{\leq 8} u(0) \|_{L^2}^2 + \sum_{N \geq 8} N^{2s} \sup_{t \in [0,T]} \| P_N u (t) \|^2_{L^2}.
\end{split}
\end{equation*}

The energy estimate in short-time Fourier restriction spaces takes the following form: 
\begin{lemma}[Linear~energy~estimate]
\label{lem:LinearEnergyEstimate}
Let $s \in \R$ and $u \in E^s(T) \cap F^s(T)$, $v \in \mathcal{N}^s(T)$ such that $\partial_t u + \partial_x^3 u + \partial_x^{-1} \partial_y^2 u = v$. Then the following estimate holds:
\begin{equation*}
\| u \|_{F^s(T) } \lesssim \| u \|_{E^s(T)} + \| v \|_{\mathcal{N}^s(T)}.
\end{equation*}
\end{lemma}

In the following we see how to trade modulation regularity for powers of the time scale. Define the following variant of $X_N$-spaces for $b \in \R$:
\begin{equation*}
\| f_N \|_{X_N^b} = \sum_{L \geq 1} L^b \| \eta_L(\tau-\omega(\xi,\eta)) f_N(\tau,\xi,\eta) \|_{L^2_{\tau,\xi,\eta}}.
\end{equation*}
The short-time spaces $F_N^b$, $\mathcal{N}_N^b$ are defined like above, but based on $X_N^b$ instead $X_N$. We have the following:
\begin{lemma}[{\cite[Lemma~3.4]{GuoOh2018}}]
\label{lem:ModulationSlack}
Let $T \in (0,1]$, and $b < \frac{1}{2}$. Then it holds:
\begin{equation*}
\| P_N u \|_{F_N^b} \lesssim T^{\frac{1}{2}-b-} \| P_N u \|_{F_N}
\end{equation*}
for any function $u:[-T,T] \times \T^2 \to \C$.
\end{lemma}
The lemma is used to obtain mitigating factors $T^\delta$ for the short-time bilinear and energy estimates to handle large initial data.

\section{Linear Strichartz estimates}
\label{section:StrichartzEstimates}

\subsection{$L^4$-Strichartz estimate for KP-II}

In the following we consider functions $f: \T^2 \to \C$ with
\begin{equation*}
\operatorname{supp}(\hat{f}) \subseteq \{ |\xi| \sim N, \, |\eta| \lesssim N^2 \}.
\end{equation*}
This region matches unit frequencies after carrying out anisotropic scaling $\xi \to \xi/N$, $\eta \to \eta/N^2$, which is a symmetry of the KP-II dispersion relation.

\subsection{Notation and general decoupling result}
We introduce notation to formulate the result from \cite{GuthMaldagueOh2024}.

Let $\phi: \R^2 \to \R$ be a smooth function. We define the manifold
\begin{equation*}
\mathcal{M}_{\phi} = \{ (\xi_1,\xi_2, \phi(\xi_1,\xi_2)) : (\xi_1,\xi_2) \in [0,1]^2 \} \subseteq \R^3.
\end{equation*}
$\mathcal{N}_\delta(\mathcal{M}_\phi)$ denotes the $\delta$-neighborhood of the manifold $\mathcal{M}_{\phi}$. We define $\delta$-flat sets, which trivialize the Fourier extension operator associated to $\mathcal{M}_{\phi}$. For $S \subseteq \R^2$ and $f \in \mathcal{S}(\R^3)$ we let $f_S$ denote the Fourier projection to $S \times \R$. Notably, for technical reaonss a smooth cutoff is preferable; see \cite{GuthMaldagueOh2024}.
\begin{definition}[$\delta$-flat set]
Let $\phi: \R^2 \to \R$ be smooth. We say that $S \subseteq [0,1]^2$ is $(\phi,\delta)$-flat if
\begin{equation*}
\sup_{u,v \in S} \big| \phi(u) - \phi(v) - \nabla \phi(u) \cdot (u-v) \big| \leq \delta.
\end{equation*}
\end{definition}

The following $\ell^2$-decoupling result was recently proved by Guth--Maldague--Oh:
\begin{theorem}[{\cite[Theorem~1.2]{GuthMaldagueOh2024}}]
\label{thm:GeneralL2Decoupling}
Let $\phi: \R^2 \to \R$ be a smooth function, and $0 < \delta < 1 $, and $\varepsilon > 0$. Then there exists a sufficiently large number $A$ depending on $\varepsilon$ and $\phi$ satisfying the following: For any $\delta > 0$ there exists a collection $\mathcal{S}_\delta$ of finitely overlapping sets $S$ such that
\begin{itemize}
\item[(i)] the overlapping is $\mathcal{O}(\log (\delta^{-1}))$ such that for the indicator functions $\chi_S$ the following estimate holds:
\begin{equation*}
\sum_{S \in \mathcal{S}_\delta} \chi_S \leq C_{\varepsilon,\phi} \log(\delta^{-1}),
\end{equation*}
\item[(ii)] any $S \in \mathcal{S}_\delta$ is $(\phi,A \delta)$-flat,
\item[(iii)] the following estimate holds:
\begin{equation*}
\| f \|_{L^4} \leq C_{\varepsilon, \phi} \delta^{-\varepsilon} \big( \sum_{S \in \mathcal{S}_\delta} \| f_S \|_{L^4}^2 \big)^{\frac{1}{2}}
\end{equation*}
for all functions with Fourier support contained in $\mathcal{N}_\delta(\mathcal{M}_\phi)$.
\end{itemize}
\end{theorem}

\subsection{Linear Strichartz estimates for KP-II}

As an application of Theorem \ref{thm:GeneralL2Decoupling}, we prove the following essentially sharp $L^4$-Strichartz estimate for the KP-II equation:
\begin{theorem}
\label{thm:L4Strichartz}
Let $f \in L^2(\T^2)$ with 
\begin{equation*}
\operatorname{supp}(\hat{f}) \subseteq A_{N,N^2} := \{ (\xi,\eta) \in \R^2 : |\xi| \sim N, \quad |\eta| \lesssim N^2 \}.
\end{equation*}
Then the following Strichartz estimate holds:
\begin{equation*}
\| S_{\mathrm{KP}}(t) f \|_{L^4_{t,x,y}([0,1] \times \T^2) } \lesssim_\varepsilon N^{\frac{1}{8}+\varepsilon} \| f \|_{L^2(\T^2)}.
\end{equation*}
\end{theorem}

Having Theorem \ref{thm:GeneralL2Decoupling} at disposal, the key point to obtain favorable Strichartz estimates is to find an upper bound on the size of $(\omega,\delta)$-flat sets. This is the content of the following proposition:

\begin{proposition}[Bound~on~flat~sets]
\label{prop:FlatSets}
Consider the manifold
\begin{equation*}
\mathcal{M}_{\mathrm{KP}} = \{ (\xi,\eta,\tau) \in \R^3 : |\xi| \sim 1, \; |\eta| \lesssim 1, \; \tau = \omega(\xi,\eta) \}.
\end{equation*}
Then $\delta$-flat sets $S \subseteq \{ (\xi,\eta) \in \R^2 : |\xi| \sim 1, \; |\eta| \lesssim 1 \}$ are contained in balls of size comparable to $\delta^{\frac{1}{3}}$.
\end{proposition}
Taking Proposition \ref{prop:FlatSets} for granted, we finish the proof of Theorem \ref{thm:L4Strichartz}:
\begin{proof}[{Proposition~\ref{prop:FlatSets}~implies~Theorem~\ref{thm:L4Strichartz}}]
By $L^2$-boundedness of the propagator it suffices to consider times $t \in [0,1/10]$.
In the first step we normalize to unit frequencies by the anisotropic scaling $x \to N x$, $y \to N^2 y$, $t \to N^3 t$ which corresponds to $\xi \to \xi /N$, $\eta \to \eta/N^2$ and leaves the characteristic surface invariant:
\small
\begin{equation*}
\begin{split}
&\quad \Big\| \sum_{(\xi,\eta) \in A_{N,N^2} \cap \Z^2} e^{i(x \xi + y \eta + t (\xi^3 - \frac{\eta^2}{\xi})} \hat{f}(\xi,\eta) \Big\|_{L^4_{t,x,y}([0,1/10] \times \T^2)} \\
&= N^{-\frac{3}{4}} N^{-\frac{1}{4}} N^{-\frac{2}{4}} \\
&\quad \times \Big\| \sum_{(\xi',\eta') \in A_{1,1} \cap (\Z/N \times \Z/N^2)} e^{i(x' \xi' + y' \eta' + t'( (\xi')^3 - (\eta')^2 / \xi')} \underbrace{\hat{f}(N \xi', N^2 \eta')}_{=: b_{\xi',\eta'}} \Big\|_{L^4_{t',x',y'}([0,N^3/10] \times N \T \times N^2 \T)}.
\end{split}
\end{equation*}
\normalsize
We use spatial periodicity in $x$ and $y$ to obtain an isotropic domain of integration:
\begin{equation*}
\begin{split}
&\quad \Big\| \sum_{(\xi',\eta') \in A_{1,1} \cap (\Z/N \times \Z/N^2)} e^{i(x' \xi' + y' \eta' + t'((\xi')^3 - (\eta')^2 / \xi')} b_{\xi',\eta'} \Big\|_{L^4_{t',x',y'}([0,N^3/10] \times N \T \times N^2 \T)} \\
&\lesssim N^{-\frac{2}{4}} N^{-\frac{1}{4}} \Big\| \sum_{(\xi',\eta') \in A_{1,1} \cap (\Z/N \times \Z/N^2)} e^{i(x' \xi' + y' \eta' + t'((\xi')^3 - (\eta')^2 / \xi'))} b_{\xi',\eta'} \Big\|_{L^4_{t',x',y'}(B_{N^3})}.
\end{split}
\end{equation*}
Above and in the following $B_{M} = \{ x \in \R^3 : |x| \leq M \}$ denotes a ball of radius $M > 0$ in three dimensions. Now we use continuous approximation and approximate the Dirac comb
\begin{equation*}
f_b = \sum_{(\xi',\eta') \in A_{1,1} \cap (\Z /N \times \Z/N^2)} b_{\xi',\eta'} \delta_{\xi',\eta'} \in \mathcal{S}'(\R^2)
\end{equation*}
with a mollified version:
\begin{equation*}
f_{b,\lambda}(\xi,\eta) = \sum_{(\xi',\eta') \in A_{1,1} \cap (\Z/N \times \Z/N^2)} b_{\xi',\eta'} \psi_{\xi',\eta',\lambda}(\xi,\eta)
\end{equation*}
with $\psi_{\xi',\eta',\lambda} \in C^\infty_c(\R^2)$ with $\operatorname{supp}(\psi_{\xi',\eta',\lambda}) \subseteq B_2((\xi',\eta'),1/\lambda)$ such that $\psi_{\xi',\eta',\lambda} \to \delta_{\xi',\eta'} \text{ in } \mathcal{S}'(\R^2)$ as $\lambda \to \infty$.

\smallskip

Let $\mathcal{E}_{\mathrm{KP}}$ denote the Fourier extension operator associated to the KP-II characteristic surface:
\begin{equation*}
\mathcal{E}_{\mathrm{KP}} f = \int_{ \{ |\xi| \sim 1, \; |\eta| \lesssim 1 \} } e^{i(x \xi + y \eta + t \omega(\xi,\eta))} f(\xi,\eta) d\xi d\eta
\end{equation*}
Let $\delta := N^{-3}$. We can choose $\lambda$ large enough such that
\begin{equation*}
\big\| \mathcal{E}_{\mathrm{KP}} f_b \big\|_{L^4_{t',x',y'}(B_{\delta^{-1}})} \lesssim \big\| \mathcal{E}_{\mathrm{KP}} f_{b,\lambda} \big\|_{L^4_{t',x',y'}(B_{\delta^{-1}})}.
\end{equation*}
Let $w_{B_{\delta^{-1}}}$ be a weight which is compactly supported in Fourier space on a ball of size $\delta$, comparable to $1$ on $B_{\delta^{-1}}$, and rapidly decaying off $B_{\delta^{-1}}$. We can dominate
\begin{equation*}
\big\| \mathcal{E}_{\mathrm{KP}} f_b \big\|_{L^4_{t',x',y'}(B_{\delta^{-1}})} \lesssim \big\| \mathcal{E}_{\mathrm{KP}} f_{b,\lambda} \cdot w_{B_{\delta^{-1}}} \big\|_{L^4_{t',x',y'}(\R^3)}.
\end{equation*}
The function $F = \mathcal{E}_{\mathrm{KP}} f_{b,\lambda} \cdot w_{B_{\delta^{-1}}}$ has Fourier support in a $\delta$-neighbourhood of $\mathcal{M}_{\mathrm{KP}}$. We can apply Theorem \ref{thm:GeneralL2Decoupling} to find
\begin{equation*}
\| F \|_{L^4_{t',x',y'}(\R^3)} \lesssim_\varepsilon \delta^{-\varepsilon} \big( \sum_{\theta \in \mathcal{S}_\delta} \| F_\theta \|_{L^4_{t',x',y'}(\R^3)}^2 \big)^{\frac{1}{2}}
\end{equation*}
with $\theta$ denoting rectangles at most of size $\delta^{\frac{1}{3}} \times \delta^{\frac{1}{3}}$, which follows from Proposition \ref{prop:FlatSets}, and $\mathcal{S}_\delta$ being a $\log(\delta^{-1})$-overlapping cover. We reverse the continuous approximation (cf. \cite[Section~2]{Schippa2023Refinements}) by letting $\lambda \to \infty$:
\begin{equation*}
\begin{split}
\| F_\theta \|_{L^4(\R^3)} &= \| \mathcal{E}_{\mathrm{KP}} f_{b,\theta,\lambda} \cdot w_{B_{\delta^{-1}}} \|_{L^4_{t',x',y'}(\R^3)} \\
&\to \big\| \sum_{(\xi',\eta') \in (\Z / N \times \Z / N^2) \cap \theta} e^{i(x' \xi' + y' \eta' + t' \omega(\xi',\eta'))} b_{\xi',\eta'} \big\|_{L^4_{t',x',y'}(w_{B_{N^3}})}.
\end{split}
\end{equation*}
Observe that $\theta$ is a rectangle of size at most $1/N \times 1/N$, which means that  $\# \{ \xi' \in \Z/N : \xi' \in \pi_{\xi'} \theta \}$ is $\mathcal{O}(1)$. By rapid decay and periodicity it suffices to estimate
\begin{equation*}
\big\| \sum_{(\xi',\eta') \in (\Z / N \times \Z/N^2) \cap \theta \cap A_{1,1} } e^{i(x' \xi' + \eta' y' + t' \omega(\xi',\eta'))} b_{\xi',\eta'} \big\|_{L^4_{t',x',y'}(B_{N^3})}.
\end{equation*}
We use periodicity of $x'$ with period $N$ and of $y'$ with period $N^2$ to reverse the enlargement of the domain of integration:
\begin{equation*}
\begin{split}
&\quad \big\| \sum_{(\xi',\eta') \in (\Z /N \times \Z /N^2) \cap \theta } e^{i(x' \xi' + y' \eta' + t' \omega(\xi',\eta'))} b_{\xi',\eta'} \big\|_{L^4_{t',x',y'}(B_{N^3})} \\
&\lesssim N^{\frac{2}{4}} N^{\frac{1}{4}}
\big\| \sum_{(\xi',\eta') \in (\Z /N \times \Z /N^2) \cap \theta } e^{i(x' \xi' + y' \eta' + t' \omega(\xi',\eta'))} b_{\xi',\eta'} \big\|_{L^4_{t',x',y'}([0,N^3] \times N \T \times N^2 \T)}.
\end{split}
\end{equation*}
We can moreover reverse the scaling to find:
\begin{equation*}
\begin{split}
&\quad \big\| \sum_{(\xi',\eta') \in (\Z /N \times \Z /N^2) \cap \theta } e^{i(x' \xi' + y' \eta' + t' \omega(\xi',\eta'))} b_{\xi',\eta'} \big\|_{L^4_{t',x',y'}([0,N^3] \times N \T \times N^2 \T)} \\
&= N^{\frac{3}{4}} N^{\frac{1}{4}} N^{\frac{2}{4}} \big\| \sum_{(\xi,\eta) \in \Z^2 \cap \theta_{N,N^2}} e^{i(x \xi + y \eta + t \omega(\xi,\eta))} \hat{f}(\xi,\eta) \big\|_{L^4_{t,x,y}([0,1] \times \T^2)}.
\end{split}
\end{equation*}
In the above display $\theta_{N,N^2}$ denotes the anisotropic dilation of $\theta$ with dilation factor $N$ into $\xi$ and $N^2$ into $\eta$-direction.
We remark that finally every scaling to unit frequencies has been reversed and so has been the enlargement of the domain of integration by periodicity. All scaling factors cancel and we obtain
\begin{equation}
\label{eq:DecouplingConsequence}
\| S_{\mathrm{KP}}(t) f_{\theta} \|_{L^4_{t,x,y}([0,1] \times \T^2)} \lesssim_\varepsilon N^\varepsilon \big( \sum_{ \theta \in S_\delta} \| S_{\mathrm{KP}}(t) f_{\theta_{N,N^2}} \|_{L_{t,x,y}^4([0,1] \times \T^2)}^2 \big)^{\frac{1}{2}}.
\end{equation}

By Proposition \ref{prop:FlatSets} $\theta_{N,N^2}$ is at most of size $\sim 1 \times N$. So we can find intervals $I_{\theta,\xi}$ and $I_{\theta,\eta}$ of length $\sim 1$ and $N$, respectively, such that
\begin{equation*}
\theta_{N,N^2} \subseteq I_{\theta,\xi} \times I_{\theta,\eta} \subseteq \tilde{\theta}_{N,N^2},
\end{equation*}
where $\tilde{\theta}_{N,N^2}$ denotes a slight enlargement of $\theta_{N,N^2}$.

 By an application of Minkowski's inequality we have
\begin{equation}
\label{eq:EstimateThetaI}
\begin{split}
&\quad \big\| \sum_{(\xi,\eta) \in A_{N,N^2} \cap \theta_{N,N^2}} e^{i( x \xi + y \eta + t (\xi^3 - \eta^2 / \xi))} b_{\xi,\eta} \big\|_{L^4_{t,x,y}([0,1] \times \T^2)} \\
&\lesssim \sum_{\xi \in I_{\theta,\xi}} \big\| \sum_{\eta \in I_{\theta,\eta}} e^{i(y \eta - t \eta^2/ \xi)} b_{\xi,\eta} \big\|_{L^4_{t,y}([0,1] \times \T)}.
\end{split}
\end{equation}
A change of variables, keeping in mind that $|\xi| \sim N$, implies
\begin{equation}
\label{eq:EstimateThetaII}
\begin{split}
\big\| \sum_{\eta \in I_{\theta,\eta}} e^{i(y \eta - t \eta^2 / \xi)} b_{\xi,\eta} \big\|_{L^4_{t,y}([0,1] \times \T)} &\lesssim N^{\frac{1}{4}} \big\| \sum_{\eta \in I_{\theta,\eta}} e^{i(y \eta - t' \eta^2)} b_{\xi,\eta} \big\|_{L^4_{t,y}([0,N^{-1}] \times \T)} \\
&\lesssim N^{\frac{1}{8}} \big\| \sum_{\eta \in I_{\theta,\eta}} e^{i(y \eta - t' \eta^2)} b_{\xi,\eta} \big\|_{L^8_{t}([0,N^{-1}],L^4_y (\T))}.
\end{split}
\end{equation}
This expression can be estimated by short-time (lossless) Strichartz estimates due to Staffilani--Tataru \cite{StaffilaniTataru2002} and Burq--Gérard--Tzvetkov \cite{BurqGerardTzvetkov2004} after applying Galilean invariance to shift the $\eta$-frequencies to an interval of length $N$ centered at the origin:
\begin{equation}
\label{eq:EstimateThetaIII}
\big\| \sum_{\eta \in I_{\theta,\eta}} e^{i(y \eta - t' \eta^2)} b_{\xi,\eta} \big\|_{L_t^8([0,N^{-1}],L^4(\T))} \lesssim \big( \sum_{\eta \in I_{\theta,\eta}} |b_{\xi,\eta}|^2 \big)^{\frac{1}{2}}.
\end{equation}
Taking \eqref{eq:EstimateThetaI} and \eqref{eq:EstimateThetaIII} together gives
\begin{equation*}
\begin{split}
&\quad \big\| \sum_{(\xi,\eta) \in A_{N,N^2} \cap \theta_{N,N^2}} e^{i( x \xi + y \eta + t (\xi^3 - \eta^2 / \xi))} b_{\xi,\eta} \big\|_{L^4_{t,y}([0,1] \times \T)} \\
&\lesssim N^{\frac{1}{8}} \big( \sum_{(\xi,\eta) \in \theta_{N,N^2}} |b_{\xi,\eta} |^2 \big)^{\frac{1}{2}}.
\end{split}
\end{equation*}
Plugging this estimate into \eqref{eq:DecouplingConsequence} and using the overlap being bounded with $\log(N)$ we find
\begin{equation*}
\| S_{\mathrm{KP}}(t) f \|_{L^4_{t,x,y}([0,1] \times \T^2)} \lesssim_\varepsilon N^{\frac{1}{8}+2 \varepsilon} \| f \|_{L^2(\T^2)}.
\end{equation*}
The proof is complete.
\end{proof}

It remains to show Proposition \ref{prop:FlatSets}:
\begin{proof}[Proof~of~Proposition~\ref{prop:FlatSets}]
We follow the argument of \cite[Remark~4.1]{GuthMaldagueOh2024}. It suffices to check that for $\gamma(t)$ parametrizing a line with unit speed we have
\begin{equation}
\label{eq:LineExpansion}
\omega(\gamma(t)) = a_0 + a_1 t + a_2 t^2 + a_3 t^3 + E(t)
\end{equation}
with $|a_2| + |a_3| \gtrsim 1$ and $|E(t)| \leq c t^4$. Then we can infer that for $u,v \in R$, $|u-v| \ll 1$, $R$ being a rectangle:
\begin{equation*}
|\omega(u) - \omega(v) - \nabla \omega(u) \cdot (u-v) | \gtrsim |u-v|^3.
\end{equation*}
Letting $\gamma(t) = u + \frac{t (v-u)}{|v-u|}$ a Taylor expansion of $\omega(\gamma(t))$ yields the claim. Write $\gamma(t) = (\xi_0,\eta_0) + t (\xi',\eta')$, and recall that
\begin{equation*}
|\xi_0| \sim 1, \quad |\eta_0| \lesssim 1, \quad |(\xi',\eta')| = 1.
\end{equation*}

To simplify the evaluation of
\begin{equation*}
\omega((\xi_0,\eta_0) + t (\xi',\eta')) = (\xi_0 + t \xi')^3 - \frac{(\eta_0 + t \eta')^2}{\xi_0 + t \xi'},
\end{equation*}
we use the Galilean invariance:
\begin{equation}
\label{eq:GalileanInvarianceKPII}
\omega(\xi,\eta + A \xi) = \omega(\xi, \eta) - 2 A \eta + A^2 \xi.
\end{equation}

First, we deal with the simpler case $|\xi'| \ll |\eta'|$. We find
\begin{equation*}
\begin{split}
(\xi_0 + t \xi')^3 - \frac{(\eta_0 + t \eta')^2}{\xi_0 + t \xi'} &= \xi_0^3 + \mathcal{O}(\xi') - \frac{\eta_0^2 + 2 t \eta_0 \eta' + t^2 (\eta')^2}{\xi_0} (1 + \mathcal{O}(\xi')) \\
&= \xi_0^3 - \frac{\eta_0^2}{\xi_0} + \frac{2t \eta_0}{\xi_0} - \frac{t^2 (\eta')^2}{\xi_0} + \mathcal{O}(\xi').
\end{split}
\end{equation*}
Consequently, the coefficient $a_2$ in \eqref{eq:LineExpansion} is given by $a_2 = - \frac{(\eta')^2}{\xi_0} + \mathcal{O}(\xi')$ and hence $|a_2| \gtrsim 1$. 

\smallskip

Next, we consider the case $|\xi'| \gtrsim 1$. In this case we apply \eqref{eq:GalileanInvarianceKPII} with $\xi = \xi_0 + t \xi'$ and $\eta = \eta_0 + t \eta'$, $A = - \eta' / \xi'$:
Hence,
\begin{equation*}
\omega(\xi,\eta) = \omega(\xi,\eta + A \xi) + 2 A \eta - A^2 \xi = \omega(\xi,\eta - \frac{\eta'}{\xi'} \xi) + \text{Lin}(t).
\end{equation*}
The linear term $\text{Lin}(t) = 2 A \eta - A^2 \xi$ can be disregarded for the computation, it remains to analyze
\begin{equation*}
\begin{split}
\omega(\xi,\eta - \frac{\eta'}{\xi'} \xi) &= \omega(\xi, \eta_0 - \frac{\eta'}{\xi'} \xi_0) \\
&= (\xi_0 + t \xi')^3 - \frac{\bar{\eta}^2}{\xi_0 + t \xi'} \\
&= (\xi_0^3 + 3t \xi_0^2 \xi' + 3 t^2 \xi_0 (\xi')^2 + t^3 (\xi')^3) \\
&\quad - \frac{\bar{\eta}^2}{\xi_0}(1 - t \frac{\xi'}{\xi_0} + t^2 \frac{\xi'^2}{\xi_0^2} - t^3 \frac{(\xi')^3}{\xi_0^3} + \mathcal{O}(t^4)).
\end{split}
\end{equation*}
Above we have set $\bar{\eta} = \eta_0 - \frac{\eta'}{\xi'} \xi_0$. Note that uniform bounds on the remainder term $\mathcal{O}(t^4)$ follow from $|\xi_0| \sim 1$ and boundedness of the remaining parameters.
We obtain
\begin{equation*}
\omega(\xi,\eta - \frac{\eta'}{\xi'} \xi) = \text{Lin}(\xi',\eta') + t^2 (\xi')^2 \xi_0 \big( 3 - \frac{\bar{\eta}^2}{\xi_0^4} \big) + t^3 (\xi')^3 (1 + \frac{\bar{\eta}^2}{\xi_0^4} ) + \mathcal{O}(t^4).
\end{equation*}
$a_2$ can possibly vanish, but we have the lower bound $|a_3| \gtrsim 1$. This finishes the proof.
\end{proof}

\begin{remark}[Shape~of~$\delta$-flat~sets]\label{rem:gal}
The proof gives a more precise description of the shape of the $\delta$-flat rectangles $R$: Fixing $(\xi_0,\eta_0) \in R$, we have that $|\eta-\eta_0| \lesssim \delta^{\frac{1}{2}}$ for any $(\xi_0,\eta) \in R$. This was pointed out in the first part of the proof.
Secondly, there is possibly a direction $(\xi',\eta')$ with $|\xi'| \gtrsim 1$ in which $R$ is extended to length $\delta^{\frac{1}{3}}$ as argued in the second part of the proof.
\begin{figure}[h]
        \begin{tikzpicture}
            \draw (-3,0)--(0,2);
            \draw (-3,0)--(-3,1);
            \draw (-3,1)--(0,3);
            \draw (0,2)--(0,3);
            \node[left] at(-3,0.5){$\delta^{\frac12}$};
            \node[below] at(-1.2,1.2){$\delta^{\frac13}$};
             \draw (3,0)--(6.6,0);
            \draw (3,0)--(3,1);
            \draw (3,1)--(6.6,1);
            \draw (6.6,0)--(6.6,1);
            \node[left] at(3,0.5){$\delta^{\frac12}$};
            \node[below] at(4.8,0){$\delta^{\frac13}$};
           \draw[->](0.2,1.3)to[bend left](2.8,1);\node at (1.7,1.9) {$\eta\mapsto \eta-A \xi $};
        \end{tikzpicture} 
        \caption{Galilean transformation to shift the long direction into the $\xi$-direction.}
    \end{figure}
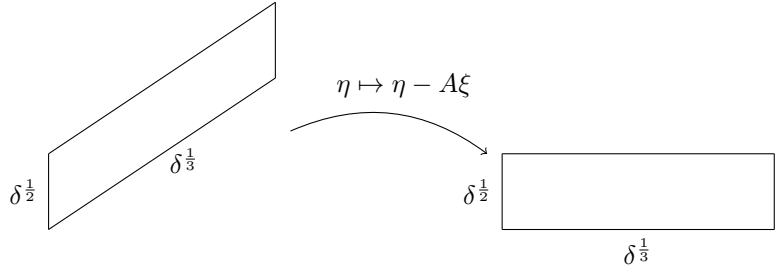
\end{remark}

\subsection{Sharpness of the estimate}\label{subsec:sharp}
We conclude the section with an example showing sharpness up to endpoints:

Define $f: \T^2 \to \C$ by
\begin{equation*}
\hat{f}(\xi,\eta) = \begin{cases}
 1 ,\quad \xi = N, \; \eta = 0,\ldots,N^{\frac{1}{2}}, \\
 0, \quad \text{else}.
 \end{cases}
\end{equation*}
We have
\begin{equation*}
\begin{split}
\| S_{\mathrm{KP}}(t) f \|_{L^4_{t,x,y}([0,1] \times \T^2)} &=  \big\| \sum_{\xi = N, \eta = 0, \ldots, N^{\frac{1}{2}}} e^{i(x \xi + y \eta + t \omega(\xi,\eta))} \big\|_{L^4_{t,x,y}([0,1] \times \T^2)} \\
&= \big\| \sum_{\eta=0}^{N^{\frac{1}{2}}} e^{i (y \eta - t \frac{\eta^2}{N} )} \big\|_{L^4_{t,y}([0,1] \times \T)} \\
&= N^{\frac{1}{4}} \big\| \sum_{\eta = 0}^{N^{\frac{1}{2}}} e^{i(y \eta - t \eta^2)} \big\|_{L^4_{t,y}([0,N^{-1}] \times \T)}.
\end{split}
\end{equation*}
On the time-scale $N^{-1}$ are no oscillations of $e^{it \Delta_y} \bar{f}$. Therefore, carrying out the integration over $t \in [0,N^{-1}]$ we find
\begin{equation*}
N^{\frac{1}{4}} \big\| \sum_{\eta = 0}^{N^{\frac{1}{2}}} e^{i(y \eta - t \eta^2)} \big\|_{L^4_{t,y}([0,N^{-1}] \times \T)} \sim \big\| \sum_{\eta = 0}^{N^{\frac{1}{2}}} e^{i y \eta } \big\|_{L^4_y(\T)} \sim N^{\frac{3}{8}} \sim N^{\frac{1}{8}} \| \sum_{\eta=0}^{N^{\frac{1}{2}}} e^{iy \eta} \|_{L^2}.
\end{equation*}

\subsection{Improved estimate on frequency-dependent time intervals}

We record the following refinement on frequency-dependent times:
\begin{theorem}
\label{thm:ImprovedL4Strichartz}
Let $\alpha \in (0,1]$, and $f \in L^2(\T^2)$ with 
\begin{equation*}
\operatorname{supp}(\hat{f}) \subseteq A_{N,N^2} := \{ (\xi,\eta) \in \R^2 : |\xi| \sim N, \quad |\eta| \lesssim N^2 \}.
\end{equation*}
\begin{itemize}
\item[(i)] The following Strichartz estimate holds:
\begin{equation}
\label{eq:ShorttimeStrichartzI}
\| S_{\mathrm{KP}}(t) f \|_{L^4_{t,x,y}([0,N^{-\alpha}] \times \T^2) } \lesssim_\varepsilon N^{\frac{1}{8}+\frac{\alpha}{12}-\frac{\alpha}{8}+\varepsilon} \| f \|_{L^2(\T^2)}.
\end{equation}
\item[(ii)]
If it holds $|\eta| \ll N^2$ for any $(\xi,\eta) \in \operatorname{supp}(\hat{f})$, then we have
\begin{equation}
\label{eq:ShorttimeStrichartzII}
\| S_{\mathrm{KP}}(t) f \|_{L^4_{t,x,y}([0,N^{-\alpha}] \times \T^2) } \lesssim_\varepsilon N^{\frac{1}{8}-\frac{\alpha}{8}+\varepsilon} \| f \|_{L^2(\T^2)}.
\end{equation}
\end{itemize}
\end{theorem}
\begin{proof}
We follow the arguments from the proof of Theorem \ref{thm:L4Strichartz}.
The key ingredient remains $\ell^2$-decoupling, which after using the anisotropic scaling $t \to N^3 t$, $x \to N x$, $y \to N^2 y$ yields the following inequality:
\begin{equation*}
\| S_{\mathrm{KP}}(t) f \|_{L^4_{t,x,y}([0,N^{-\alpha}] \times \T^2)} \lesssim_\varepsilon N^\varepsilon \big( \sum_{\theta \in \mathcal{S}_\delta} \| S_{\mathrm{KP}}(t) f_{\tilde{\theta}} \|_{L^4_{t,x,y}([0,N^{-\alpha}] \times \T^2)}^2 \big)^{\frac{1}{2}}.
\end{equation*}
Presently, we have $\delta^{-1} = N^{3-\alpha}$, and after Galilean transform (see Remark \ref{rem:gal}) we obtain anisotropically inflated $\tilde{\theta}$ of size at most $N \delta^{\frac{1}{3}} \times N^2 \delta^{\frac{1}{2}} = N^{\frac{\alpha}{3}} \times N^{\frac{1}{2}+\frac{\alpha}{2}}$. More precisely, we firstly rescale to a parallelogram of size $N \delta^{\frac{1}{3}} \times N^2 \delta^{\frac{1}{2}} \sim N^{\frac{\alpha}{3}} \times N^{\frac{1}{2}+\frac{\alpha}{2}}$. Now we pick the integer centers $\eta_1$, $\eta_2$ of the (long) $N^{\frac{1}{2}+\frac{\alpha}{2}}$-intervals at the $\xi$-endpoints $\xi_1$, $\xi_2$. We have $|\xi_2 - \xi_1| \sim N^{\frac{\alpha}{3}}$. We choose $A \in \R$ such that
\begin{equation*}
\eta_2 - A \xi_2 = \eta_2' = \eta_1' = \eta_1 - A \xi_1.
\end{equation*}
This Galilean transform maps the $N^{\frac{\alpha}{3}} \times N^{\frac{1+\alpha}{2}}$-parallelogram to a rectangle of the same size. But this transform does not necessarily leave the frequency lattice $\Z^2$ invariant. Still, we can choose an integer approximation $A' = [A]$ and see that after transform $\eta \to \eta - A' \xi$ the parallelogram can still be contained in a rectangle of size $N^{\frac{\alpha}{3}} \times N^{\frac{1+\alpha}{2}}$. Indeed, we have for $\eta_i'' = \eta_i - A' \xi_i$ the estimate $|\eta_2'' - \eta_1'' | \lesssim N^{\frac{\alpha}{3}} \lesssim N^{\frac{1+\alpha}{2}}$. In the following we suppose that $\tilde{\theta}$ is a rectangle of size $N^{\frac{\alpha}{3}} \times N^{\frac{1+\alpha}{2}}$.

\smallskip

 Let $(\xi_0,\eta_0)$ denote the center of the rectangle and write for $\tilde{\theta} \ni (\xi,\eta) = (\xi_0,\eta_0) + (\xi',\eta')$. We can expand
\begin{equation*}
\begin{split}
\xi^3 &= (\xi_0+\xi')^3 = \xi_0^3 + 3 \xi_0^2 \xi' + 3 \xi_0 (\xi')^2 + (\xi')^3 = \text{Lin}(\xi') + 3 \xi_0 (\xi')^2 + \mathcal{O}(N^\alpha), \\
\frac{(\eta_0 + \eta')^2}{\xi_0 + \xi'} &= \frac{\eta_0^2 + 2 \eta_0 \eta' + (\eta')^2}{\xi_0} \big( 1 - \frac{\xi'}{\xi_0} + \big( \frac{\xi'}{\xi_0} \big)^2 - \big( \frac{\xi'}{\xi_0} \big)^3 + \mathcal{O}( \big( \frac{\xi'}{\xi_0} \big)^4 ) \\
&= \text{Lin}(\xi',\eta') + \frac{\eta_0^2}{\xi_0} \frac{(\xi')^2}{\xi_0^2} - \frac{\eta_0^2}{\xi_0^4} (\xi')^3 + \mathcal{O}(N^\alpha) \\
&\quad - \frac{2 \eta_0 \eta' \xi'}{\xi_0^2} + \frac{2 \eta_0 \eta'}{\xi_0} \big( \frac{\xi'}{\xi_0} \big)^2 + \mathcal{O}(N^\alpha).
\end{split}
\end{equation*}
The linear terms and $\mathcal{O}(N^\alpha)$ terms can be discarded as they do not contribute to the oscillations by frequency-dependent time localization.
We remain with the pruned phase function
\begin{equation*}
\bar{\omega}(\xi',\eta')= 3 \xi_0 (\xi')^2 - \frac{\eta_0^2}{\xi_0^3} (\xi')^2 + \frac{2 \eta_0 \eta' \xi'}{\xi_0^2}
\end{equation*}
with $(\xi',\eta') \in [-N^{\alpha/3},N^{\alpha/3}] \times [-N^{\frac{1}{2}+\frac{\alpha}{2}},N^{\frac{1}{2}+\frac{\alpha}{2}}] = R'$.
We summarize
\begin{equation*}
\| S_{\mathrm{KP}}(t) f_{\tilde{\theta}} \|_{L^4_{t,x,y}([0,N^{-\alpha}] \times \T^2)} \lesssim \| \sum_{(\xi',\eta') \in R'} e^{i(x \xi' + y \eta' + t \bar{\omega}(\xi',\eta'))} a_{\xi',\eta'} \|_{L^4_{t,x,y}([0,N^{-\alpha}] \times \T^2)}
\end{equation*}
with $\| a_{\xi',\eta'} \|_{\ell^2} = \| f_{\tilde{\theta}} \|_{L^2}$.

Let $I_\eta = [-N^{\frac{1}{2}+\frac{\alpha}{2}},N^{\frac{1}{2}+\frac{\alpha}{2}}]$, $I_\xi = [-N^{\alpha/3},N^{\alpha/3}]$. By applying Bernstein's inequality in the $\eta$-frequency estimating the $L^4_y$-norm in terms of the $L^2_y$-norm, which incurs a factor $N^{\frac{1}{8}+\frac{\alpha}{8}}$, and Minkowski's inequality we obtain
\begin{equation}
\label{eq:AuxStrichartz0}
\begin{split}
&\quad \| \sum_{(\xi',\eta') \in R'} e^{i(x \xi' + y \eta' + t \bar{\omega}(\xi',\eta'))} a_{\xi',\eta'} \|_{L^4_{t,x,y}([0,N^{-\alpha}] \times \T^2)} \\
&\lesssim N^{\frac{1}{8}+\frac{\alpha}{8}} \big( \sum_{\eta' \in I_\eta}  \big\| \sum_{\xi' \in I_\xi} e^{i (x \xi' + y \eta' + \xi_0 t ( ( 3 - \frac{\eta_0^2}{\xi_0^4} ) (\xi')^2 + \frac{2 \eta_0 \eta' \xi'}{\xi_0^3} ) )} a_{\xi',\eta'} \big\|^2_{L_{t,x}^4([0,N^{-\alpha}] \times \T)} \big)^{\frac{1}{2}}.
\end{split}
\end{equation}
By a change of variables $t' = t \xi_0$, keeping in mind $|\xi_0| \sim N$, and Galilean invariance, we find
\begin{equation}
\label{eq:AuxStrichartzI}
\begin{split}
&\quad \big\| \sum_{\xi' \in I_\xi} e^{i (x \xi' + y \eta' + \xi_0 t ( ( 3 - \frac{\eta_0^2}{\xi_0^4} ) (\xi')^2 + \frac{2 \eta_0 \eta' \xi'}{\xi_0^3} ) )} a_{\xi',\eta'} \big\|_{L_{t,x}^4([0,N^{-\alpha}] \times \T)} \\
 &\lesssim N^{-\frac{1}{4}} \big\| \sum_{\xi' \in I_\xi} e^{i (x \xi'  + t' ( 3 - \frac{\eta_0^2}{\xi_0^4} ) (\xi')^2 )} a_{\xi',\eta'} \big\|_{L_{t,x}^4([0,N^{1-\alpha}] \times \T)}.
 \end{split}
\end{equation}
If $|\eta_0| \ll N^2$, we can use the $L^4_{t,x,y}$-Strichartz estimate for the one-dimensional Schr\"odinger equation on the unit time scale. The long-time $N^{1-\alpha}$ incurs a factor $N^{\frac{1-\alpha}{4}}$:
\begin{equation}
\label{eq:AuxStrichartzII}
\big\| \sum_{\xi' \in I_{\xi}} e^{i (x \xi' + t' (3-\eta_0^2 / \xi_0^4)(\xi')^2)} a_{\xi',\eta'} \big\|_{L^4_{t,x}([0,N^{1-\alpha}] \times \T)} \lesssim N^{\frac{1-\alpha}{4}} \| a_{\xi',\eta'} \|_{\ell^2_{\xi'}}.
\end{equation}
By \eqref{eq:AuxStrichartzI} and \eqref{eq:AuxStrichartzII} we can evaluate  \eqref{eq:AuxStrichartz0} as
\begin{equation*}
\begin{split}
&\quad \sum_{\eta' \in I_{\eta}} \big\| \sum_{\xi' \in I_\xi} e^{i (x \xi' + y \eta' + \xi_0 t ( ( 3 - \frac{\eta_0^2}{\xi_0^4} ) (\xi')^2 + \frac{2 \eta_0 \eta' \xi'}{\xi_0^3} ) )} \big\|_{L_{t,x,y}^4([0,N^{-\alpha}] \times \T^2)} \\
&\lesssim N^{\frac{1}{8}- \frac{\alpha}{8}+\varepsilon} \| a_{\xi',\eta'} \|_{\ell^2_{\xi',\eta'}}.
\end{split}
\end{equation*}

In case $|\eta_0| \sim N^2$, we simply apply H\"older's and Bernstein's inequalities to estimate \eqref{eq:AuxStrichartz0}:
\begin{equation*}
\begin{split}
&\quad \| \sum_{(\xi',\eta') \in R'} e^{i(x \xi' + y \eta' + t \bar{\omega}(\xi',\eta'))} a_{\xi',\eta'} \|_{L^4_{t,x,y}([0,N^{-\alpha}] \times \T^2)} \\
&\lesssim N^{-\frac{\alpha}{4}} N^{\frac{1}{8} + \frac{\alpha}{8} + \frac{\alpha}{12} + \varepsilon} \| a_{\xi',\eta'} \|_{\ell^2_{\xi',\eta'}}.
\end{split}
\end{equation*}
The additional power $N^{\frac{\alpha}{12}}$ comes from the possible lack of dispersion in the $\xi$-direction.
%
\end{proof}

We remark that the example for $\alpha=0$ can be easily modified to show sharpness of the second estimate \eqref{eq:ShorttimeStrichartzII}. Let $N \in 2^{\N_0}$, $|\xi_0| \sim N$, and $f: \T^2 \to \C$ with
\begin{equation*}
\hat{f}(\xi,\eta) = 
\begin{cases}
1, \quad 0 \leq \eta \leq N^{\frac{1}{2}+\frac{\alpha}{2}}, \quad \xi = \xi_0, \\
0, \quad \text{else}.
\end{cases}
\end{equation*}
The lack of oscillations and the sharpness of Bernstein's inequality $\| f \|_{L^4(\T^2)} \lesssim N^{\frac{1}{8}+\frac{\alpha}{8}} \| f \|_{L^2(\T)}$ implies the estimate
\begin{equation*}
\big\| \sum_{0 \leq \eta \leq N^{\frac{1}{2}+\frac{\alpha}{2}}} e^{i(x \xi_0 + y \eta + t \omega(\xi_0,\eta))} \big\|_{L^4_{t,x,y}([0,N^{-\alpha}] \times \T^2)} \gtrsim N^{-\frac{\alpha}{4}} N^{\frac{1}{8}+\frac{\alpha}{8}} \| f \|_{L^2(\T^2)}.
\end{equation*}

\subsection{Strichartz estimates for large $\eta$-frequencies}

We shall moreover need Strichartz estimates for
\begin{equation*}
\operatorname{supp}(\hat{f}) \subseteq \{ (\xi,\eta,\tau) : |\xi| \sim N, \quad |\eta| \in [k N^2, (k+1)N^2 ] \}
\end{equation*}
with $k \in \N$, $k \gg 1$. To normalize $\eta$-frequencies we consider the shifted phase function
\begin{equation*}
\omega_0(\xi,\eta ) = \xi^3 - \frac{(\eta + \bar{\eta_0})^2}{\xi} \text{ such that } |\eta| \lesssim N^2.
\end{equation*}
By anisotropic scaling $\xi \to \xi/N$, $\eta \to \eta / N^2$ we find a phase function
\begin{equation*}
\omega_1(\xi,\eta) = \xi^3 - \frac{(\eta + \eta_0)^2}{\xi} \text{ with } |\xi| \sim 1, \quad |\eta| \lesssim 1, \quad \eta_0 = \frac{\bar{\eta_0}}{N^2} \in [k-2,k+2].
\end{equation*}
However, the derivatives in $\xi$ are not uniformly bounded, for which reason we decompose the $\xi$-support into intervals of length $\sim \eta_0^{-1}$. We aim to apply decoupling to the phase function
\begin{equation}
\label{eq:ShiftedPhaseFunction}
\bar{\omega}(\xi,\eta) = (\xi_0 + \xi / \eta_0)^3 - \frac{(\eta + \eta_0)^2}{(\xi_0 + \xi / \eta_0)}.
\end{equation}
Above 
\begin{equation}
\label{eq:ParameterRanges}
|\xi_0| \sim 1, \; |\eta_0| \gg 1, \; |\xi| \lesssim 1 \text{ and } |\eta| \lesssim 1.
\end{equation}

\begin{proposition}
\label{prop:FlatSetsShifted}
Let $\xi_0,\eta_0, \xi, \eta$ be like in \eqref{eq:ParameterRanges}. Then the $\delta$-flat sets of $\bar{\omega}$ defined in \eqref{eq:ShiftedPhaseFunction} are contained in rectangles of size at most $\delta^{\frac{1}{2}}$ in $\eta$-direction and $\eta_0 \delta^{\frac{1}{3}}$ in a direction $(\xi',\eta')$ with $|\xi'| \gtrsim 1$.
\end{proposition}
\begin{proof}
To this end, we parametrize a line segment with unit speed by
\begin{equation*}
\gamma(t) = (\xi_1,\eta_1) + t (\xi',\eta') \text{ with } \| (\xi',\eta') \| = 1.
\end{equation*}
In the simple case $|\eta'| \gg |\xi'|$ we see like in the proof of Proposition \ref{prop:FlatSets} that
\begin{equation*}
\begin{split}
\bar{\omega}(\gamma(t)) &= \text{Lin}(t) + t^2 \mathcal{O}(\frac{1}{\eta_0^2}) + t^3 \mathcal{O}(1/\eta_0^3) - \frac{(\eta_1 + t \eta' + \eta_0)^2}{\xi_0 ( 1 + (\xi_1 + t \xi')/ \eta_0)} \\
&= \text{Lin}(t) + c t^2 + \mathcal{O}(t^3) \text{ with } |c| \gtrsim 1. 
\end{split}
\end{equation*}
This points out that for $|\xi'| \ll |\eta'|$ the $\delta$-flat sets are of length $\mathcal{O}(\delta^{\frac{1}{2}})$ into the direction $(\xi',\eta')$.

Next, suppose that $|\xi'| \gtrsim 1$. To simplify the expression $\eta^2 / \xi$, we use the Galilean invariance $\eta \to \eta - A \xi$ with $A = \frac{\eta' \eta_0}{\xi'}$ and 
\begin{equation*}
\eta = \eta_1 + t \eta' + \eta_0, \quad \xi = \xi_0 \big( 1 + \frac{\xi_1 + t\xi'}{\xi_0 \eta_0} \big)
\end{equation*}
such that
\begin{equation*}
\eta - A \xi = \eta_1 + \eta_0 - \frac{\eta'}{\xi'} \eta_0 (\xi_0 + \frac{\xi_1}{\xi_0 \eta_0} ) = \eta_1 + \eta_0 - \frac{\eta'}{\xi'} \eta_0 \xi_0 - \frac{\eta' \xi_1}{\xi' \xi_0} = \bar{\eta}.
\end{equation*}
We obtain
\begin{equation*}
\bar{\omega}(\xi_1 + t\xi', \eta_1 + t \eta') = (\xi_0 + (\xi_1 + t \xi')/\eta_0)^3 - \frac{\bar{\eta}^2}{(\xi_0 + (\xi_1 + t \xi')/\eta_0)} + \text{Lin}(t).
\end{equation*}
We shall prove that
\begin{equation*}
| \bar{\omega}(\gamma(0))''| \gtrsim \eta_0^{-2} \text{ or } |\bar{\omega}(\gamma(0))'''| \gtrsim \eta_0^{-3}.
\end{equation*}
This can be recast as
\begin{equation}
\label{eq:LowerBoundDerivatives}
\big( \frac{\eta_0}{\xi'} \big)^2 | \bar{\omega}(\gamma(0))''| + \big( \frac{\eta_0}{\xi'} \big)^3 |\bar{\omega}(\gamma(0))'''| \gtrsim 1.
\end{equation}
We find
\begin{equation*}
\begin{split}
\bar{\omega}(\gamma(t))'' &= 6 \big( \frac{\xi'}{\xi_0} \big)^2 (\xi_0 + (\xi_1 + t \xi')/\eta_0) - \frac{2 \bar{\eta}^2 (\xi'/\eta_0)^2}{(\xi_0 + (\xi_1 + t \xi')/\eta_0)^3}, \\
\bar{\omega}(\gamma(t))''' &= 6 \big( \frac{\xi'}{\eta_0} \big)^3 + \frac{6 \bar{\eta}^2 (\xi'/\eta_0)^3}{(\xi_0 + (\xi_1 + t \xi')/\eta_0)^4}.
\end{split}
\end{equation*}
Secondly,
\begin{equation*}
\begin{split}
&\quad \big( \frac{\eta_0}{\xi'} \big)^2 |\bar{\omega}(\gamma(0))''| + \big( \frac{\eta_0}{\xi'} \big)^3 |\bar{\omega}(\gamma(0))'''| \\
 &= |6 (\xi_0 + \xi_1 / \eta_0) - \frac{2 \bar{\eta}^2}{(\xi_0 + \xi_1 / \eta_0)^3} \big| + \big| 6 + \frac{6 \bar{\eta}^2}{(\xi_0 + \xi_1 / \eta_0)^4} \big|.
\end{split}
\end{equation*}
Let $\bar{\xi} = \xi_0 + \xi_1 / \eta_0$ and observe that $|\bar{\xi}| \sim 1$. Equivalent to \eqref{eq:LowerBoundDerivatives} is
\begin{equation*}
|6 \bar{\xi}^4 - 2 \bar{\eta}^2 | + | 6 \bar{\xi}^4 + 6 \bar{\eta}^2 | \gtrsim 1,
\end{equation*}
which follows from $|6 \bar{\xi}^4 + 6 \bar{\eta}^2 | \geq 6 \bar{\xi}^4 \gtrsim 1$.
Finally, bounds for higher derivatives $|\bar{\omega}(\gamma(t))^{(k)}| \lesssim 1$ follow from uniform boundedness of higher derivatives of $\bar{\omega}$. The proof is complete.
\end{proof}

\begin{proposition}
\label{prop:L4StrichartzShifted}
Let $f: \T^2 \to \C$ with 
\begin{equation*}
\operatorname{supp}(\hat{f}) \subseteq \{ (\xi,\eta) \in \R^2 : |\xi| \sim N, \; \xi \in I, \; |I| \sim N / k, \; |\eta| \in [k N^2, (k+1) N^2] \}
\end{equation*}
with $1 \ll k \lesssim N$. Then the following estimate holds:
\begin{equation*}
\| S_{\mathrm{KP}}(t) f \|_{L^4_{t,x,y}([0,1] \times \T^2)} \lesssim_\varepsilon N^{\frac{1}{8}+\varepsilon} k^{\frac{1}{12}} \| f \|_{L^2(\T^2)}.
\end{equation*}
\end{proposition}
\begin{proof}
We carry out the same reductions like in the proof of Theorem \ref{thm:ImprovedL4Strichartz}. We use the scaling $t \to N^3 t$, $x \to N x$, $y \to N^2 y$ and periodicity in $x$ and $y$ to find
\small
\begin{equation*}
\begin{split}
&\quad \| S_{\mathrm{KP}}(t) f \|^4_{L^4_{t,x,y}([0,1] \times \T^2)} \\
 &= \big\| \sum_{(\xi,\eta) \in A_{N,N^2} \cap \Z^2} e^{i(x \xi + y \eta + t \omega_0(\xi,\eta))} \hat{f}(\xi,\eta) \big\|^4_{L^4_{t,x,y}([0,1] \times \T^2)} \\
&= N^{-6} \big\| \sum_{(\xi',\eta') \in A^{I'}_{1,1} \cap (\Z /N \times \Z /N^2)} e^{i(x' \xi'+ y' \eta' + t((\xi')^3 - \frac{(\eta' + \eta_0)^2}{\xi'})} a_{\xi',\eta'} \big\|^4_{L^4_{t',x',y'}([0,N^3] \times N \T \times N^2 \T)} \\
&= N^{-6} N^{-2} N^{-\frac{3}{2}} \big\| \sum_{(\xi',\eta') \in A^{I'}_{1,1} \cap (\Z / N \times \Z / N^2)} e^{i(x' \xi' + y' \eta' + t' ((\xi')^3 - \frac{(\eta' + \eta_0)^2}{\xi'})} a_{\xi',\eta'} \big\|^4_{L^4_{t',x',y'}(B_{N^3})}.
\end{split}
\end{equation*}
\normalsize
Above $\eta_0 \in [k-2,k+2]$ and by the Fourier support condition on $f$ we have $\xi' \in I'$ with $|I'| \sim 1/k \sim 1/\eta_0$.

Now we change to continuous approximation
\begin{equation*}
\begin{split}
&\quad \big\| \sum_{(\xi',\eta') \in A^{I'}_{1,1} \cap (\Z / N \times \Z / N^2)} e^{i(x' \xi' + y' \eta' + t' ((\xi')^3 - \frac{(\eta' + \eta_0)^2}{\xi'})} a_{\xi',\eta'} \big\|^4_{L^4_{t',x',y'}(B_{N^3})} \\
 &\lesssim \big\| \int_{I' \times [-1,1]} e^{i(x' \xi' + y' \eta' + t'((\xi')^3 - \frac{(\eta'+\eta_0)^2}{\xi'} ))} f(\xi',\eta') d\xi' d\eta' \big\|_{L^4_{t',x',y'}(B_{N^3})}^4.
 \end{split}
\end{equation*}
We normalize the $\xi'$-support by $\xi' = \xi_0 + \xi''/\eta_0$ and correspondingly let $x'= \eta_0 x''$. We obtain
\begin{equation*}
\begin{split}
&\big\| \int_{I' \times [-1,1]} e^{i(x' \xi' + y' \eta' + t'((\xi')^3 - \frac{(\eta'+\eta_0)^2}{\xi'} ))} f(\xi',\eta') d\xi' d\eta' \big\|_{L^4_{t',x',y'}(B_{N^3})}^4 \\
&\lesssim \eta_0 \big\| \int_{(\xi'',\eta') \in [-1,1]^2 } e^{i(x'' \xi'' + y' \eta' + t' \bar{\omega}(\xi'',\eta'))} f''(\xi'',\eta') d\xi'' d\eta' \big\|^4_{L^4_{t',x'',y'}(\mathcal{E}_{N^3 \times N^3 / \eta_0 \times N^3})}.
\end{split}
\end{equation*}
We cover the space-time ellipsoid 
\begin{equation*}
\mathcal{E}_{N^3 \times N^3 / \eta_0 \times N^3} = \{ (t,x,y) \in \R^3 : \, |t| \leq N^3, \; |x| \leq N^3 / \eta_0, \; |y| \leq N^3 \}
\end{equation*}
 with balls $(B_{\delta^{-1}})$ of size $\delta^{-1} = N^3 / \eta_0$ and seek to apply decoupling to the oscillatory integral
\begin{equation*}
\begin{split}
&\quad \big\| \int_{(\xi'',\eta') \in (-1,1)^2} e^{i(x'' \xi'' + y' \eta' + t' \bar{\omega}(\xi'',\eta'))} f''(\xi'',\eta') d\xi'' d\eta' \big\|_{L^4_{t',x'',y'}(B_{\delta^{-1}})} \\
&\lesssim \big\| w_{B_{\delta^{-1}}} \int_{(\xi'',\eta') \in (-1,1)^2} e^{i(x'' \xi'' + y' \eta' + t' \bar{\omega}(\xi'',\eta'))} f''(\xi'',\eta') d\xi'' d\eta' \big\|_{L^4_{t',x'',y'}(\R^3)}.
\end{split}
\end{equation*}

Invoking Theorem \ref{thm:GeneralL2Decoupling} we find
\begin{equation*}
\begin{split}
&\quad \big\| w_{B_{\delta^{-1}}} \int_{(\xi'',\eta') \in (-1,1)^2} e^{i( x'' \xi'' + y' \eta' + t' \bar{\omega}(\xi'',\eta'))} f''(\xi'',\eta') d\xi'' d\eta' \big\|_{L^4_{t',x'',y'}(\R^3)} \\
&\lesssim_\varepsilon \delta^{-\varepsilon} \big( \sum_{\substack{\theta: \delta-\text{flat} \\
\text{set of } \bar{\omega}}} \big\| w_{B_{\delta^{-1}}} \int_\theta e^{i(x'' \xi'' + y' \eta' + t' \bar{\omega}(\xi'',\eta'))} f''(\xi'',\eta') d\xi'' d\eta' \big\|^2_{L^4(\R^3)} \big)^{\frac{1}{2}}.
\end{split}
\end{equation*}

As a consequence of Minkowski's inequality we obtain decoupling on $\mathcal{E}_{N^3 \times N^3/ \eta_0 \times N^3}$ into $\delta$-flat sets (recall $\delta = \eta_0 / N^3$):
\small
\begin{equation*}
\begin{split}
&\quad \big\| \int_{(\xi'',\eta') \in (-1,1)^2} e^{i(x'' \xi'' + y' \eta' + t' \bar{\omega}(\xi'',\eta'))} f''(\xi'',\eta') d\xi'' d\eta' \big\|^4_{L^4_{t',x'',y'}(\mathcal{E}_{N^3 \times N^3 / \eta_0 \times N^3})} \\
&\lesssim_\varepsilon N^\varepsilon \big( \sum_{\theta: \delta\text{-flat set}} \big\| \int_{(\xi'',\eta') \in \theta} e^{i(x'' \xi'' + y' \eta' + t' \bar{\omega}(\xi'',\eta'))} f''(\xi'',\eta') d\xi'' d\eta' \big\|^2_{L^4_{t',x'',y'}(w_{\mathcal{E}_{N^3 \times N^3 / \eta_0 \times N^3}})} \big)^2.
\end{split}
\end{equation*}
\normalsize
Now we reverse all the scalings, the inflation by periodicity and the continuous approximation to find the $\ell^2$-decoupling inequality
\begin{equation}
\label{eq:DecouplingShiftedPhase}
\begin{split}
&\quad \big\| \sum_{(\xi,\eta) \in \Z^2} e^{i(x \xi + y \eta + t \omega(\xi,\eta))} \hat{f}(\xi,\eta) \big\|_{L^4_{t,x,y}([0,1] \times \T^2)} \\
&\lesssim_\varepsilon N^\varepsilon \big( \sum_{\substack{\theta: \delta\text{-flat} \\
\text{set for } \bar{\omega}}} \big\| \sum_{(\xi,\eta) \in \tilde{\theta}} e^{i(x \xi + y \eta + t \omega(\xi,\eta))} \hat{f}(\xi,\eta) \big\|^2_{L^4_{t,x,y}([0,1] \times \T^2)} \big)^{\frac{1}{2}}.
\end{split}
\end{equation}

By Proposition \ref{prop:FlatSetsShifted} the $\delta$-flat sets for $\bar{\omega}$ can be contained into rectangles with long side of length at most $\eta_0 \cdot \eta_0^{\frac{1}{3}} N^{-1}$ into direction $\mathfrak{n}_1 = (\xi_1',\eta_1')$ with $|\xi'| \gtrsim 1$ and for $(\tilde{\xi},\tilde{\eta}) \in R'$ into $\eta'$-direction of length at most $\eta_0^{\frac{1}{2}} N^{-\frac{3}{2}}$. When reversing the scaling and the normalization of the Fourier support $\xi' = \xi_0 + \xi''/\eta_0$ we find length of $\eta_0^{\frac{1}{3}}$ into $\mathfrak{n}_1$-direction and $\eta_0^{\frac{1}{2}} N^{\frac{1}{2}}$ into $\eta'$-direction. After applying Bernstein's inequality in $\xi$, which incurs a factor of $\eta_0^{\frac{1}{12}}$, we find
\begin{equation}
\label{eq:BernsteinShiftedPhase}
\begin{split}
&\quad \big\| \sum_{(\xi,\eta) \in \tilde{\theta}} e^{i(x \xi + y \eta + t \omega(\xi,\eta))} \hat{f}(\xi,\eta) \big\|_{L^4_{t,x,y}([0,1] \times \T^2)} \\
&\lesssim \eta_0^{\frac{1}{12}}  \big( \sum_{\xi \in \pi_{\xi}(\tilde{\theta})}  \big\| \sum_{\eta \in I_{\xi}} e^{i(y \eta + t \frac{\eta^2}{\xi})} a_{\xi,\eta} \big\|^2_{L^4_{t,y}([0,1] \times \T)} \big)^{\frac{1}{2}}.
\end{split}
\end{equation}

By the above we have $|I_\xi| \sim \eta_0^{\frac{1}{2}} N^{\frac{1}{2}}$. Recall that $|\xi| \sim N$. We carry out a change of variables $t/\xi = t'$:
\begin{equation*}
\big\| \sum_{\eta \in I_{\xi}} e^{i(y \eta + t \frac{\eta^2}{\xi})} a_{\xi,\eta} \big\|_{L^4_{t,y}([0,1] \times \T)} \lesssim N^{\frac{1}{4}} \big\| \sum_{\eta \in I_{\xi}} e^{i(y \eta + t \eta^2)} a_{\xi,\eta} \big\|_{L^4_{t,y}([0,N^{-1}] \times \T)}. 
\end{equation*}
Now we apply H\"older's inequality in time and a lossless Strichartz estimate on frequency dependent times due to \cite{StaffilaniTataru2002} and \cite{BurqGerardTzvetkov2004} noting that $|I_{\xi}| \lesssim N$ by assumption $|\eta_0| \lesssim N$:
\begin{equation}
\label{eq:ShorttimeStrichartzShifted}
\begin{split}
\big\| \sum_{\eta \in I_{\xi}} e^{i(y \eta + t \eta^2)} a_{\xi,\eta} \big\|_{L^4_{t,y}([0,N^{-1}] \times \T)} &\lesssim N^{\frac{1}{4}} \big\| \sum_{\eta \in I_{\xi}} e^{i(y \eta + t \eta^2)} a_{\xi,\eta} \big\|_{L^8_{t}([0,N^{-1}],L^4_y(\T))} \\
&\lesssim N^{\frac{1}{8}} \| a_{\xi,\eta} \|_{\ell^2_{\eta}}.
\end{split}
\end{equation}
Taking \eqref{eq:BernsteinShiftedPhase} and \eqref{eq:ShorttimeStrichartzShifted} together we obtain
\begin{equation}
\label{eq:StrichartzFlatSetShifted}
\begin{split}
&\quad \big\| \sum_{(\xi,\eta) \in \tilde{\theta}} e^{i(x \xi + y \eta + t \omega(\xi,\eta))} \hat{f}(\xi,\eta) \big\|_{L^4_{t,x,y}([0,1] \times \T^2)} \\
&\lesssim \eta_0^{\frac{1}{12}} N^{\frac{1}{8}} \| a_{\xi,\eta} \|_{\ell^2_{\xi,\eta \in \tilde{\theta}}}.
\end{split}
\end{equation}
Plugging \eqref{eq:StrichartzFlatSetShifted} into \eqref{eq:DecouplingShiftedPhase} we complete the proof by $\log(N)$-overlap of $\tilde{\theta}$.
\end{proof}

For future reference we record the following consequence of the well-known transfer principle - for a function localized in modulation the estimates for free solutions hold with a modulation weight:
\begin{proposition}[Transfer~principle~for~$L^4$-Strichartz~estimates]
\label{prop:TransferPrinciple}
    Let $A,B \subseteq \R$, and $u \in L^2(\R \times \Z^2)$ be a function with $\operatorname{supp}(u) \subseteq \{ (\tau, \xi, \eta) \in \R^3 : \xi \in A, \, \eta \in B, \, |\tau - \omega(\xi,\eta)| \leq L \}$.  Assume that the estimate
    \begin{equation*}
        \| S_{\mathrm{KP}}(t) f \|_{L^4_{t,x,y}([0,1] \times \T^2)} \lesssim C(A,B) \| f \|_{L^2(\T^2)}
    \end{equation*}
    holds for functions with $\operatorname{supp}(\hat{f}) \subseteq \{ (\xi,\eta) \in \Z^2 : \xi \in A, \, \eta \in B \}$. Then the following estimate holds:
    \begin{equation*}
        \| \mathcal{F}^{-1} [u] \|_{L^4_{t,x,y}(\R \times \T^2)} \lesssim C(A,B) L^{\frac{1}{2}} \| u \|_{L^2_{t,x,y}(\R \times \T^2)}.
    \end{equation*}
\end{proposition}
\begin{proof}
This follows from the representation
\[
\mathcal{F}^{-1} [u](t)=\frac{1}{2\pi}\int_{-L}^L e^{it\tau } S_{\mathrm{KP}}(t)f(\tau) d\tau, \; f(\tau):=\mathcal{F}_t (S_{\mathrm{KP}}(-\cdot) \mathcal{F}^{-1}[u])(\tau),
\]
the assumed estimate, the Cauchy-Schwarz inequality, and Plancherel.
\end{proof}

\section{Bilinear Strichartz estimates}
\label{section:BilinearStrichartz}

The following Strichartz estimate was proved in \cite[Section~4]{Bourgain1993}:
\begin{proposition}[Bourgain]
\label{prop:BilinearStrichartzBourgain}
Let $f_j : \R \times \Z \times \Z \to \C$ such that for $(\tau_j,\xi_j,\eta_j) \in \operatorname{supp}(f_j)$ the following holds:
\begin{equation*}
\xi_j \in I_j, \; |\xi_j| \sim N_j, \; |\tau_j - \omega(\xi_j,\eta_j)| \leq L_j,
\end{equation*}
where $I_j \subseteq \Z$ is an interval and $N_j$, $L_j$ are dyadic numbers. Let $L_{\max} = \max(L_1,L_2)$ and $N_{\max} = \max(N_1,N_2)$. Then the estimate holds:
\begin{equation*}
\| f_1 * f_2 \|_{L^2(\R \times \Z^2)} \lesssim B(N_1,N_2,L_1,L_2,|I_1|,|I_2|) \| f_1 \|_{L^2} \| f_2 \|_{L^2},
\end{equation*}
where
\begin{equation}
\label{eq:BConstant}
B \lesssim L_{\min}^{\frac{1}{2}} L_{\max}^{\frac{1}{4}} N_{\min}^{\frac{1}{4}} + L_{\min}^{\frac{1}{2}} I_{\min}^{\frac{1}{2}} + \min \{ L_{\min}^{\frac{1}{2}} L_{\max}^{\frac{1}{2}+\varepsilon} \big( \frac{N_{\min}}{N_{\max}} \big)^{\frac{1}{4}} , L_{\min}^{\frac{1}{2}} L_{\max}^{\frac{1}{4}} N_{\min}^{\frac{1}{4}} I_{\min}^{\frac{1}{2}} \}.
\end{equation}
Here $X_{\min} = \min \{ X_1,X_2 \}$ and $X_{\max} = \max \{ X_1, X_2 \}$ for $X \in \{ L, N \}$. Moreover, $I_{\min} = \min (|I_1|,|I_2|)$ and $I_{\max} = \max (|I_1|,|I_2|)$.
\end{proposition}
We could apply this in the proof of Theorem \ref{thm:LWPKPII} and with Theorem \ref{thm:L4Strichartz} would still get a result below $L^2$. Instead, we use Proposition \ref{prop:bilinear-sp} below, which we hope makes the argument easier to follow.
Before stating Proposition \ref{prop:bilinear-sp}, we state the well-known bilinear estimate based on transversality, recall \eqref{eq:trans}, which is frequently used. The proof is implicit in  \cite[Section~4]{Bourgain1993}, we give the details for the sake of completeness.
\begin{proposition}
\label{prop:SimplifiedBilinearStrichartz}
Let $f_i \in L^2(\R \times \Z^2)$ satisfy the assumptions of Proposition \ref{prop:BilinearStrichartzBourgain} and suppose in addition 
\begin{equation*}
\big| \frac{\eta_1}{\xi_1} - \frac{\eta_2}{\xi_2} \big| \sim D
\end{equation*}
for $(\xi_i,\eta_i) \in \operatorname{supp}_{\xi,\eta} f_i$. Then the following estimate holds:
\begin{equation*}
\| f_1 * f_2 \|_{L^2(\R \times \Z^2)} \lesssim I_{\min}^{\frac{1}{2}} L_{\min}^{\frac{1}{2}} \langle L_{\max} / D \rangle^{\frac{1}{2}} \prod_{i=1}^2 \| f_i \|_{L^2(\R \times \Z^2)}.
\end{equation*}
\end{proposition}
\begin{proof}
By convolution constraint and almost orthogonality we can suppose that the $\xi$-support of $f_i$, $i=1,2$ is contained in intervals of length $I_{\min}$. Suppose by symmetry that $L_1 = L_{\min}$. We apply the Cauchy-Schwarz inequality to find
\begin{equation*}
\begin{split}
&\quad \int \big| \int f_1(\tau_1,\xi_1,\eta_1) f_2(\tau-\tau_1,\xi-\xi_1,\eta-\eta_1) d\tau_1 d\xi_1 d\eta_1 \big|^2 d\tau d\xi d\eta \\
&\lesssim \sup_{\tau,\xi,\eta} A(\tau,\xi,\eta) \iint |f_1(\tau_1,\xi_1,\eta_1)|^2 |f_2(\tau-\tau_1,\xi-\xi_1,\eta-\eta_1)|^2 d\tau_1 d\xi_1 d\eta_1 d\tau d\xi d\eta \\
&\lesssim \text{meas}_{\R \times \Z^2} A(\tau,\xi,\eta) \| f_1 \|_{L^2_{\tau,\xi,\eta}}^2 \| f_2 \|_{L^2_{\tau \xi \eta}}^2
\end{split}
\end{equation*}
with
\begin{equation*}
\begin{split}
A(\tau,\xi,\eta) &= \{ (\tau_1,\xi_1,\eta_1) \in \R \times \Z^2 : (\tau_1,\xi_1,\eta_1) \in \operatorname{supp}(f_1), \\
&\quad (\tau-\tau_1,\xi-\xi_1,\eta-\eta_1) \in \operatorname{supp}(f_2) \}.
\end{split}
\end{equation*}
We estimate the measure as follows:
By assumption we have
\begin{equation}
\label{eq:xi1EstimateBilSimplified}
\# \{ \xi_1 \in \pi_{\xi_1}(A(\tau,\xi,\eta)) \} \lesssim I_{\min}.
\end{equation}
For fixed $\xi_1,\xi,\eta$ and $\tau$ we count the number of $\eta_1$ by writing
\begin{equation*}
\begin{split}
\tau_1 - \omega(\xi_1,\eta) + ((\tau-\tau_1)-\omega(\xi-\xi_1,\eta-\eta_1)) &= \tau- \omega(\xi_1,\eta_1) - \omega(\xi-\xi_1,\eta-\eta_1) \\ &=: g_{\tau,\xi,\xi_1,\eta}(\eta_1).
\end{split}
\end{equation*}
This shows that $g_{\tau,\xi,\xi_1,\eta_1}$ ranges in an interval of at most $L_{\max}$ and since
\begin{equation*}
\partial_{\eta_1} g = 2 \big( \frac{\eta-\eta_1}{\xi - \xi_1} - \frac{\eta_1}{\xi_1} \big), \quad |\partial_{\eta_1} g | \sim D,
\end{equation*}
we can estimate the number of $\eta_1$ for fixed $\xi_1$ by the mean-value theorem:
\begin{equation}
\label{eq:eta1EstimateBilSimplified}
\# \{ \eta_1: (\xi_1,\eta_1) \in A_{\xi_1,\eta_1}(\tau,\xi,\eta) \} \lesssim \langle L_{\max} / D \rangle.
\end{equation}
For fixed $\xi_1$, $\eta_1$ we have
\begin{equation}
\label{eq:tau1EstimateBilSimplified}
| \{ \tau_1 \in \R : (\tau_1,\xi_1,\eta_1) \in A(\tau,\xi,\eta) \} | \lesssim L_{\min}.
\end{equation}
Taking \eqref{eq:xi1EstimateBilSimplified}, \eqref{eq:eta1EstimateBilSimplified}, and \eqref{eq:tau1EstimateBilSimplified} together we complete the proof.
\end{proof}

We record the following bilinear estimate hinging on second-order transversality, which is again implicit in \cite[Section~4]{Bourgain1993}:
\begin{proposition}
\label{prop:SecondOrderBilinearEstimate}
Let $N_i,L_i \in 2^{\N_0}$, $N_2 \lesssim N_1$, and $f_i \in L^2(\R \times \Z^2)$ satisfy the assumptions of Proposition \ref{prop:BilinearStrichartzBourgain}. Suppose that the intervals $I_i$ satisfy $|I_i| \leq A$ for $i=1,2$. Assume for $(\tau_i,\xi_i,\eta_i) \in \operatorname{supp}(f_i)$ that
\begin{equation*}
|\xi_i| \sim N_i, \quad \xi_i \in I_i, \quad |\tau_i - \omega(\xi,\eta)| \leq L_i.
\end{equation*}
Then the following estimate holds:
\begin{equation*}
\| f_1 * f_2 \|_{L^2_{\tau,\xi,\eta}} \lesssim \langle A \rangle^{\frac{1}{2}} L_{\min}^{\frac{1}{2}} \langle L_{\max} N_2 \rangle^{\frac{1}{4}} \prod_{i=1}^2 \| f_i \|_{L^2}.
\end{equation*}
\end{proposition}
\begin{proof}
The argument follows along the lines of the previous proof. The difference is that we are considering the second derivative in $\eta_1$ of $f_{\tau,\xi,\xi_1}$:
\begin{equation*}
\partial_{\eta_1}^2 f =2 \big( \frac{1}{\xi_1} + \frac{1}{\xi - \xi_1} \big).
\end{equation*}
We can suppose that $\xi_1$ and $\xi-\xi_1$ are of the same sign, as a consequence of the symmetry of the dispersion relation and complex conjugation. We find the bound
\begin{equation*}
|\partial_{\eta_1}^2 f| \sim \frac{1}{N_{\min}}.
\end{equation*}
By the second order mean value inequality we obtain the alternative estimate
\begin{equation*}
\label{eq:eta1EstimateBilSimplifiedII}
\# \{\eta_1:  (\xi_1,\eta_1) \in A_{\xi_1,\eta_1}(\tau,\xi,\eta) \} \lesssim (L_{\max} N_{\min})^{\frac{1}{2}}.
\end{equation*}
Taking \eqref{eq:xi1EstimateBilSimplified}, \eqref{eq:eta1EstimateBilSimplifiedII}, and \eqref{eq:tau1EstimateBilSimplified} together yields the claim.
\end{proof}

The following bilinear estimate relates to a special case of Proposition \ref{prop:BilinearStrichartzBourgain} and can be proved with the arguments in \cite{Bourgain1993}. Here we give a different proof based on the biorthogonality, which underpins the proof of the C\'ordoba-Fefferman square function estimate \cite{Cordoba1982} (see also \cite{Fefferman1973,Cordoba1979}). We refer to \cite[Section~3.2]{Demeter2020} for a modern treatise.

\begin{proposition}
\label{prop:bilinear-sp}
Let $D^* \geq 2$, $f_i : \R \times \Z \times \Z \to \C$ such that for $(\tau_i,\xi_i,\eta_i) \in \operatorname{supp}(f_i)$ the following holds:
\begin{equation*}
|\xi_i| \sim N_i ,\; |\tau_i - \omega(\xi_i,\eta_i)| \lesssim L_i, \; \big| \frac{\eta_1}{\xi_1} - \frac{\eta_2}{\xi_2} \big| \lesssim D^*
\end{equation*}
for dyadic numbers satisfying $N_2\ll N_1$ and $N_1^\alpha \leq L_{\max}\lesssim N_1^2N_2$.
Then the following estimate holds:
\begin{equation}\label{eq:bilinear-sp}
\| f_1 * f_2 \|_{L^2(\R \times \Z^2)} \lesssim \log(D^*) ( N_2^{\frac{1}{2}} N_1^{-\frac{\alpha}{2}} + N_2^{\frac{1}{4}} N_1^{-\frac{\alpha}{4}}) \prod_{i=1}^2 L_i^{\frac{1}{2}} \| f_i \|_{L^2}.
\end{equation}
\end{proposition}
\begin{proof}
We carry out a dyadic decomposition in $D$ such that we can suppose in the following that
\begin{equation*}
\big| \frac{\eta_1}{\xi_1} - \frac{\eta_2}{\xi_2} \big| \sim D.
\end{equation*}
Summation over $(\frac{L_{\max}}{N_2})^{\frac{1}{2}} \leq D \leq D^*$ incurs the logarithmic loss in $D^*$.

Applying Proposition \ref{prop:SimplifiedBilinearStrichartz} directly yields
\begin{equation*}
\| f_{1,N_1,L_1} * f_{2,N_2,L_2} \|_{L^2_{\tau,\xi,\eta}} \lesssim N_2^{\frac{1}{2}} L_{\min}^{\frac{1}{2}} \langle L_{\max}/D \rangle^{\frac{1}{2}} \prod_{i=1}^2 \|f_{i,N_i,L_i} \|_{L^2_{\tau,\xi,\eta}},
\end{equation*}
which suffices for $D \gtrsim N_1$, but does not in case $D \ll N_1$.

\medskip

In the latter case we shall see that we can significantly decrease the effective $\xi$-support using a variant of the \emph{Córdoba-Fefferman square function estimate.} By symmetry of the dispersion relation, we suppose in the following that $\xi_1, \xi_2 > 0$. 

\medskip

We rewrite the convolution constraint as
\begin{equation*}
\begin{cases}
\xi_1 + \xi_2 &= \xi_3 + \xi_4, \\
\eta_1 + \eta_2 &= \eta_3 + \eta_4, \\
\tau_1 + \tau_2 &= \tau_3 + \tau_4.
\end{cases}
\Leftrightarrow
\begin{cases}
\xi_1 + \xi_2 &= \xi_3 + \xi_4, \\
\eta_1 + \eta_2 &= \eta_3 + \eta_4, \\
\xi_1^3 - \frac{\eta_1^2}{\xi_1} + \xi_2^3 - \frac{\eta_2^2}{\xi_2} &= \xi_3^3 - \frac{\eta_3^2}{\xi_3} + \xi_4^3 - \frac{\eta_4^2}{\xi_4} + \mathcal{O}(L_{\max}).
\end{cases}
\end{equation*}
From the third line we subtract $\omega(\xi_1+\xi_2,\eta_1+\eta_2) = \omega(\xi_3+\xi_4,\eta_3+\eta_4)$ to find
\begin{equation*}
\begin{cases}
\xi_1 + \xi_2 &= \xi_3 + \xi_4, \\
\eta_1 + \eta_2 &= \eta_3 + \eta_4, \\
3 \xi_1 \xi_2 (\xi_1 + \xi_2) - \frac{(\eta_1 \xi_2 - \eta_2 \xi_1)^2}{(\xi_1+\xi_2) \xi_1 \xi_2} &= 3 \xi_3 \xi_4 (\xi_3 + \xi_4) - \frac{(\eta_3 \xi_4 - \eta_4 \xi_3)^2}{(\xi_3 + \xi_4) \xi_3 \xi_4} + \mathcal{O}(L_{\max}).
\end{cases}
\end{equation*}

Note that
\begin{equation*}
\begin{split}
\frac{(\eta_1 \xi_2 - \eta_2 \xi_1)^2}{(\xi_1+\xi_2) \xi_1 \xi_2} &= \frac{(\eta_1 / \xi_1 - \eta_2 / \xi_2)^2 \xi_1 \xi_2}{\xi_1 + \xi_2} = \mathcal{O}(D^2 \cdot N_2), \\
\frac{(\eta_3 \xi_4 - \eta_4 \xi_3)^2}{(\xi_3+\xi_4) \xi_3 \xi_4} &= \mathcal{O}(D^2 \cdot N_2).
\end{split}
\end{equation*}
We rescale to unit frequencies using the anisotropic scaling $\xi_i \to \xi_i / N_1$, $\eta_i \to \eta_i / N_1^2$:
\begin{equation*}
\begin{cases}
\xi'_1 + \xi'_2 &= \xi'_3 + \xi'_4, \\
\eta'_1 + \eta'_2 &= \eta'_3 + \eta'_4, \\
\xi'_1 \xi'_2 (\xi'_1 + \xi'_2) &= \xi'_3 \xi'_4 (\xi'_3 + \xi'_4) + \mathcal{O}(\frac{D^2 N_2}{N_1^3} + \frac{L_{\max}}{N_1^3}).
\end{cases}
\end{equation*}

Let $A = \frac{D^2 N_2}{N_1^3} + \frac{L_{\max}}{N_1^3}$ and note that $A \ll 1$ under the above assumptions. The above implies by $\xi_1'+\xi_2' = \xi_3' + \xi_4' \sim 1$
\begin{equation*}
\begin{cases}
\xi_1' + \xi_2' &= \xi_3' + \xi_4', \\
\xi_1' \xi_2' &= \xi_3' \xi_4' + \mathcal{O}(A).
\end{cases}
\end{equation*}
Furthermore, by squaring the first line and subtracting the second line twice,
\begin{equation*}
\begin{cases}
\xi_1' + \xi_2' &= \xi_3' + \xi_4', \\
(\xi_1')^2 + (\xi_2')^2 &= (\xi_3')^2 + (\xi_4')^2 + \mathcal{O}(A).
\end{cases}
\end{equation*}

By the bilinear C\'ordoba--Fefferman square function estimate, this yields an almost orthogonal decomposition of $\xi'_i$ into intervals of length $A$. For self-contained\-ness we include the argument. The second line of the above display implies by convolution constraint and the first line
\begin{equation*}
(\xi_1')^2 + (\xi_2')^2 = (\xi_3')^2 + (\xi_4')^2 + \mathcal{O}(A) \Rightarrow 2 (\xi_1' - \xi_2')(\xi_2'-\xi_4') = \mathcal{O}(A).
\end{equation*}
Since $|\xi_1'-\xi_2'| \sim 1$ by $N_2 \ll N_1$ and rescaling, we find $|\xi_2'-\xi_4'| \lesssim A$ and correspondingly, $|\xi_1'-\xi_3'| \lesssim A$. For the $\xi_i$-frequencies after scaling back this translates into an almost orthogonal decomposition into intervals in $\Z$ of length
\begin{equation*}
\ell =  \langle A N_1 \rangle \sim \frac{D^2 N_2}{N_1^2} + \frac{L_{\max}}{N_1^2} + 1.
\end{equation*}
We denote the additional decomposition in $\xi$ into intervals of length $\ell$ by
\begin{equation*}
f_{i,N_i,L_i} = \sum_{I_i} f^{I_i}_{i,N_i,L_i}
\end{equation*}
to obtain the square function estimate
\begin{equation}
\label{eq:AlmostOrthogonalDec}
\int |f_{1,N_1,L_1} * f_{2,N_2,L_2}|^2 \lesssim \sum_{I_1,I_2} \int |f^{I_1}_{1,N_1,L_1} * f^{I_2}_{2,N_2,L_2} |^2.
\end{equation}

In case $D \sim \big( \frac{L_{\max}}{N_2} \big)^{\frac{1}{2}}$ we apply the bilinear Strichartz estimate from Proposition \ref{prop:SecondOrderBilinearEstimate} to find
\begin{equation*}
\| f^{I_1}_{1,N_1,L_1} * f^{I_2}_{2,N_2,L_2} \|_{L^2} \lesssim \langle L_{\max}^{\frac{1}{2}} / N_1 \rangle L_{\min}^{\frac{1}{2}} (L_{\max} N_2 )^{\frac{1}{4}} \prod_{i=1}^2 \| f^{I_i}_{i,N_i,L_i} \|_2.
\end{equation*}
Then the claim follows from the almost orthogonal summation over $I_1$ and $I_2$. This argument is omitted in the following for brevity. If $L_{\max} \lesssim N_1^2$, then we can use the lower bound on $L_{\max}$ to find after plugging in the almost orthogonal decomposition \eqref{eq:AlmostOrthogonalDec}
\begin{equation*}
\begin{split}
\| f_{1,N_1,L_1} * f_{2,N_2,L_2} \|_{L^2} &\lesssim L_{\min}^{\frac{1}{2}} L_{\max}^{\frac{1}{2}} (N_1^{-\alpha} N_2 )^{\frac{1}{4}} \prod_{i=1}^2 \| f_{i,N_i,L_i} \|_2 \\
&\lesssim \langle N_1^{-\alpha} N_2 \rangle^{\frac{1}{2}} \prod_{i=1}^2 L_i^{\frac{1}{2}}\| f_{i,N_i,L_i} \|_2.
\end{split}
\end{equation*}
If $L_{\max} \gtrsim N_1^2$, we use the upper bound $L_{\max} \lesssim N_1^2 N_2$ to obtain
\begin{equation*}
\| f_{1,N_1,L_1} * f_{2,N_2,L_2} \|_{L^2_{\tau,\xi,\eta}} \lesssim N_2^{\frac{1}{2}} / N_1^{\frac{1}{2}} \prod_{i=1}^2 L_i^{\frac{1}{2}} \| f_{i,N_i,L_i} \|_2. 
\end{equation*}

Next, suppose that $ \big( \frac{L_{\max}}{N_2} \big)^{\frac{1}{2}} \ll D \lesssim N_1$: In this case we obtain an almost orthogonal decomposition of $\xi$-frequencies \eqref{eq:AlmostOrthogonalDec} into intervals of length
\begin{equation*}
\langle AN_1 \rangle \sim \frac{D^2 N_2}{N_1^2} + 1.
\end{equation*}
In case $D^2 N_2 \lesssim N_1^2$, we apply Proposition \ref{prop:SecondOrderBilinearEstimate} after plugging in \eqref{eq:AlmostOrthogonalDec} to find
\begin{equation*}
\begin{split}
\| f_{1,N_1,L_1} * f_{2,N_2,L_2} \|_{L^2_{\tau,\xi,\eta}} &\lesssim L_{\min}^{\frac{1}{2}} (L_{\max} N_2)^{\frac{1}{4}} \prod_{i=1}^2 \| f_{i,N_i,L_i} \|_2 \\
&\lesssim (N_2 N_1^{-\alpha})^{\frac{1}{4}} \prod_{i=1}^2 L_i^{\frac{1}{2}} \| f_{i,N_i,L_i} \|_2.
\end{split}
\end{equation*}
This yields an acceptable contribution.
In the ultimate estimate we used the minimum size of $L_{\max}$.

\medskip

We turn to $D^2 N_2 \gtrsim N_1^2$. In this case we apply Proposition \ref{prop:SimplifiedBilinearStrichartz} after plugging in \eqref{eq:AlmostOrthogonalDec} to find
\begin{equation*}
\| f_{1,N_1,L_1} * f_{2,N_2,L_2} \|_{L^2_{\tau,\xi,\eta}} \lesssim \frac{D N_2^{\frac{1}{2}}}{N_1} L_{\min}^{\frac{1}{2}} \langle L_{\max} / D \rangle^{\frac{1}{2}} \prod_{i=1}^2 \| f_{i,N_i,L_i} \|_2.
\end{equation*}
Summing in $D \lesssim N_1$ we obtain by the minimum size $L_{\max} \gtrsim N_1^\alpha$ 
\begin{equation*}
\sum_{\frac{N_1}{N_2^{\frac{1}{2}}} \lesssim D \lesssim N_1} \frac{D N_2^{\frac{1}{2}}}{N_1} L_{\min}^{\frac{1}{2}} \langle L_{\max} / D \rangle^{\frac{1}{2}} \prod_{i=1}^2 \| f_{i,N_i,L_i} \|_2 \lesssim N_2^{\frac{1}{2}} N_1^{-\frac{\alpha}{2}} \prod_{i=1}^2 L_i^{\frac{1}{2}} \| f_{i,N_i,L_i} \|_2.
\end{equation*}
This completes the proof.
\end{proof}
\section{Outline of the proof of the nonlinear estimates}
\label{section:Outline}
In this section we explain our argument to show the short-time Fourier restriction estimates for the new well-posedness results in broad strokes. Technical details are provided in the forthcoming sections.

\subsection{Estimating solutions at negative Sobolev regularity}

We begin with a description of Fourier restriction norms without frequency-dependent time localization. Fourier restriction norms for $s,b \in \R$ with modulation weight are defined as
\begin{equation*}
\begin{split}
X^{s,b} &= \{ u : \R \times \T^2 \to \C \; | \| u \|_{X^{s,b}} < \infty \}, \\
\| u \|^2_{X^{s,b}} &= \int_{\R} \sum_{\substack{ (\xi,\eta) \in \Z^2, \\ \xi \neq 0 }} \langle \xi \rangle^{2s} \langle \tau - \omega(\xi,\eta) \rangle^{2b} \langle (\tau-\omega(\xi,\eta))/|\xi|^3 \rangle^{\frac{1}{2}} |\hat{u}(\tau,\xi,\eta)|^2 d\tau.
\end{split}
\end{equation*}
In the limiting case $b=\frac{1}{2}$ a modification becomes necessary for the transfer principle to remain valid. We denote the modified space by $\bar{X}^{s,\frac{1}{2}}$. To avoid technicalities, we do not elaborate on Besov refinements like in \eqref{eq:FourierRestrictionBesov}, which salvage the transfer principle in the limiting case. We refer to \cite[p.~656]{MolinetSautTzvetkov2011} for details on a possible refinement in the limiting case $b=1/2$. Indeed, the Besov refinements coincide with the $X^{s,b}$-norm for functions localized in frequency and modulation. For this reason, we pretend to carry out the estimates in the $X^{s,\frac{1}{2}}$-norm.

Bourgain showed analytic well-posedness via the estimates for solutions
\begin{equation*}
\left\{ \begin{array}{cl}
\| u \|_{\bar{X}^{0,\frac{1}{2}}} &\lesssim \| u_0 \|_{L^2} + \| \partial_x (u^2) \|_{\bar{X}^{0,-\frac{1}{2}}}, \\
\| \partial_x(u^2) \|_{\bar{X}^{0,-\frac{1}{2}}} &\lesssim \| u \|^2_{\bar{X}^{0,\frac{1}{2}}},
\end{array} \right.
\end{equation*}
and a corresponding estimate for differences of solutions (though the arguments to carry out the crucial bilinear  estimate are identical).

\medskip

After Littlewood-Paley decomposition we can assume the functions to be dyadically localized, which up to summation question reduces the crucial bilinear estimate to
\begin{equation*}
N^s \| P_N \partial_x (P_{N_1} u_1 P_{N_2} u_2 ) \|_{\bar{X}^{0,-\frac{1}{2}}} \lesssim N_1^s \| P_{N_1} u \|_{\bar{X}^{0,\frac{1}{2}}} N_2^s \| P_{N_2} u_2 \|_{\bar{X}^{0,\frac{1}{2}}}.
\end{equation*}
We have by duality
\begin{equation*}
    \| P_N \partial_x (P_{N_1} u_1 P_{N_2} u_2) \|_{\bar{X}^{0,-\frac{1}{2}}} = \sup_{\| u \|_{\bar{X}^{0,\frac{1}{2}}} = 1} \iint P_N \overline{u} \partial_x (P_{N_1} u_1 P_{N_2} u_2) dx dy dt.
\end{equation*}
We introduce modulation projections to write
\begin{equation*}
    \iint P_N \overline{u} \partial_x (P_{N_1} u_1 P_{N_2} u_2) = \sum_{L, L_1, L_2} \iint (P_N Q_L \partial_x \overline{u}) \, ( P_{N_1} Q_{L_1} u_1 ) \, ( P_{N_2} Q_{L_2} u_2).
\end{equation*}
For $i=1,2$ we let
\begin{equation*}
    f_{i,N_i,L_i} = |\mathcal{F}_{t,x,y}[P_{N_i} Q_{L_i} u_i]|
\end{equation*}
and $f_{3,N_3,L_3} = |\mathcal{F}_{t,x,y}[P_N Q_L u]|$. The bound
\begin{equation*}
    \iint P_N Q_L \overline{u} \partial_x (P_{N_1} Q_{L_1} u_1 P_{N_2} Q_{L_2} u_2) \lesssim  N C(N,N_1,N_2) \| P_N Q_L u \|_{\bar{X}^{0,\frac{1}{2}}} \prod_{i=1}^2 \| P_{N_i} Q_{L_i} u \|_{\bar{X}^{0,\frac{1}{2}}}
\end{equation*}
is for $L_i \lesssim N_i^3$ implied by 
\begin{equation*}
\int f_{3,N_3,L_3} (f_{1,N_1,L_1} * f_{2,N_2,L_2} ) d\xi d\eta d\tau \lesssim C(N,N_1,N_2) \prod_{i=1}^3 L_i^{\frac{1}{2}} \| f_{i,N_i,L_i} \|_{L^2_{\tau,\xi,\eta}}.
\end{equation*}
For modulations $L_i \gg N_i^3$ the weight $(1+L_i / N_i^3)^{\frac{1}{4}}$ must be taken into account. 

\smallskip

Suppose in the following that the frequencies are comparable $N \sim N_1 \sim N_2$.
Recall the defocusing character of the resonance function
\begin{equation*}
    \Omega(\xi,\xi_1,\xi_2) = 3 \xi_1 \xi_2(\xi_1+\xi_2) + \frac{(\eta_1 \xi_2 - \eta_2 \xi_1)^2}{\xi_1 \xi_2 (\xi_1+ \xi_2)}.
\end{equation*} 
By convolution constraint, we have 
\begin{equation*}
    -\Omega = \tau - \omega(\xi_1+\xi_2,\eta_1+\eta_2) - (\tau_1 - \omega(\xi_1,\eta_1)) - (\tau_2 - \omega(\xi_2,\eta_2)) \Rightarrow L_{\max} \gtrsim N_1 N_2 N. 
\end{equation*}
Problematic interactions arise in the \emph{resonant case} $N_1^3 \sim L_{\max} \gg L_{\text{med}}$. 
Suppose by symmetry $L_{\max} \sim L_3$ and note that this is the minimal value $L_{\max}$ can attain. By an almost orthogonality argument and Galilean invariance we can suppose the localization $|\eta_i| \lesssim N_1^2$.

By an application of H\"older's inequality together with the $L^4_{t,x,y}$-Strichartz estimates provided by Theorem \ref{thm:L4Strichartz} yields
\begin{equation*}
\begin{split}
\big| \int f_{3,N_3,L_3} (f_{1,N_1,L_1} * f_{2,N_2,L_2}) d\tau d\xi d\eta \big| &\leq \| f_{3,N_3,L_3} \|_{L^2_{\tau,\xi,\eta}} \prod_{i=1}^2 \| \mathcal{F}^{-1}_{t,x,y}[f_{i,N_i,L_i}] \|_{L^4_{t,x,y}} \\
&\lesssim_\varepsilon \| f_{3,N_3,L_3} \|_{L^2_{\tau,\xi,\eta}} N_1^{\frac{1}{4}+\varepsilon} \prod_{i=1}^2 L_i^{\frac{1}{2}} \| f_{i,N_i,L_i} \|_{L^2_{\tau,\xi,\eta}} \\
&\lesssim_\varepsilon N_1^{-\frac{5}{4}+\varepsilon} \prod_{i=1}^3 L_i^{\frac{1}{2}} \| f_{i,N_i,L_i} \|_{L^2_{\tau,\xi,\eta}}.
\end{split}
\end{equation*}

This shows that the resonant High$\times$High $\rightarrow$ High-interaction can be estimated without frequency-dependent time localization at regularities $s>-\frac{1}{4}$.

Let $M \geq N_1$. If $L_{\max} \sim M^2 N_1 \gg N_1^3$, but $L_{\text{med}} \ll L_{\max}$, which is referred to as \emph{non-resonant case}, the corresponding localization $|\eta_i| \lesssim M^2 N_1$ leads to lossier $L^4$-Strichartz estimates provided by Proposition \ref{prop:L4StrichartzShifted}. Still an estimate via $L^4$-Strichartz estimates remains possible for negative Sobolev regularity provided that $M \lesssim N_1^{1+a_1}$ for some $a_1 > 0$ to be determined. On the other hand, for $M \gtrsim N_1^{1+a_1}$, where we only use that $a_1 > 0$, since the transversality constant satisfies
\begin{equation*}
D \sim \big| \frac{\eta-\eta_1}{\xi - \xi_1} - \frac{\eta_1}{\xi_1} \big| \gtrsim N_1^{1+a_1},
\end{equation*}
 we can use a bilinear Strichartz estimate to find
\begin{equation*}
\begin{split}
\big| \int f_{3,N_3,L_3} (f_{1,N_1,L_1} * f_{2,N_2,L_2}) d\tau d\eta d\xi \big| &\leq \| f_{3,N_3,L_3} \|_{L^2_{\tau,\xi,\eta}} \| f_{1,N_1,L_1} * f_{2,N_2,L_2} \|_{L^2_{\tau,\xi,\eta}} \\
&\lesssim (N_1^{3+2a_1})^{-\frac{1}{2}} N_1^{\frac{1}{2}} \prod_{i=1}^3 L_i^{\frac{1}{2}} \| f_{i,N_i,L_i} \|_{L^2_{\tau,\xi,\eta}}.
\end{split}
\end{equation*}

Interpolation between the new $L^4$-Strichartz estimates and the bilinear $L^2$-Stri\-chartz estimate already due to Bourgain is the key to show the new well-posedness results.

\smallskip

We remark that there is also the \emph{strongly non-resonant case}, $L_{\max} \sim L_{\text{med}}$. In this case, since $L_{\max} \gtrsim N_1^3$, the bilinear Strichartz estimate provided by Proposition \ref{prop:BilinearStrichartzBourgain} (more precisely, using the ultimate constant in \eqref{eq:BConstant}) allows us to cover the full subcritical range.

Suppose $L_{\max} = L_3$, and $L_{\text{med}} = L_1$ by symmetry. We have
\begin{equation*}
\begin{split}
L^{-\frac{1}{4}} N^{-\frac{3}{4}} \| 1_{D_{N,L}} (f_{1,N_1,L_1} * f_{2,N_2,L_2}) \|_{L^2_{\tau,\xi,\eta}} &\lesssim L^{-\frac{1}{4}} N^{-\frac{3}{4}} N^{\frac{3}{4}} L_2^{\frac{1}{2}} L_1^{\frac{1}{4}} \prod_{i=1}^2 \| f_{i,N_i,L_i} \|_{L^2_{\tau,\xi,\eta}} \\
&\lesssim L^{-\frac{1}{4}} N^{-\frac{3}{4}} \prod_{i=1}^2 L_i^{\frac{1}{2}} \| f_{i,N_i,L_i} \|_{L^2_{\tau,\xi,\eta}}.
\end{split}
\end{equation*}
The strongly non-resonant case never poses a problem.

\medskip

Next, consider the High$\times$Low $\rightarrow$ High-interaction $N_1 \sim N_3 \gg N_2$ and suppose that $L_{\max} = N_1^2 N_2 \gg L_{\text{med}}$. Moreover, suppose that $L_3 = L_{\max}$. If $N_2 \lesssim N_1^{1-a_2}$, the linear $L^4$-Strichartz estimates are not effective anymore. In case of small transversality $D \lesssim N_2$ it is a consequence of the biorthogonality used in the proof of the Córdoba--Fefferman square function estimate to effectively reduce the $\xi$-frequency support:
\begin{equation}
\label{eq:AlmostOrthogonal}
\int f_{3,N_3,L_3} (f_{1,N_1,L_1} * f_{2,N_2,L_2}) d\tau d\xi d\eta \lesssim \| f_{3,N_3,L_3} \|_{L^2_{\tau,\xi,\eta}} \sum_{I_1 \sim I_2} \| f_{1,N_1,L_1}^{I_1} * f_{2,N_2,L_2}^{I_2} \|_{L^2_{\tau,\xi,\eta}}.
\end{equation}

The argument states in summary that there is an almost orthogonal decomposition of the $\xi_i$-frequencies of $f$ into intervals $I_i$ of length $N_2 (N_2 / N_1)$; the resulting functions with additionally restricted frequency support are denoted by $f_{I_i}$. With $I_1 \sim I_2$ we mean that for any interval $I_1$ there are only finitely many $I_2$ such that 
\begin{equation*}
    \int f_{3,N_3,L_3} ( f^{I_1}_{1,N_1,L_1} * f^{I_2}_{2,N_2,L_2} ) \neq 0.
\end{equation*}

With \eqref{eq:AlmostOrthogonal} at hand, the bilinear Strichartz estimates (Proposition \ref{prop:BilinearStrichartzBourgain}, \ref{prop:SimplifiedBilinearStrichartz}) become more effective.

\medskip

The interaction we cannot overcome without frequency-dependent time localization is the High$\times$Low $\rightarrow$ High-interaction with $N_2 \lesssim N_1^{1-a_2}$ and large transversality $D \sim N_1$. In this case the the bilinear Strichartz estimate only yields
\begin{equation}
\label{eq:SketchBilinear}
\begin{split}
&\quad \int f_{3,N_3,L_3} (f_{1,N_1,L_1} * f_{2,N_2,L_2}) d\tau d\xi d\eta \\
 &\lesssim \| f_{3,N_3,L_3} \|_{L^2_{\tau,\xi,\eta}} \| f_{1,N_1,L_1} * f_{2,N_2,L_2} \|_{L^2_{\tau,\xi,\eta}} \\
&\lesssim (N_1^2 N_2)^{-\frac{1}{2}} L_3^{\frac{1}{2}} N_2^{\frac{1}{2}} (L_1 \wedge L_2)^{\frac{1}{2}} \langle (L_1 \vee L_2 )/N_1 \rangle^{\frac{1}{2}} \prod_{i=1}^3 \| f_{i,N_i,L_i} \|_{L^2_{\tau,\xi,\eta}}.
\end{split}
\end{equation}

But in case $L_1 \sim L_2 \sim 1$ there is no effective smoothing. However, frequency-dependent time localization $T=T(N)=N^{-\alpha}$ increases the minimum size of modulations $L_{\min} \gtrsim N_1^\alpha$, and short-time bilinear estimates can be carried out at negative Sobolev regularity.
Moreover note, if $L_2 = L_{\max} \gg L_{\text{med}}$, i.e., the low frequency carries the high modulation the transversality is weaker, but the modulation weight $(1+L_i/N_i^3)^{\frac{1}{4}}$ comes to rescue. We can estimate the derivative nonlinearity by the above argument for $s \geq - \frac{\alpha}{2} + \varepsilon$ as a consequence of the bilinear Strichartz estimate and frequency-dependent time localization.

\medskip

Introducing frequency-dependent time localization to mitigate the derivative nonlinearity necessitates also to prove energy estimates with frequency-dependent time localization.
Regarding short-time energy estimates, we use essentially the same modulation localized estimates like for the short-time bilinear estimates. When carrying out the estimates for solutions, we take advantage of the symmetries of the equation, i.e., real-valuedness and conservative derivative nonlinearity, to assign the derivative always to the low frequency. We obtain several conditions on the regularity when estimating solutions as outlined above. Choosing $0< \alpha< \frac{1}{14}$, we can fit all constraints under the umbrella
\begin{equation*}
    s \geq - \frac{\alpha}{2} + \varepsilon.
\end{equation*}

\subsection{Estimating differences of solutions}

When estimating the energy for differences of solutions, we are lacking the symmetry which allows us to integrate by parts. Let $v=u_1-u_2$ denote the difference of two solutions. A particular case we have to estimate is the High$\times$Low-interaction, $N_2 \lesssim N_1^{1-\kappa} \lesssim N_1 \sim N_3$ given by
\begin{equation}
  \iint P_{N_1} v \partial_x ( P_{N_2} v P_{N_3} u_i ) dx dt.
\end{equation}
Relying again on the bilinear estimate, this interaction can only be controlled estimating the difference of the solution at lower regularity than the solution itself. Using the estimate from \eqref{eq:SketchBilinear} we are led to the condition, taking into account frequency-dependent time localization and the derivative loss,
\begin{equation*}
    N_1^{2s'} N_1^{1+\alpha} (N_1^2 N_2)^{-\frac{1}{2}} N_2^{\frac{1}{2}} N_1^{-\frac{\alpha}{2}} \lesssim N_1^{s'} N_2^{s'} N_1^s.
\end{equation*}
This gives the constraint
\begin{equation*}
    s' \leq \frac{s- \alpha/2}{\kappa}.
\end{equation*}
This suggests to choose $\kappa$ as large as possible to take advantage of the frequency difference as much as possible.
But when considering the case $N_2 \gtrsim N_1^{1-\kappa}$ applying two $L^4$-Strichartz estimates leads us to the condition
\begin{equation*}
    N_1^{2s'} N_1^{1+\alpha} (N_1^2 N_2)^{-\frac{1}{2}} N_2^{\frac{1}{8}+\varepsilon} N_1^{\frac{1}{8}+ \varepsilon} \lesssim N_1^{s'} N_2^{s'} N_1^s.
\end{equation*}
This condition shows that $\kappa < 2/3$ is admissible for an estimate at negative regularities. But the regularity increases choosing $\kappa$ large, so for this estimate a small $\kappa$ is more desirable. We need to compromise these two estimates and choose to simplify the expressions $\kappa = \frac{1}{8}$ and $s' = 20s$. Taking all cases into considerations, we find the short-time energy estimate to hold true for $0<\alpha<\frac{1}{36}$ with $s'=20s$ and 
\begin{equation*}
    -\frac{\alpha}{2} + \varepsilon \leq s \leq - \frac{5 \alpha}{12}.
\end{equation*}

When considering the short-time estimate of the nonlinearity for differences of solutions
\begin{equation*}
    \| \partial_x (u_i v) \|_{\mathcal{N}^{s'}(T)} \lesssim T^\delta \| u_i \|_{F^s(T)} \| v \|_{F^{s'}(T)},
\end{equation*}
the High$\times$High $\rightarrow$ Low-case requires particular attention. Here we need to further reduce the time-localization $0<\alpha<\frac{1}{45}$ to be able to prove the short-time nonlinear estimate for $s=-\frac{\alpha}{2}+\varepsilon$ and $s'=20s$.

The fact that the thresholds for time-localizations do not match exactly 
indicates that the interpolation argument is non-optimal. Here it is the principle which matters: we can prove local well-posedness below $L^2$; the optimization of the non-linear expressions is presently not pursued.

\smallskip

\begin{remark}\label{rem:quasi}
    It remains open whether it is possible to solve the KP-II equation on $\T^2$ below $L^2$ via Picard iteration. This seems to be tied to a possible improvement of Bourgain's bilinear Strichartz estimate from Proposition \ref{prop:BilinearStrichartzBourgain} also for large transversality $D \sim N_1$. It is conceivable that bilinear decoupling estimates play a role for further improvement.
It might also make a difference whether the Cauchy problem is posed on a rational or irrational torus.
\end{remark}

\section{Short-time nonlinear estimates}
\label{section:ShorttimeNonlinear}

The purpose of this section is to propagate nonlinear interactions of solutions and differences of solutions in short-time function spaces. We show the following:
\begin{proposition}
\label{prop:ShorttimeBilinearEstimates}
Let $\alpha > 0,$ and $\varepsilon > 0$.
\begin{itemize}
\item[(i)] Let $0 < \alpha < \frac{1}{4}$, and define
\begin{equation*}
s_1(\alpha,\varepsilon) = \max( - \frac{\alpha}{2} + \varepsilon , - \frac{1}{8} + \frac{3 \alpha}{8} + \varepsilon ).
\end{equation*}
Then there is $\delta(\varepsilon) > 0$ such that the following estimate holds for all $T \in (0,1]$ and $0>s \geq s_1(\alpha,\varepsilon)$:
\begin{equation}
\label{eq:ShorttimeBilinearSolutions}
\| \partial_x (u_1 u_2) \|_{\mathcal{N}^s(T)} \lesssim T^\delta \| u_1 \|_{F^s(T)} \| u_2 \|_{F^s(T)}.
\end{equation}
\item[(ii)] Let $0<\alpha<\frac{1}{7}$, and define
\begin{equation*}
s_2(\alpha,\varepsilon) = \max( - \frac{\alpha}{2} + \varepsilon , \frac{3 \alpha - 1}{84} + \varepsilon ).
\end{equation*}
Let $0 > s \geq s_2(\alpha,\varepsilon)$. Then it holds for $s'=20s$:
\begin{equation}
\label{eq:ShorttimeBilinearDifferences}
\| \partial_x (u v) \|_{\mathcal{N}^{s'}(T)} \lesssim T^\delta \| u \|_{F^s(T)} \| v \|_{F^{s'}(T)}.
\end{equation}
\end{itemize}
\end{proposition}

\begin{remark}
We note that choosing $\alpha$ small enough, we can fit the conditions under one umbrella.
Indeed, choosing $0 < \alpha < \frac{1}{7}$, we find that
\begin{equation*}
s_1(\alpha,\varepsilon) = - \frac{\alpha}{2} + \varepsilon.
\end{equation*}
Moreover, for $0<\alpha<\frac{1}{45}$, we obtain
\begin{equation*}
s_2(\alpha,\varepsilon) = - \frac{\alpha}{2} + \varepsilon.
\end{equation*}
\end{remark}

\begin{proof}
First, we turn to the proof of (i).
The claim follows from dyadic summation of frequency localized estimates:
\begin{equation}
\label{eq:ShorttimeFrequencyLocalizedEstimate}
\| P_N \partial_x (P_{N_1} u_1 P_{N_2} u_2) \|_{\mathcal{N}_N(T)} \lesssim C(N,N_1,N_2) T^{\delta} \prod_{i=1}^2 \| P_{N_i} u_i \|_{F_{N_i}(T)}.
\end{equation}
The constant $C(N,N_1,N_2)$ has to take into account the derivative loss, which gives a factor of $N$, and possibly increasing the time localization in the High$\times$High $\rightarrow$ Low-interaction $N_1 \sim N_2 \gg N$. The backbone of \eqref{eq:ShorttimeBilinearSolutions} is formed by estimates
\begin{equation*}
    L^{-\frac{1}{2}} \| 1_{D_{N,L}} (f_{1,N_1,L_1} * f_{2,N_2,L_2}) \|_{L^2_{\tau,\xi,\eta}} \lesssim C'(N,N_1,N_2) L_{\max}^{0-} \prod_{i=1}^2 L^{\frac{1}{2}} \| f_{i,N_i,L_i} \|_{L^2_{\tau,\xi,\eta}},
\end{equation*}
or a variant with modulation weights $(1+L/N^3)^{\frac{1}{4}}$, or $(1+L_i/N_i^3)^{\frac{1}{4}}$. Above $\operatorname{supp}(f_{i,N_i,L_i}) \subseteq D_{N_i,L_i}$. The room in the modulation variable comes in handy for the dyadic summation and to apply Lemma \ref{lem:ModulationSlack}, which yields the additional gain of $T^{\delta}$.

\smallskip

We carry out standard reductions to reduce \eqref{eq:ShorttimeFrequencyLocalizedEstimate} to estimates localized in modulation. In the following we suppose that $N_1 \geq N_2$ by symmetry.\\
\textbf{High-Low- / High-High-interaction}: $N \sim N_1 \gtrsim N_2$. Consider extensions $P_{N_i} \tilde{u}_i$ of $P_{N_i} u_i$ such that
\begin{equation*}
\| P_{N_i} \tilde{u}_i \|_{F_{N_i}} \leq 2 \| P_{N_i} u_i \|_{F_{N_i}(T)}.
\end{equation*}
To lighten the notation, we redenote $P_{N_i} \tilde{u}_i $ as $u_{i,N_i}$. By the definition of the $\mathcal{N}_N$-norm we bound the left hand side of \eqref{eq:ShorttimeFrequencyLocalizedEstimate} by
\begin{equation*}
\sup_{t_N \in \R} \| (\tau - \omega(\xi,\eta) + i N^\alpha)^{-1} \xi \, 1_{A_N}(\xi) \, (f_{1,N_1} * f_{2,N_2} ) \|_{X_N}
\end{equation*}
with
\begin{equation*}
f_{i,N_i} = \mathcal{F}_{t,x,y}[u_{i,N_i} \cdot \eta_0(N^\alpha(t-t_N))].
\end{equation*}
For $L_i \geq N^\alpha$ we define
\begin{equation*}
f_{i,N_i,L_i} = 
\begin{cases}
1_{D_{N_i,\leq L_i}} f_{i,N_i}, &\quad L_i = N^\alpha, \\
1_{D_{N_i,L_i}} f_{i,N_i}, &\quad L_i > N^\alpha.
\end{cases}
\end{equation*}
By the function space properties it suffices to obtain an estimate
\begin{equation*}
\begin{split}
&\quad \sum_{L \geq N^\alpha} L^{-\frac{1}{2}} (1+L/N^3)^{\frac{1}{4}} N \| 1_{D_{N,L}} (f_{1,N_1,L_1} * f_{2,N_2,L_2}) \|_{L^2_{\tau,\xi,\eta}} \\
&\lesssim C(N,N_1,N_2) \prod_{i=1}^2 L_i^{\frac{1}{2}} (1+L_i / N_i^3)^{\frac{1}{4}} \| f_{i,N_i,L_i} \|_{L^2_{\tau,\xi,\eta}}.
\end{split}
\end{equation*}
We suppose in the following without loss of generality that $f_{i,N_i,L_i}: \R \times \Z^2 \to \R_{\geq 0}$.
Denote $L_{\max} = \max(L,L_1,L_2)$, $L_{\text{med}} = \max(\{L,L_1,L_2\} \backslash L_{\max} )$, and $L_{\min} = \min(L,L_1,L_2)$.
We remark that we always have the following estimate as a consequence of Proposition \ref{prop:SecondOrderBilinearEstimate}:
\begin{equation*}
\| 1_{D_{N,L}} (f_{1,N_1,L_1} * f_{2,N_2,L_2}) \|_{L^2_{\tau,\xi,\eta}} \leq (L_1 \wedge L_2)^{\frac{1}{2}} N_2^{\frac{3}{4}} (L_1 \vee L_2)^{\frac{1}{4}} \prod_{i=1}^2 \| f_{i,N_i,L_i} \|_2.
\end{equation*}
Interpolation of this estimate with trilinear estimates
\begin{equation}
\label{eq:InterpolationEstimate}
\begin{split}
&\quad L^{-\frac{1}{2}} (1+L / N^3)^{\frac{1}{4}} \| 1_{D_{N,L}} (f_{1,N_1,L_1} * f_{2,N_2,L_2}) \|_{L^2_{\tau,\xi,\eta}} \\
&\lesssim C(N,N_1,N_2)  \prod_{i=1}^2 L_i^{\frac{1}{2}} (1+L_i/N_i^3)^{\frac{1}{4}} \| f_{i,N_i,L_i} \|_2.
\end{split}
\end{equation}
allows us to gain a factor $L_{\max}^{0-}$ in the above display at the cost of $N_1^{0+}$, which is always admissible. This allows us to apply Lemma \ref{lem:ModulationSlack} to obtain the factors $T^\delta$ in \eqref{eq:ShorttimeBilinearSolutions} and \eqref{eq:ShorttimeBilinearDifferences}. In the following we focus on establishing \eqref{eq:InterpolationEstimate}.


\medskip

\textbf{High-High-interaction:} $N_1 \sim N_2 \sim N$. 

\smallskip

$\bullet$ \emph{Resonant subcase:} Suppose that $L_{\max} \sim N^3 \gg L_{\text{med}}$, which necessitates a bound for the transversality
\begin{equation*}
D = \big| \frac{\eta_1}{\xi_1} - \frac{\eta_2}{\xi_2} \big| \lesssim N.
\end{equation*}
Moreover, in this case the modulation weights $1+L_i/N_i^3$ are comparable to $1$. We carry out a decomposition of $\eta_i/\xi_i$ into intervals $I_i$ of length $N$:
\begin{equation}
\label{eq:AlmostOrthogonalDetailsI}
    f_{1,N_1,L_1} * f_{2,N_2,L_2} = \sum_{I_1, I_2} f^{I_1}_{1,N_1,L_1} * f^{I_2}_{2,N_2,L_2}
\end{equation}
with 
\begin{equation*}
f^{I_i}_{i,N_i,L_i}(\tau,\xi,\eta) = \begin{cases}
f_{i,N_i,L_i}(\tau,\xi,\eta), \quad &\eta_i/\xi_i \in I_i, \\
0, \quad &\text{else}.
\end{cases}
\end{equation*}
By the bound on $D$ we can restrict the sum over neighboring intervals
\begin{equation}
\label{eq:AlmostOrthogonalDetailsII}    \sum_{I_1, I_2} f^{I_1}_{1,N_1,L_1} * f^{I_2}_{2,N_2,L_2} = \sum_{I_1 \sim I_2} f^{I_1}_{1,N_1,L_1} * f^{I_2}_{2,N_2,L_2}.
\end{equation}

When we consider the convolution
\begin{equation*}
f^{I_1}_{1,N_1,L_1} * f^{I_2}_{2,N_2,L_2},
\end{equation*}
we can suppose that $\big| \frac{\eta_i}{\xi_i} \big| \lesssim N$ by the Galilean invariance $\eta \to \eta - A \xi$. This yields the localization $|\eta_i| \lesssim N^2$. 

We suppose in the following by symmetry that $L = L_{\max}$. The symmetry in $L_i$ is not perfect, but another choice only causes a logarithmic factor by summing over $L \lesssim L_{\max}$. Choosing $\varepsilon>0$ in $s_1(\alpha,\varepsilon)$ takes care of logarithmic factors and further allows for estimates with modulation $L_i^{\frac{1}{2}-}$.

\medskip

 After plugging in the almost orthogonal decomposition we can apply H\"older's inequality together with the $L^4$-Strichartz estimate provided by Theorem \ref{thm:ImprovedL4Strichartz} and the transfer principle Proposition \ref{prop:TransferPrinciple} to find:
\begin{equation}
\label{eq:HighHighL4Strichartz}
\begin{split}
L^{-\frac{1}{2}} N \| 1_{D_{N,L}} (f^{I_1}_{1,N_1,L_1} * f^{I_2}_{2,N_2,L_2} ) \|_{L^2_{\tau,\xi,\eta}} &\lesssim N^{-\frac{1}{2}} \prod_{i=1}^2 \| \mathcal{F}^{-1}_{t,x,y} [ f^{I_i}_{i,N_i,L_i}] \|_{L^4_{t,x,y}} \\
&\lesssim_\varepsilon N^{-\frac{1}{2}} (N^{\frac{1}{8}-\frac{\alpha}{8}+\frac{\alpha}{12}+\varepsilon} )^2 \prod_{i=1}^2 L_i^{\frac{1}{2}} \| f^{I_i}_{i,N_i,L_i} \|_{L^2} \\
&\lesssim_\varepsilon N^{-\frac{1}{4}-\frac{\alpha}{12}+\varepsilon} \prod_{i=1}^2 L_i^{\frac{1}{2}} \| f^{I_i}_{i,N_i,L_i} \|_{L^2}.
\end{split}
\end{equation}
By carrying out the sum over $I_i$, which incurs no loss by \eqref{eq:AlmostOrthogonalDetailsI} and \eqref{eq:AlmostOrthogonalDetailsII}, we find the estimate
\begin{equation*}
    L^{-\frac{1}{2}} N \| 1_{D_{N,L}} (f_{1,N_1,L_1} * f_{2,N_2,L_2} ) \|_{L^2_{\tau,\xi,\eta}} \lesssim_\varepsilon N^{-\frac{1}{4}-\frac{\alpha}{12}+\varepsilon} \prod_{i=1}^2 L_i^{\frac{1}{2}} \| f_{i,N_i,L_i} \|_{L^2}. 
\end{equation*}

This estimate gives 
\begin{equation}
\label{eq:ShorttimeBilinearRegularityI}
s_1(\alpha,\varepsilon) \geq -\frac{1}{4}-\frac{\alpha}{12}+\varepsilon.
\end{equation}

We remark that there is an alternative estimate using bilinear Strichartz estimates with large transversality for
\begin{equation*}
\big| \frac{\eta_1}{\xi_1} - \frac{\eta_2}{\xi_2} \big| \sim N.
\end{equation*}
We find supposing that $L_2 \geq L_1$ to ease notation:
\begin{equation*}
N L^{-\frac{1}{2}} \| 1_{D_{N,L}} (f_{1,N_1,L_1} * f_{2,N_2,L_2}) \|_{L^2_{\tau,\xi,\eta}} \lesssim N^{\frac{1}{2}} L_1^{\frac{1}{2}} \langle L_2 / N \rangle^{\frac{1}{2}} \prod_{i=1}^2 \| f_{i,N_i,L_i} \|_{L^2_{\tau,\xi,\eta}}.
\end{equation*}
With $L_2 \gtrsim N^\alpha$ we obtain
\begin{equation*}
N L^{-\frac{1}{2}} \| 1_{D_{N,L}} (f_{1,N_1,L_1} * f_{2,N_2,L_2}) \|_{L^2_{\tau,\xi,\eta}} \lesssim N^{-\frac{\alpha}{2}} \prod_{i=1}^2 L_i^{\frac{1}{2}} \| f_{i,N_i,L_i} \|_{L^2}.
\end{equation*}
This gives 
\begin{equation*}
s_1(\alpha,\varepsilon) \geq -\frac{\alpha}{2}+\varepsilon.
\end{equation*}
We shall see that this condition comes up in the majority of the estimates as consequence of bilinear Strichartz estimates with minimum modulation. However, in this specific case it is inferior to \eqref{eq:ShorttimeBilinearRegularityI}.


Note moreover that the estimate \eqref{eq:HighHighL4Strichartz} improves for $|\eta_i| \ll N^2$ by invoking Theorem \ref{thm:ImprovedL4Strichartz}, which corresponds to $D \ll N$ after invoking Galilean invariance and an almost orthogonal decomposition.

\medskip

$\bullet$ \emph{Non-resonant subcase:} $L_{\max} \sim M^2 N_1$ for $M \gg N_1$ and $L_{\text{med}} \ll L_{\max}$. Suppose by symmetry that $L = L_{\max}$. Clearly, we have the transversality localization
\begin{equation*}
D \sim \big| \frac{\eta_1}{\xi_1} - \frac{\eta - \eta_1}{\xi - \xi_1} \big| \sim M.
\end{equation*}
By the arguments from above, namely almost orthogonal decomposition in $\eta_i/ \xi_i$ and Galilean invariance, we can localize $|\eta_i| \lesssim MN$. Indeed, we first carry out an almost orthogonal decomposition of $\eta_i / \xi_i$ as follows:
\begin{equation*}
\| 1_{D_{N,L}} (f_{1,N_1,L_1} * f_{2,N_2,L_2}) \|_{L^2_{\tau,\xi,\eta}} \lesssim \sum_{M_1 \sim M_2} \| 1_{D_{N,L}} (f^{M_1}_{1,N_1,L_1} * f^{M_2}_{2,N_2,L_2}) \|_{L^2_{\tau,\xi,\eta}},
\end{equation*}
where $\operatorname{supp}_{\xi,\eta}(f^{M_i}_{i,N_i,L_i}) \subseteq \{ (\xi,\eta) \in \R^2 : \eta_i / \xi_i \in I_{M_i} \}$ with $|I_{M_i}| \sim M$. Secondly, by Galilean invariance, we can suppose that $I_{M_i}$ is centered at the origin. Then we find $|\eta_i| \lesssim MN$ in the support of $f_i$.

Next, we divide the support $\eta_i \in [-3 MN, 3 MN]$ for $(\tau_i, \xi_i,\eta_i) \in \operatorname{supp}(f_i)$ into intervals $J_i= [k_i N^2, (k_i+1) N^2]$ of length $N^2$, $|k_i| \lesssim M/N$. This incurs a factor of $(M/N)^{\frac{1}{2}}$. Moreover, to apply Proposition \ref{prop:L4StrichartzShifted}, we further decompose the $\xi$-support into intervals of length $N^2/M$, which incurs another factor of $(M/N)^{\frac{1}{2}}$. The resulting decomposition reads
\begin{equation*}
\sum_{M_1 \sim M_2} \| 1_{D_{N,L}} (f^{M_1}_{1,N_1,L_1} * f^{M_2}_{2,N_2,L_2}) \|_{L^2_{\tau,\xi,\eta}} \lesssim \sum_{M_1 \sim M_2} \sum_{I_1,I_2} \| 1_{D_{N,L}} (f^{M_1}_{1,N_1,L_1} * f^{M_2}_{2,N_2,L_2}) \|_{L^2_{\tau,\xi,\eta}}.
\end{equation*}

We estimate by Proposition \ref{prop:L4StrichartzShifted}
\begin{equation*}
\begin{split}
&\quad \| 1_{D_{N,L}} (f_{1,N_1,L_1}^{I_1} * f_{2,N_2,L_2}^{I_2} ) \|_{L^2_{\tau,\xi,\eta}} \\
&\leq \| \mathcal{F}^{-1} [f^{I_1}_{1,N_1,L_1} ] \|_{L^4_{t,x,y}(\R^3)} \| \mathcal{F}^{-1} [f^{I_2}_{2,N_2,L_2}] \|_{L^4_{t,x,y}(\R^3)} \\
&\lesssim_\varepsilon \big( N^{\frac{1}{8}+\varepsilon} \big)^2 \big( \frac{M}{N} \big)^{\frac{1}{6}} \prod_{i=1}^2 L_i^{\frac{1}{2}} \| f_{i,N_i,L_i}^{I_i} \|_{L^2}.
\end{split}
\end{equation*}
Collecting the factors we find
\begin{equation*}
\begin{split}
&\quad L^{-\frac{1}{2}} (1+L/N^3)^{\frac{1}{4}} \| 1_{D_{N,L}} (f_{1,N_1,L_1} * f_{2,N_2,L_2} ) \|_{L^2_{\tau,\xi,\eta}} \\
&\lesssim_\varepsilon \big( \frac{1}{M^2 N} \big)^{\frac{1}{4}} N^{-\frac{3}{4}} \big( \frac{M}{N} \big) \big( \frac{M}{N} \big)^{\frac{1}{6}} N^{\frac{1}{4}+\varepsilon} \prod_{i=1}^2 L_i^{\frac{1}{2}} \| f_{i,N_i,L_i} \|_2.
\end{split}
\end{equation*}
We impose a bound on $M$: $N \ll M \lesssim N^{1+a_1}$. Carrying out the sum over $M$, we obtain
\begin{equation*}
L^{-\frac{1}{2}} (1+L/N^3)^{\frac{1}{4}} \| 1_{D_{N,L}} (f_{1,N_1,L_1} * f_{2,N_2,L_2}) \|_{L^2_{\tau,\xi,\eta}} \lesssim_\varepsilon N^{-\frac{3}{2}} N^{\frac{a_1}{2}+ \frac{a_1}{6}} N^{\frac{1}{4}+\varepsilon} \prod_{i=1}^2 L_i^{\frac{1}{2}} \| f_i \|_{L^2}.
\end{equation*}
This gives a regularity constraint taking into account the derivative loss
\begin{equation}
\label{eq:ShorttimeBilinearRegularityII}
s_1(\alpha,\varepsilon) \geq -\frac{1}{4}+\frac{2 a_1}{3} + \varepsilon.
\end{equation}

For $M \gtrsim N^{1+a_1}$ we can use a bilinear Strichartz estimate from Proposition \ref{prop:SimplifiedBilinearStrichartz} with $D \sim M$ and $L_i \gtrsim N^\alpha$ to find
\begin{equation*}
\begin{split}
&\quad L^{-\frac{1}{2}} (1+L/N^3)^{\frac{1}{4}} \| 1_{D_{N,L}} (f_{1,N_1,L_1} * f_{2,N_2,L_2}) \|_{L^2_{\tau,\xi,\eta}} \\ &\lesssim (N^{2(1+a_1)} N)^{-\frac{1}{4}} N^{-\frac{3}{4}} N^{\frac{1}{2}-\frac{\alpha}{2}} \prod_{i=1}^2 L_i^{\frac{1}{2}} \| f_{i,N_i,L_i} \|_{L^2_{\tau,\xi,\eta}}.
\end{split}
\end{equation*}
This gives a regularity threshold
\begin{equation}
\label{eq:ShorttimeBilinearRegularityIII}
s_1(\alpha,\varepsilon) \geq -\frac{a_1}{2} - \frac{\alpha}{2}+\varepsilon.
\end{equation}

$\bullet$ \emph{Strongly non-resonant case:} $L_{\max} \sim L_{\text{med}} \gtrsim N^3$. Suppose that $L \ll L_{\max}$. In this case we apply duality and Proposition \ref{prop:SecondOrderBilinearEstimate} to find
\begin{equation*}
\| 1_{D_{N,L}} (f_{1,N_1,L_1} * f_{2,N_2,L_2}) \|_{L^2_{\tau,\xi,\eta}} \lesssim L^{\frac{1}{2}} N^{\frac{3}{4}} L_1^{\frac{1}{4}} \prod_{i=1}^2 \| f_{i,N_i,L_i} \|_2.
\end{equation*}
This gives
\begin{equation*}
\begin{split}
&\quad \sum_{L \ll L_1} L^{-\frac{1}{2}} (1+L/N^3)^{\frac{1}{4}} \| 1_{D_{N,L}}(f_{1,N_1,L_1} * f_{2,N_2,L_2}) \|_{L^2_{\tau,\xi,\eta}} \\
&\lesssim \sum_{L \ll L_1} (1+L/N^3)^{\frac{1}{4}} N^{\frac{3}{4}} L_1^{\frac{1}{4}} N^{-\frac{3}{2}} L_2^{\frac{1}{2}} \prod_{i=1}^2 \| f_{i,N_i,L_i} \|_2 \\
&\lesssim \log(N) \big( \frac{L_1}{N^3} \big)^{\frac{1}{4}} L_1^{\frac{1}{2}} N^{-\frac{3}{2}} L_2^{\frac{1}{2}} \prod_{i=1}^2 \| f_{i,N_i,L_i} \|_2.
\end{split}
\end{equation*}
This covers the full subcritical range
\begin{equation*}
s_1(\alpha,\varepsilon) \geq - \frac{1}{2} + \varepsilon.
\end{equation*}
In case $L \sim L_{\max}$ we obtain again from Proposition \ref{prop:SecondOrderBilinearEstimate} the estimate
\begin{equation*}
\begin{split}
&\quad \sum_{L \geq N^3} (1+L/N^3)^{\frac{1}{4}} L^{-\frac{1}{2}} \| 1_{D_{N,L}} (f_{1,N_1,L_1}* f_{2,N_2,L_2}) \|_{L^2_{\tau,\xi,\eta}} \\
&\lesssim N^{-\frac{3}{2}} N^{\frac{3}{4}} L_1^{\frac{1}{2}} L_2^{\frac{1}{4}} \prod_{i=1}^2 \| f_{i,N_i,L_i} \|_2 \\
&\lesssim N^{-\frac{3}{2}} \prod_{i=1}^2 L_i^{\frac{1}{2}} \| f_{i,N_i,L_i} \|_2.
\end{split}
\end{equation*}
This covers again the full subcritical range.

\medskip

\textbf{High-Low-High-interaction:} Suppose $N_2 \ll N_1 \sim N$. By the above reductions it suffices to show estimates:
\begin{equation*}
\begin{split}
&\quad L^{-\frac{1}{2}} (1+L/N^3)^{\frac{1}{4}} \| 1_{D_{N,L}} (f_{1,N_1,L_1} * f_{2,N_2,L_2}) \|_{L^2_{\tau,\xi,\eta}}  \\ &\lesssim C(N,N_1,N_2) \prod_{i=1}^2 L_i^{\frac{1}{2}} (1+L_i/N_i^3)^{\frac{1}{4}} \| f_{i,N_i,L_i} \|_{L^2_{\tau,\xi,\eta}}
\end{split}
\end{equation*}
for $L,L_1 \gtrsim N^\alpha$ and $L_2 \geq \min(N_2^3,N_1^{\alpha})$. In case $N_2 \ll N_1$ we use the weight $(1+L_2/N_2^3)^{\frac{1}{4}}$ for the function $f_{2,N_2,L_2}$. For applicabiliy of Lemma \ref{lem:WeightedXNEstimate} we suppose that $L_2 \geq \min(N_2^3,N_1^{\alpha})$. The essence of Lemma \ref{lem:WeightedXNEstimate} is that we must not use the modulation weight for time localization $N_1^{\alpha} \ll N_2^{-3}$, and since the argument depends on using the modulation weight, we localize time at most to intervals $N_2^{-3}$ for $u_2$.

\smallskip

$\bullet$ \emph{Resonant case:} $L_{\max} \sim N_1^2 N_2 \gg L_{\text{med}}$. Suppose that $L=L_{\max}$ (this case is up to a logarithmic factor symmetric to $L_1 = L_{\max}$). In this case we can omit the weight $(1+L/N^3)$. We carry out a dyadic decomposition (Whitney decomposition) in the transversality: $1 \leq D \leq N_1$. This causes another logarithmic factor in $N$. For $D \in 2^{\N_0} \cap [1,N_1]$ we divide $(\eta_i/\xi_i)$ into intervals of length $D$. For $D=1$ we suppose that
\begin{equation*}
\big| \frac{\eta_1}{\xi_1} - \frac{\eta_2}{\xi_2} \big| \lesssim 1,
\end{equation*}
whereas for $D > 1$ we suppose
\begin{equation*}
\big| \frac{\eta_1}{\xi_1} - \frac{\eta_2}{\xi_2} \big| \sim D.
\end{equation*}
For fixed $D$ we can decompose $\eta_i/\xi_i$ by almost orthogonality into intervals of length $D$, which we can suppose to be centered at the origin by Galilean invariance. Then we find
\begin{equation*}
\big| \frac{\eta_i}{\xi_i} \big| \lesssim D \Rightarrow |\eta_i| \lesssim D N_i.
\end{equation*}
Since $|\eta_2| \lesssim D N_2$, we can carry out another almost orthogonality to decompose $\eta_i$ into intervals of length $D N_2$. Secondly, we can trivially decompose the $\xi$-frequency support into intervals of length $N_2$ by almost orthogonality and convolution constraint.

After the frequency localizations we have several possibilities to estimate the bilinear expression 
\begin{equation*}
\| 1_{D_{N,L}} (f_{1,N_1,L_1} * f_{2,N_2,L_2} ) \|_{L^2_{\tau,\xi,\eta}}
\end{equation*}
depending on $1 \leq D \leq N_1$ and the ratio $N_2 / N_1 \ll 1$. The estimate via two $L^4$-Strichartz estimates is carried out in the next section. Presently, the bilinear Strichartz estimate from Proposition \ref{prop:bilinear-sp} suffices.

Since $\big| \frac{\eta_1}{\xi_1} - \frac{\eta_2}{\xi_2} \big| \lesssim N_1$, applying Proposition \ref{prop:bilinear-sp} gives
\begin{equation*}
\begin{split}
&\quad L^{-\frac{1}{2}} (1+L/N^3)^{\frac{1}{4}} \| 1_{D_{N,L}} (f_{1,N_1,L_1} * f_{2,N_2,L_2})\|_{L^2_{\tau,\xi,\eta}} \\
&\lesssim (N_1^2 N_2)^{-\frac{1}{2}} \log(N_1) (N_2^{\frac{1}{4}} N_1^{-\frac{\alpha}{4}} + N_2^{\frac{1}{2}} N_1^{-\frac{\alpha}{2}}) \prod_{i=1}^2 L_i^{\frac{1}{2}} \| f_{i,N_i,L_i} \|_2.
\end{split}
\end{equation*}
This gives the regularity condition
\begin{equation}
\label{eq:ShorttimeBilinearRegularityVI}
s_1(\alpha,\varepsilon) \geq \max( -\frac{1}{4} - \frac{\alpha}{4} + \varepsilon , - \frac{\alpha}{2} + \varepsilon).
\end{equation}
Indeed, the estimate
\begin{equation*}
\begin{split}
&\quad L^{-\frac{1}{2}} (1+L/N^3)^{\frac{1}{4}} \| 1_{D_{N,L}} (f_{1,N_1,L_1} * f_{2,N_2,L_2})\|_{L^2_{\tau,\xi,\eta}} \\
&\lesssim (N_1^2 N_2)^{-\frac{1}{2}} \log(N_1) N_2^{\frac{1}{4}} N_1^{-\frac{\alpha}{4}} \prod_{i=1}^2 L_i^{\frac{1}{2}} \| f_{i,N_i,L_i} \|_2
\end{split}
\end{equation*}
is summable for regularities $s_1(\alpha,\varepsilon) \geq - \frac{1}{4} - \frac{\alpha}{4} + \varepsilon$. The second summand gives correspondingly rise to the second constraint $s_1(\alpha,\varepsilon) \geq - \frac{\alpha}{2} + \varepsilon$.
Since $0<\alpha<\frac{1}{4}$, the latter value is always dominant.

\medskip

In case $L_2 = L_{\max}$, i.e., the case when the low frequency carries the high modulation, the modulation weight $(1+L_2/N_2^3)^{\frac{1}{4}}$ comes to rescue. Firstly, we use the usual almost orthogonality and Galilean invariance argument to localize $|\eta_i| \lesssim N_1^2$. Then we use duality to 
infer
\begin{equation*}
\| 1_{D_{N,L}}( f_{1,N_1,L_1} * f_{2,N_2,L_2} ) \|_{L^2_{\tau,\xi,\eta}} = \sup_{\| g_{N,L} \|_{L^2} = 1} \big| \iint g_{N,L} (f_{1,N_1,L_1} * f_{2,N_2,L_2} ) d\xi d\eta  d\tau \big|.
\end{equation*}
Now we use H\"older's inequality, Plancherel's theorem and the $L^4_{t,x,y}$-Strichartz estimates \eqref{eq:ShorttimeStrichartzI} due to Theorem \ref{thm:ImprovedL4Strichartz} to find
\begin{equation*}
\begin{split}
&\quad \big| \iint g_{N,L} (f_{1,N_1,L_1} * f_{2,N_2,L_2} ) d\xi d\eta  d\tau \big| \\
 &\lesssim \| f_{2,N_2,L_2} \|_{L^2_{\tau,\xi,\eta}} \| g_{N,L} * f_{1,N_1,L_1} \|_{L^2} \\
&\lesssim \| f_{2,N_2,L_2} \|_{L^2_{\tau,\xi,\eta}} \| \mathcal{F}_{t,x,y}^{-1} [ g_{N,L} ] \|_{L^4_{t,x,y}} \| \mathcal{F}_{t,x,y}^{-1} [ f_{1,N_1,L_1} ] \|_{L^4_{t,x,y}} \\
&\lesssim \| f_{2,N_2,L_2} \|_{L^2_{\tau,\xi,\eta}} (L L_1)^{\frac{1}{2}} N^{2(\frac{1}{8}-\frac{\alpha}{8}+\frac{\alpha}{12}+\varepsilon)} \| g_{N,L} \|_{L^2_{\tau,\xi,\eta}} \| f_{2,N_2,L_2} \|_{L^2_{\tau,\xi,\eta}}.
\end{split}
\end{equation*}
In conclusion we have
\begin{equation*}
\begin{split}
&\quad L^{-\frac{1}{2}} \| 1_{D_{N,L}} (f_{1,N_1,L_1} * f_{2,N_2,L_2}) \|_{L^2_{\tau,\xi,\eta}} \\
 &\lesssim N_1^{2 \big( \frac{1}{8} - \frac{\alpha}{24} + \varepsilon \big)} N_1^{-\frac{3}{2}} \prod_{i=1}^2 L_i^{\frac{1}{2}} (1+L_i / N_i^3)^{\frac{1}{4}} \| f_{i,N_i,L_i} \|_{L^2_{\tau,\xi,\eta}}.
 \end{split}
\end{equation*}
This gives a regularity threshold from taking into account the derivative loss:
\begin{equation}
\label{eq:ShorttimeBilinearRegularityVII}
s_1(\alpha,\varepsilon) \geq - \frac{1}{4} - \frac{\alpha}{12} + \varepsilon.
\end{equation}

$\bullet$ \emph{Non-resonant case:} $\cdot L = L_{\max} \gg N_1^2 N_2$. Firstly, we suppose that $L_i \leq N_i^3$ for $i=1,2$. In this case we can apply a bilinear Strichartz estimate provided by Proposition \ref{prop:SimplifiedBilinearStrichartz} using large transversality $D \gtrsim N_1$ to find
\begin{equation*}
\begin{split}
&\quad \sum_{L \geq N_1^2 N_2} L^{-\frac{1}{2}} (1+L / N_1^3)^{\frac{1}{4}} \| f_{1,N_1,L_1} * f_{2,N_2,L_2} \|_{L^2_{\tau,\xi,\eta}} \\
 &\lesssim (N_1^2 N_2)^{-\frac{1}{2}} N_2^{\frac{1}{2}} N_1^{-\frac{\alpha}{2}} \prod_{i=1}^2 L_i^{\frac{1}{2}} \| f_{i,N_i,L_i} \|_{L^2},
 \end{split}
\end{equation*}
which gives again $s_1(\alpha,\varepsilon) \geq - \frac{\alpha}{2} + \varepsilon$.

Next, suppose that $L_2 \geq N_2^3$ and $L_1 \leq N_1^3$. In this case we can apply a different bilinear Strichartz estimate to find
\begin{equation*}
\begin{split}
L^{-\frac{1}{2}} (1+L/N_1^3)^{\frac{1}{4}} \| f_{1,N_1,L_1} * f_{2,N_2,L_2} \|_{L^2_{\tau,\xi,\eta}} &\lesssim (N_1^2 N_2)^{-\frac{1}{2}} L_1^{\frac{1}{2}} N_2^{\frac{3}{4}} L_2^{\frac{1}{4}} \prod_{i=1}^2 \| f_{i,N_i,L_i} \|_{L^2} \\
&\lesssim N_1^{-1} N_2^{-\frac{1}{2}} \prod_{i=1}^2 L_i^{\frac{1}{2}} \| f_{i,N_i,L_i} \|_{L^2}.
\end{split}
\end{equation*}
This gives a regularity threshold of
\begin{equation}
\label{eq:ShorttimeBilinearRegularityVIII}
s_1(\alpha,\varepsilon) \geq - \frac{1}{2} +\varepsilon.
\end{equation}
The cases $L_1 \geq N_1^3$ and $L_2 \leq N_2^3$ can be handled in the same way and lead to a more favorable estimate.

\smallskip

$\bullet$ \emph{Strongly non-resonant case:} $L_{\max} \sim L_{\text{med}} \gtrsim N_1^2 N_2$. In this case the estimate can be carried out by the bilinear Strichartz estimate provided by Proposition \ref{prop:SecondOrderBilinearEstimate} like in the High-High-High-interaction. Details are omitted to avoid repetition. This gives the regularity threshold
\begin{equation*}
s_1(\alpha,\varepsilon) \geq - \frac{1}{2} + \varepsilon.
\end{equation*}

\medskip

\textbf{Low-High-High-interaction:} In the following we estimate
\begin{equation}
\label{eq:HighHighLowShorttimeEstimate}
\| P_N \partial_x (P_{N_1} u_1 P_{N_2} u_2) \|_{\mathcal{N}_N} \lesssim C(N,N_1,N_2) \| P_{N_1} u_1 \|_{F_{N_1}} \| P_{N_2} u_2 \|_{F_{N_2}}
\end{equation}
for $N \ll N_1 \sim N_2$. 

By the definition of the $\mathcal{N}_N$-norm the left hand-side is bounded by
\begin{equation*}
\sup_{t_N \in \R} \| (\tau - \omega(\xi,\eta) + i N^\alpha)^{-1} \xi 1_{A_N}(\xi) \mathcal{F}_{t,x,y}[ u_{1,N_1} u_{2,N_2} \eta_0(N^\alpha(t-t_N)) ] \|_{X_N}.
\end{equation*}
Since $N \ll N_1$, we cannot bound
\begin{equation*}
\| \mathcal{F}_{t,x,y}[u_{1,N_1} \eta_0(N^\alpha(t-t_N))] \|_{X_{N_1}} \lesssim \| u \|_{F_{N_1}}.
\end{equation*}
For this reason we introduce additional time localization
\begin{equation*}
\eta_0(N^\alpha(t-t_N)) = \sum_{t_{N_1}} \eta_0(N^\alpha(t-t_N)) \gamma^2(N_1^\alpha(t-t_{N_1}))
\end{equation*}
with a suitable bump function $\gamma \in C^\infty_c(-1,1)$.

\smallskip

Clearly, $\# \{ t_{N_1} \} \sim (N_1/N)^\alpha$. So the additional time localization incurs a factor of $(N_1/N)^\alpha$ and we are left with proving estimates
\begin{equation*}
\begin{split}
&\quad \sup_{t_N \in \R} \| (\tau-\omega(\xi,\eta)+i N^\alpha)^{-1} \xi 1_{A_N}(\xi) \mathcal{F}_{t,x,y}[ u_{1,N_1} \gamma(N_1^\alpha(t-t_{N_1})) ] \\
&\quad * \mathcal{F}_{t,x,y}[ u_{2,N_2} \gamma(N_1^\alpha(t-t_{N_1})) \eta_0(N^\alpha(t-t_N)) ] \|_{X_N} \\
&\lesssim C(N,N_1,N_2) \prod_{i=1}^2 \| \mathcal{F}_{t,x,y}[ u_{i,N_i} \gamma(N_1^\alpha(t-t_{N_1})) \eta_0(N^\alpha(t-t_N)) ] \|_{X_{N_i}} 
\end{split}
\end{equation*}
because
\begin{equation*}
\| \mathcal{F}_{t,x,y}[u_{i,N_i} \gamma(N_1^\alpha(t-t_{N_1})) \eta_0(N^\alpha(t-t_N)) ] \|_{X_{N_i}} \lesssim \| u_{i,N_i} \|_{F_{N_i}}.
\end{equation*}

We introduce the notation
\begin{equation*}
f_{1,N_1} = \mathcal{F}_{t,x,y}[u_{1,N_1} \gamma(N_1^\alpha(t-t_{N_1})) ], \; f_{2,N_2} = \mathcal{F}_{t,x,y}[u_{2,N_2} \gamma(N_1^\alpha(t-t_{N_1})) \eta_0(N^{\alpha}(t-t_N)) ]
\end{equation*}
for the space-time Fourier transforms and break its support according to modulation localization:
\begin{equation*}
f_{i,N_i,L_i} = 
\begin{cases}
1_{D_{N_i,L_i}} f_{i,N_i}, \quad &L_i > N_1^\alpha, \\
1_{D_{N_i,\leq L_i}} f_{i,N_i}, \quad &L_i = N_1^\alpha.
\end{cases}
\end{equation*}
Again we suppose by the invariance of the norms that $f_i \geq 0$.
The minimum modulation localization $N_1^\alpha$ reflects the frequency-dependent time localization $T=T(N_1)=N_1^{-\alpha}$. With the above notation we reduce to
\begin{equation*}
\begin{split}
&\quad L^{-\frac{1}{2}} (1+L/N^3)^{\frac{1}{4}} \| f_{1,N_1,L_1} * f_{2,N_2,L_2} \|_{L^2_{\tau,\xi,\eta}} \\
 &\lesssim C(N,N_1,N_2) \prod_{i=1}^2 L_i^{\frac{1}{2}} (1+L_i/N_i^3)^{\frac{1}{4}} \| f_{i,N_i,L_i} \|_{L^2_{\tau,\xi,\eta}}
\end{split}
\end{equation*}
with $L \geq N^\alpha$, from which \eqref{eq:HighHighLowShorttimeEstimate} follows from dyadic summation taking into account the derivative loss $N$ and additional time localization $(N_1/N)^\alpha$.

\medskip

$\bullet$ \emph{Resonant case: }$L_{\max} = N_1^2 N$.\\
 First, we suppose that the high frequency carries the high modulation $L_1 = L_{\max}$. Moreover, we suppose that $L \lesssim N^3$. In this case we use duality and a bilinear Strichartz estimate from Proposition \ref{prop:bilinear-sp}:
 \begin{equation}
 \begin{split}
	\| 1_{D_{N,L}} (f_{1,N_1,L_1} * f_{2,N_2,L_2}) \|_{L^2_{\tau,\xi,\eta}} &= \sup_{\| g_{N,L} \|_{L^2} = 1} \iint g_{N,L} (f_{1,N_1,L_1} * f_{2,N_2,L_2} ) \\
	&\leq \| f_{1,N_1,L_1} \|_{L^2_{\tau,\xi,\eta}} \sup_{\| g_{N,L} \|_2 = 1} \| g_{N,L} * f_{2,N_2,L_2} \|_{L^2_{\tau,\xi,\eta}}.
	\end{split}
 \end{equation}
 Noting that we have the transversality bound
 \begin{equation*}
 \big| \frac{\eta}{\xi} - \frac{\eta_2}{\xi_2} \big| \lesssim N_1,
 \end{equation*}
applying Proposition \ref{prop:bilinear-sp} yields
\begin{equation*}
\| g_{N,L} * f_{2,N_2,L_2} \|_{L^2_{\tau,\xi,\eta}} \lesssim \log(N_1) (N^{\frac{1}{4}} N_1^{-\frac{\alpha}{4}} + N^{\frac{1}{2}} N_1^{-\frac{\alpha}{2}} ) (L L_2)^{\frac{1}{2}} \| g_{N,L} \|_{L^2} \| f_{2,N_2,L_2} \|_{L^2_{\tau,\xi,\eta}}.
\end{equation*}

Consequently,
\begin{equation*}
\begin{split}
&\quad L^{-\frac{1}{2}} \| 1_{D_{N,L}} (f_{1,N_1,L_1} * f_{2,N_2,L_2} ) \|_{L^2_{\tau,\xi,\eta}} \\
&\lesssim \log(N_1) (N^{\frac{1}{4}} N_1^{-\frac{\alpha}{4}} + N^{\frac{1}{2}} N_1^{-\frac{\alpha}{2}} ) (N_1^2 N)^{-\frac{1}{2}} \prod_{i=1}^2 L_i^{\frac{1}{2}} \| f_{i,N_i,L_i} \|_{L^2}.
\end{split}
\end{equation*}
Taking into account the additional time localization, the derivative loss, and the Sobolev regularity of the functions, we find the condition
\begin{equation*}
\big( \frac{N_1}{N} \big)^\alpha N \cdot N^s \log(N_1) (N^{\frac{1}{4}} N_1^{-\frac{\alpha}{4}} + N^{\frac{1}{2}} N_1^{-\frac{\alpha}{2}} ) (N_1^2 N)^{-\frac{1}{2}} \lesssim N_1^{2s}.
\end{equation*}
For $0 < \alpha < \frac{1}{4}$ and $s>-\frac{1}{4}$, this gives the regularity threshold:
\begin{equation}
\label{eq:ShorttimeBilinearRegularityIX}
s_1(\alpha,\varepsilon) \geq - \frac{\alpha}{2} + \varepsilon.
\end{equation}

Next, suppose that $L \gtrsim N^3$. In this case we use duality and the bilinear Strichartz estimate provided by Proposition \ref{prop:SecondOrderBilinearEstimate}:
\begin{equation*}
\begin{split}
&\quad L^{-\frac{1}{4}} N^{-\frac{3}{4}} \| 1_{D_{N,L}} (f_{1,N_1,L_1} * f_{2,N_2,L_2} ) \|_{L^2_{\tau,\xi,\eta}} \\
&\leq L^{-\frac{1}{4}} N^{-\frac{3}{4}} \sup_{\| g_{N,L} \|_2 = 1} \iint g_{N,L} (f_{1,N_1,L_1} * f_{2,N_2,L_2} ) d\xi d\eta d\tau \\
&\lesssim L^{-\frac{1}{4}} N^{-\frac{3}{4}} \| f_{1,N_1,L_1} \|_2 \| g_{N,L} * f_{2,N_2,L_2} \|_{L^2_{\tau,\eta,\eta}} \\
&\lesssim L^{-\frac{1}{4}} N^{-\frac{3}{4}} \| f_{1,N_1,L_1} \|_2 N^{\frac{1}{2}} L_2^{\frac{1}{2}} L^{\frac{1}{4}} N^{\frac{1}{4}} \| g_{N,L} \|_2 \| f_{2,N_2,L_2} \|_2 \\
&\lesssim (N_1^2 N)^{-\frac{1}{2}} \prod_{i=1}^2 L_i^{\frac{1}{2}} \| f_{i,N_i,L_i} \|_2.
\end{split}
\end{equation*}
Taking into account the time localization, the derivative loss, and the Sobolev regularity we find the condition
\begin{equation*}
\big( \frac{N_1}{N} \big)^\alpha N (N_1^2 N)^{-\frac{1}{2}} N^s \lesssim N_1^{2s}.
\end{equation*}
For $0<\alpha< \frac{1}{4}$ and $s>-\frac{1}{4}$ this yields the condition
\begin{equation}
\label{eq:ShorttimeBilinearRegularityX}
s_1(\alpha,\varepsilon) \geq - \frac{1}{2} + \varepsilon.
\end{equation}

\medskip

Next, we turn to the resonant case with the low frequency carrying the high modulation: $ L = L_{\max} =N_1^2 N$. 

\smallskip

In this case we use two $L^4_{t,x,y}$-Strichartz estimates provided by Theorem \ref{thm:ImprovedL4Strichartz}:
\begin{equation*}
L^{-\frac{1}{4}} N^{-\frac{3}{4}} \| f_{1,N_1,L_1} * f_{2,N_2,L_2} \|_{L^2_{\tau,\xi,\eta}} \lesssim N_1^{-\frac{1}{2}} N^{-1} (N_1^{\frac{1}{8}-\frac{\alpha}{8}+\varepsilon})^{2} \prod_{i=1}^2 L_i^{\frac{1}{2}} \| f_{i,N_i,L_i} \|_{L^2_{\tau,\xi,\eta}}.
\end{equation*}
Taking into account time localization, derivative loss, and Sobolev regularity we find the condition
\begin{equation*}
\big( \frac{N_1}{N} \big)^\alpha N N_1^{-\frac{1}{2}} N^{-1} N_1^{\frac{1}{4}-\frac{\alpha}{4}+\varepsilon} N^s \lesssim N_1^{2s}.
\end{equation*}
Since $s<0$ and $\alpha > 0$, this yields
\begin{equation}
\label{eq:ShorttimeBilinearRegularityXI}
N_1^{-\frac{1}{4}+\frac{3 \alpha}{4}} N^{s-\alpha} \lesssim N_1^{2s} \Rightarrow s_1(\alpha,\varepsilon) \geq - \frac{1}{8} + \frac{3 \alpha}{8} + \varepsilon.
\end{equation}
Note that here we can use the improved $L^4_{t,x,y}$-Strichartz estimate because due to the transversality estimate
\begin{equation*}
\big| \frac{\eta_1}{\xi_1} - \frac{\eta_2}{\xi_2} \big| \lesssim N.
\end{equation*}
This implies by the usual almost orthogonality and Galilean invariance argument
\begin{equation*}
|\eta_i| \lesssim N N_1 \ll N_1^2.
\end{equation*}
On the other hand, in this case bilinear Strichartz estimates are less effective due to the lack of frequency separation in case of small transversality.
This case leads to an additional constraint in addition to $s_1(\alpha,\varepsilon) \geq - \frac{\alpha}{2}+ \varepsilon$.

\medskip

$\bullet$ \emph{Non-resonant case:} $L_1 = L_{\max} \gg N_1^2 N$. Firstly, suppose that $L \leq N^3$. Recall that $L_2 \gtrsim N_1^{\alpha}$. We use duality and the bilinear Strichartz estimate provided by Proposition \ref{prop:SimplifiedBilinearStrichartz} with transversality $D \gtrsim N_1$ to find
\begin{equation*}
\begin{split}
&\quad L^{-\frac{1}{2}} \| 1_{D_{N,L}} (f_{1,N_1,L_1} * f_{2,N_2,L_2}) \|_{L^2_{\tau,\xi,\eta}} \\
 &\lesssim L^{-\frac{1}{2}} \iint g_{N,L} (f_{1,N_1,L_1} * f_{2,N_2,L_2}) d \xi d\eta d\tau \\
&\lesssim L^{-\frac{1}{2}} \| f_{1,N_1,L_1} \|_{L^2} \| g_{N,L} * f_{2,N_2,L_2} \|_{L^2_{\tau,\xi,\eta}} \\
&\lesssim L^{-\frac{1}{2}} (N_1^2 N)^{-\frac{1}{2}} L_1^{\frac{1}{2}} N^{\frac{1}{2}} N_1^{-\frac{\alpha}{2}} (L L_2)^{\frac{1}{2}} \| g_{N,L} \|_{L^2} \| f_{2,N_2,L_2} \|_{L^2} \\
&\lesssim N_1^{-1-\frac{\alpha}{2}} \prod_{i=1}^2 L_i^{\frac{1}{2}} \| f_{i,N_i,L_i} \|_{L^2}.
\end{split}
\end{equation*}
Taking into account time localization, derivative loss, and Sobolev regularity we find
\begin{equation*}
\big( \frac{N_1}{N} \big)^\alpha N N_1^{-1-\frac{\alpha}{2}} N^s \lesssim N_1^{2s}.
\end{equation*}
This gives the condition
\begin{equation}
\label{eq:ShorttimeBilinearRegularityXII}
s_1(\alpha,\varepsilon) \geq - \frac{\alpha}{2} + \varepsilon.
\end{equation}

If $L \geq N_1^3$, we use the bilinear Strichartz estimate from Proposition \ref{prop:SimplifiedBilinearStrichartz} on $f_{1,N_1,L_1}$ and $f_{2,N_2,L_2}$ like above:
\begin{equation*}
L^{-\frac{1}{2}} (1+ L/N^3)^{\frac{1}{4}} \| f_{1,N_1,L_1} * f_{2,N_2,L_2} \|_{L^2} \lesssim N_1^{-\frac{3}{4}} N^{-\frac{3}{4}} N^{\frac{1}{2}} N_1^{-\frac{\alpha}{2}} \prod_{i=1}^2 L_i^{\frac{1}{2}} \| f_{i,N_i,L_i} \|_{L^2}.
\end{equation*}
Taking into account time localization, Sobolev regularity, and derivative loss we find the condition
\begin{equation*}
\big( \frac{N_1}{N} \big)^\alpha N^s N^{\frac{3}{4}} N_1^{-\frac{3}{4}} N_1^{-\frac{\alpha}{2}} \lesssim N_1^{2s}.
\end{equation*}
This yields again
\begin{equation}
\label{eq:ShorttimeBilinearRegularityXIII}
N^{s+\frac{3}{4}-\alpha} N_1^{-\frac{3}{4}+\frac{\alpha}{2}} \lesssim N_1^{2s} \Rightarrow s_1(\alpha,\varepsilon) \geq - \frac{\alpha}{2} + \varepsilon.
\end{equation}

$\cdot L = L_{\max} \gg N_1^2 N$: The case $L \geq N_1^3$ can be estimated like in the preceding paragraph. We consider $N_1^2 N \leq L \leq N_1^3$ and estimate the expression
\begin{equation*}
L^{-\frac{1}{4}} N^{-\frac{3}{4}} \| 1_{D_{N,L}} (f_{1,N_1,L_1} * f_{2,N_2,L_2}) \|_{L^2_{\tau,\xi,\eta}}.
\end{equation*}
We apply a bilinear Strichartz estimate due to Proposition \ref{prop:SimplifiedBilinearStrichartz} on $f_{i,N_i,L_i}$ to find
\begin{equation*}
L^{-\frac{1}{4}} N^{-\frac{3}{4}} \| 1_{D_{N,L}} (f_{1,N_1,L_1} * f_{2,N_2,L_2}) \|_{L^2_{\tau,\xi,\eta}} \lesssim N_1^{-\frac{1}{2}} N^{-1} N^{\frac{1}{2}} N_1^{-\frac{\alpha}{2}} \prod_{i=1}^2 L_i^{\frac{1}{2}} \| f_{i,N_i,L_i} \|_{L^2}.
\end{equation*}
Taking into account time localization, derivative loss, and Sobolev regularity we find
\begin{equation*}
\big( \frac{N_1}{N} \big)^{\alpha} N N^s N_1^{-\frac{1}{2}} N^{-1} N^{\frac{1}{2}} N_1^{-\frac{\alpha}{2}} \lesssim N_1^{2s}.
\end{equation*}
This gives again the condition
\begin{equation}
\label{eq:ShorttimeBilinearRegularityXIV}
N^{\frac{1}{2}-\alpha+s} N_1^{\frac{\alpha}{2}-\frac{1}{2}} \lesssim N_1^{2s} \Rightarrow s_1(\alpha,\varepsilon) \geq - \frac{\alpha}{2} + \varepsilon.
\end{equation}

\medskip

\emph{Conclusion of the proof of \eqref{eq:ShorttimeBilinearSolutions}.} We collect the regularity thresholds proved above in the separate cases and recall that we suppose that $0 < \alpha < \frac{1}{4}$.

\smallskip

In the \textbf{High-High-High-interaction} we have obtained the following conditions on the regularity \eqref{eq:ShorttimeBilinearRegularityI}-\eqref{eq:ShorttimeBilinearRegularityIII}:
\begin{equation*}
s_1(\alpha,\varepsilon) \geq - \frac{1}{4} - \frac{\alpha}{12} + \varepsilon, \quad s_1(\alpha,\varepsilon) \geq \max( - \frac{1}{4} + \frac{2 a_1}{3} + \varepsilon , - \frac{a_1}{2} - \frac{\alpha}{2} + \varepsilon).
\end{equation*}

\smallskip

In the \textbf{High-Low-High-interaction} we have obtained the following conditions by  \eqref{eq:ShorttimeBilinearRegularityVI} - \eqref{eq:ShorttimeBilinearRegularityVIII}:
\begin{equation*}
\begin{split}
s_1(\alpha,\varepsilon) &\geq \max( -\frac{\alpha}{2} + \varepsilon , - \frac{1}{4} - \frac{\alpha}{12} + \varepsilon , -\frac{1}{2} + \varepsilon ) = - \frac{\alpha}{2}+ \varepsilon.
\end{split}
\end{equation*}

\smallskip

In the \textbf{Low-High-High-interaction} we have obtained in \eqref{eq:ShorttimeBilinearRegularityIX}-\eqref{eq:ShorttimeBilinearRegularityX}, \eqref{eq:ShorttimeBilinearRegularityXI}, \eqref{eq:ShorttimeBilinearRegularityXII}-\eqref{eq:ShorttimeBilinearRegularityXIV}:
\begin{equation*}
s_1(\alpha,\varepsilon) \geq - \frac{\alpha}{2} + \varepsilon, \quad s_1(\alpha,\varepsilon) \geq - \frac{1}{8} + \frac{3 \alpha}{8} + \varepsilon, \quad s_1(\alpha,\varepsilon) \geq - \frac{\alpha}{2} + \varepsilon.
\end{equation*}

\medskip
In summary we obtain for the choice $a_1 = 0$
\begin{equation*}
s_1(\alpha,\varepsilon) = \max( - \frac{\alpha}{2}+ \varepsilon , - \frac{1}{8} + \frac{3 \alpha}{8} + \varepsilon),
\end{equation*}
which finishes the proof of \eqref{eq:ShorttimeBilinearSolutions}.

\medskip

This is balanced for $\alpha = \frac{1}{7}$, which gives
\begin{equation*}
s_1(\alpha,\varepsilon) \geq - \frac{1}{14} + \varepsilon.
\end{equation*}

\bigskip

We turn to the proof of \eqref{eq:ShorttimeBilinearDifferences}. Note that in the following we assume that $0<\alpha< \frac{1}{7}$. Then we can use the analysis from \eqref{eq:ShorttimeBilinearSolutions} supposing that 
\begin{equation}
\label{eq:s2Start}
s_2(\alpha, \varepsilon) \geq - \frac{\alpha}{2} + \varepsilon.
\end{equation}

\smallskip

\eqref{eq:ShorttimeBilinearDifferences} follows again from dyadic summation of frequency localized estimates:
\begin{equation*}
\| P_N \partial_x (P_{N_1} u P_{N_2} v ) \|_{\mathcal{N}_N} \lesssim C(N,N_1,N_2) \| P_{N_1} u \|_{F_{N_1}(T)} \| P_{N_2} v \|_{F_{N_2}(T)}.
\end{equation*}
However, the summation is different, and we need to check the \textbf{Low-High-High-interaction} for further constraints due to different regularities. Clearly, the \textbf{High-High-High-interaction} does not create an additional condition on the regularity. We briefly comment on the \textbf{High-Low-High-interaction}.

\medskip
\textbf{High-Low-High-interaction:} $N_1 \sim N \gg N_2$: We look into estimates
\begin{equation}
\label{eq:HighLowHighAsymmetricI}
\| P_N \partial_x (P_{N_1} u P_{N_2} v ) \|_{\mathcal{N}_N} \lesssim C(N,N_1,N_2) \| P_{N_1} u \|_{F_{N_1}(T)} \| P_{N_2} v \|_{F_{N_2}(T)},
\end{equation}
and
\begin{equation}
\label{eq:HighLowHighAsymmetricII}
\| P_N \partial_x (P_{N_1} v P_{N_2} u ) \|_{\mathcal{N}_N} \lesssim C(N,N_1,N_2) \| P_{N_1} v \|_{F_{N_1}(T)} \| P_{N_2} u \|_{F_{N_2}(T)}.
\end{equation}
Note that for \eqref{eq:ShorttimeBilinearDifferences} these estimates are not symmetric. But we can argue that the previously obtained constants $C(N,N_1,N_2)$ work likewise for summation with shifted regularities. Indeed, the summation condition from above reads
\begin{equation*}
N^s C(N,N_1,N_2) \lesssim N_1^s N_2^s.
\end{equation*}
Since $N \sim N_1$, this yields
\begin{equation}
\label{eq:SummationConditionHighLowHigh}
C(N,N_1,N_2) \lesssim N_2^s.
\end{equation}
For \eqref{eq:HighLowHighAsymmetricII} the summation condition reads
\begin{equation*}
N^{20s} C(N,N_1,N_2) \lesssim N_1^{20s} N_2^{s}.
\end{equation*}
This is clearly implied by $N_1 \sim N$ and \eqref{eq:SummationConditionHighLowHigh}.

Turning to \eqref{eq:HighLowHighAsymmetricI} the summation condition reads
\begin{equation*}
N^{20s} C(N,N_1,N_2) \lesssim N_1^s N_2^{20s},
\end{equation*} 
which is verified by
\begin{equation*}
N^s C(N,N_1,N_2) N^{19s} \lesssim N_1^s N_2^s N^{19s} \lesssim N_1^s N_2^{20s}.
\end{equation*}
The first inequality is \eqref{eq:SummationConditionHighLowHigh} and $N_1 \sim N$, the second inequality follows from $N \gg N_2$ and $s<0$.

\medskip

\textbf{Low-High-High-interaction:} Here we consider the estimate for $N \ll N_1 \sim N_2$:
\begin{equation*}
\| P_N \partial_x (P_{N_1} u P_{N_2} v) \|_{\mathcal{N}_N} \lesssim C(N,N_1,N_2) \| P_{N_1} u \|_{F_{N_1}} \| P_{N_2} v \|_{F_{N_2}}.
\end{equation*}
The constant $C(N,N_1,N_2)$ has been found above. Since we aim to prove \eqref{eq:ShorttimeBilinearDifferences}, the unbalanced regularity of the input functions changes the outcome.

\smallskip

\emph{Resonant case (i):} $L_{\max} = N_1^2 N$. $L_1 = L_{\max}$,  $L \lesssim N^3$: We obtain the condition taking into account time localization, regularity of the functions, and derivative loss:
\begin{equation*}
\big( \frac{N_1}{N} \big)^\alpha N N^{20s} \log(N_1) ( N^{\frac{1}{2}} N_1^{-\frac{\alpha}{2}} + N^{\frac{1}{4}} N_1^{-\frac{\alpha}{4}}) (N_1^2 N)^{-\frac{1}{2}} \lesssim N_1^{21 s}.
\end{equation*}
Up to logarithmic factors, this is implied by
\begin{equation*}
N^{20s-\alpha +1} \lesssim N_1^{21s+1-\frac{\alpha}{2}} \text{ and } N^{20s+\frac{3}{4}-\alpha} \lesssim N_1^{1+21s-\frac{3\alpha}{4}}.
\end{equation*}
The first condition holds true for $20 s_2(\alpha,\varepsilon) - \alpha + 1 \geq 0$ together with \eqref{eq:s2Start}. This gives the additional condition:
\begin{equation*}
s_2(\alpha,\varepsilon) \geq \frac{\alpha-1}{20} + \varepsilon.
\end{equation*}
The second condition holds true for $20 s_2(\alpha,\varepsilon) + \frac{3}{4} - \alpha \geq 0$ together with $s_2(\alpha,\varepsilon) \geq - \frac{1}{4} - \frac{\alpha}{4}$, which is already covered.
We record the additional assumption:
\begin{equation}
\label{eq:s2ConditionI}
s_2(\alpha,\varepsilon) \geq \frac{\alpha - \frac{3}{4}}{20} +\varepsilon.
\end{equation}

\smallskip

Next, we consider $L \geq N^3$. From time localization, regularity, and derivative loss we find
\begin{equation*}
\big( \frac{N_1}{N} \big)^\alpha N (N_1^2 N)^{-\frac{1}{2}} N^{20s} \lesssim N_1^{21 s}.
\end{equation*}
This is equivalent to
\begin{equation*}
N_1^{\alpha-1} N^{\frac{1}{2}-\alpha + 20s} \lesssim N_1^{21 s}.
\end{equation*}
Under the assumption \eqref{eq:s2Start} this holds true for
\begin{equation}
\label{eq:s2ConditionII}
s_2(\alpha,\varepsilon) \geq \frac{-1/2 + \alpha}{20} + \varepsilon.
\end{equation}

\smallskip

\emph{Resonant case (ii):} $L=L_{\max} = N_1^2 N$. We have found above
\begin{equation*}
L^{-\frac{1}{4}} N^{-\frac{3}{4}} \| f_{1,N_1,L_1} * f_{2,N_2,L_2} \|_{L^2_{\tau,\xi,\eta}} \lesssim N_1^{-\frac{1}{2}} N^{-1} (N_1^{\frac{1}{8}-\frac{\alpha}{8}+\varepsilon})^2 \prod_{i=1}^2 L_i^{\frac{1}{2}} \| f_{i,N_i,L_i} \|_{L^2_{\tau,\xi,\eta}}.
\end{equation*}
Taking into account time localization, derivative loss, and regularity of the functions we find the condition
\begin{equation*}
N^{20s} N \big( \frac{N_1}{N} \big)^\alpha N_1^{-\frac{1}{2}} N^{-1} (N_1^{\frac{1}{8}-\frac{\alpha}{8}+\varepsilon})^2 \lesssim N_1^{21 s}.
\end{equation*}
This gives the new condition
\begin{equation}
\label{eq:s2ConditionIII}
\frac{3 \alpha}{4} - \frac{1}{4} + 2 \varepsilon \leq 21 s \Rightarrow s_2(\alpha,\varepsilon) \geq \frac{3 \alpha -1}{84} +  \varepsilon.
\end{equation}

\smallskip

\emph{Non-resonant case (i):} $L_1 = L_{\max} \gg N_1^2 N$, $L_{\max} \gg L_{\text{med}}$. Firstly, we suppose that $L \leq N^3$. In this case we have proved the estimate
\begin{equation*}
L^{-\frac{1}{2}} \| 1_{D_{N,L}} (f_{1,N_1,L_1} * f_{2,N_2,L_2}) \|_{L^2_{\tau,\xi,\eta}} \lesssim N_1^{-1-\frac{\alpha}{2}} \prod_{i=1}^2 L_i^{\frac{1}{2}} \| f_{i,N_i,L_i} \|_2.
\end{equation*}
This gives the condition:
\begin{equation*}
\big( \frac{N_1}{N} \big)^\alpha N N_1^{-1-\frac{\alpha}{2}} N^{20s} \lesssim N_1^{21 s} \Leftrightarrow N^{20s+1-\alpha} \lesssim N_1^{21 s + 1 - \frac{\alpha}{2}},
\end{equation*}
which was already encountered in \eqref{eq:s2ConditionI}.

\smallskip

If $N^3 \leq L \leq N_1^3$ and $L_1 = L_{\max} =N_1^2 N$, then we find from a bilinear Strichartz estimate after invoking duality (Proposition \ref{prop:SecondOrderBilinearEstimate}):
\begin{equation*}
L^{-\frac{1}{4}} N^{-\frac{3}{4}} \| 1_{D_{N,L}} (f_{1,N_1,L_1} * f_{2,N_2,L_2}) \|_{L^2_{\tau,\xi,\eta}} \lesssim ( N_1^2 N )^{-\frac{1}{2}} \prod_{i=1}^2 L_i^{\frac{1}{2}} \| f_{i,N_i,L_i} \|_{L^2}.
\end{equation*}
This leads to the condition:
\begin{equation*}
N \big( \frac{N_1}{N} \big)^\alpha (N_1^2 N)^{-\frac{1}{2}} N^{20s} \lesssim N_1^{21 s},
\end{equation*}
which is satisfied with \eqref{eq:s2ConditionII} and \eqref{eq:s2Start}.

\smallskip

If $L \geq N_1^3$, we have showed above that
\begin{equation*}
\begin{split}
&\quad L^{-\frac{1}{2}} (1+ L/N^3)^{\frac{1}{4}} \| 1_{D_{N,L}} ( f_{1,N_1,L_1} * f_{2,N_2,L_2} ) \|_{L^2_{\tau,\xi,\eta}} \\ &\lesssim N_1^{-\frac{3}{4}} N^{-\frac{3}{4}} N^{\frac{1}{2}} N_1^{-\frac{\alpha}{2}} \prod_{i=1}^2 L_i^{\frac{1}{2}} \| f_{i,N_i,L_i} \|_{L^2_{\tau,\xi,\eta}}.
\end{split}
\end{equation*}
With time localization, Sobolev regularity, and derivative loss, we find
\begin{equation*}
\big( \frac{N_1}{N} \big)^\alpha N^{20s} N^{\frac{3}{4}} N_1^{-\frac{3}{4}} N_1^{-\frac{\alpha}{2}} \lesssim N_1^{21s}.
\end{equation*}
This is satisfied with \eqref{eq:s2Start} and \eqref{eq:s2ConditionII}.

\emph{Non-resonant case (ii):} $L=L_{\max} \gg N_1^2 N$, $L_{\max} \gg L_{\text{med}}$. Again it suffices to check $N_1^2 N \leq L \leq N_1^3$, and we estimate via a bilinear Strichartz estimate:
\begin{equation*}
\begin{split}
&\quad L^{-\frac{1}{4}} N^{-\frac{3}{4}} \| 1_{D_{N,L}} (f_{1,N_1,L_1} * f_{2,N_2,L_2} ) \|_{L^2_{\tau,\xi,\eta}} \\ &\lesssim N_1^{-\frac{1}{2}} N^{-1} N^{\frac{1}{2}} N_1^{-\frac{\alpha}{2}} \prod_{i=1}^2 L_i^{\frac{1}{2}} \| f_{i,N_i,L_i} \|_{L^2}.
\end{split}
\end{equation*}
This gives the condition
\begin{equation*}
\big( \frac{N_1}{N} \big)^\alpha \cdot N \cdot N^{20s} N_1^{-\frac{1}{2}} N^{-1} N^{\frac{1}{2}} N_1^{-\frac{\alpha}{2}} \lesssim N_1^{21 s},
\end{equation*}
which is satisfied with \eqref{eq:s2ConditionII}.

We remark that as well \eqref{eq:s2ConditionI} as \eqref{eq:s2ConditionII} are weaker conditions than \eqref{eq:s2ConditionIII} for $0<\alpha<\frac{1}{7}$. In summary, we have proved \eqref{eq:ShorttimeBilinearDifferences} for $0<\alpha<\frac{1}{7}$ and
\begin{equation*}
s_2(\alpha,\varepsilon) \geq \max( - \frac{\alpha}{2} + \varepsilon , \frac{3 \alpha - 1}{84} + \varepsilon ).
\end{equation*}
We note that for $0<\alpha<\frac{1}{45}$, we can reduce to the condition
\begin{equation*}
s_2(\alpha,\varepsilon) \geq - \frac{\alpha}{2} + \varepsilon. 
\qedhere
\end{equation*}
\end{proof}

\section{Trilinear estimates}
\label{section:Trilinear}
In preparation of the short-time energy estimate to be proved in the next section, we show a trilinear estimate, which interpolates between the High-Low-High and High-High-High-interaction.

\smallskip

We firstly handle the easier case of comparable frequencies.
\begin{proposition}
\label{prop:TrilinearEstimateHighHighHigh}
Let $\alpha \in (0,1/8)$, $N_i, N \in 2^{\N_0}$, $N_1 \sim N_2 \sim N_3 \sim N$, and $L_i \in 2^{\N_0}$ with $L_i \gtrsim N_i^\alpha$. Let $f_{i,N_i,L_i} \in L^2_{\tau,\xi,\eta}$ with $\operatorname{supp}(f_{i,N_i,L_i}) \subseteq D_{N_i,L_i}$.
 Then the following estimate holds for any $a_1 \in (0,1/8)$:
\begin{equation*}
\int (f_{1,N_1,L_1} * f_{2,N_2,L_2} ) f_{3,N_3,L_3} d\xi d\eta d\tau \lesssim N^{-1} ( N^{-\frac{1}{4}+\frac{a_1}{6}+\varepsilon} \vee N^{-a_1 - \frac{\alpha}{2}}) \prod_{i=1}^3 L_i^{\frac{1}{2}} \| f_{i,N_i,L_i} \|_2.
\end{equation*}
\end{proposition}
\begin{proof}
By convolution constraint and the resonance relation we always have $L_{\max} \gtrsim N^3$. The case $L_{\max} \sim L_{\text{med}} \gtrsim N^3$ is readily estimated by Proposition \ref{prop:SecondOrderBilinearEstimate}. Indeed, by symmetry we can suppose that $L_1 \sim L_{\max}$. Then we find
\begin{equation}
\label{eq:StronglyNonResonantCaseHighHighHighA}
\begin{split}
&\quad \| (f_{1,N_1,L_1} * f_{2,N_2,L_2}) f_{3,N_3,L_3} \|_{L^2_{\tau,\xi,\eta}} \\
 &\leq \| f_{1,N_1,L_1} \|_{L^2} \| f_{2,N_2,L_2} * f_{3,N_3,L_3} \|_{L^2} \\
&\lesssim (N^3)^{-\frac{1}{2}} L_1^{\frac{1}{2}} N^{\frac{3}{4}} (L_2 L_3)^{\frac{1}{2}} (N^3)^{-\frac{1}{4}} \prod_{i=1}^3 \| f_{i,N_i,L_i} \|_2 \\
&\lesssim N^{-\frac{3}{2}} \prod_{i=1}^3 L_i^{\frac{1}{2}} \| f_{i,N_i,L_i} \|_2.
\end{split}
\end{equation}

We suppose in the following that $L_{\max} \gg L_{\text{med}}$. 

\smallskip

$\cdot$ Resonant case: $L_{\max} \sim N^3$. Suppose by symmetry $L_3 = L_{\max}$. We can apply H\"older's inequality and two $L^4_{t,x,y}$-Strichartz estimates from Theorem \ref{thm:ImprovedL4Strichartz} to find
\begin{equation}
\label{eq:ResonantTrilinearEstimateHighHighHigh}
\begin{split}
\big| \int (f_{1,N_1,L_1} * f_{2,N_2,L_2} ) f_{3,N_3,L_3} d\xi d\eta d\tau \big| &\lesssim \prod_{i=1}^2 \| \mathcal{F}^{-1}_{t,x,y}[ f_{i,N_i,L_i} ] \|_{L^4_{t,x,y}} \| f_{3,N_3,L_3} \|_{L^2_{\tau,\xi,\eta}} \\
&\lesssim_\varepsilon N^{2 \big( \frac{1}{8} - \frac{\alpha}{8} + \frac{\alpha}{12} + \varepsilon \big)} N^{-\frac{3}{2}} \prod_{i=1}^3 L_i^{\frac{1}{2}} \| f_{i,N_i,L_i} \|_{L^2}.
\end{split}
\end{equation}
Note that this an estimate which cannot be replaced with a short-time bilinear estimate.

$\cdot$ Non-resonant case: $\cdot L_{\max} \sim M^2 N$, $M \gg N$, $L_{\text{med}} \ll L_{\max}$. In this case we have
\begin{equation*}
L_{\max} \sim |\Omega| \sim \frac{| \eta_1 \xi_2 - \eta_2 \xi_1 \big|^2}{\xi_1 \xi_2 (\xi_1+\xi_2)},
\end{equation*}
which implies
\begin{equation*}
\big| \frac{\eta_1}{\xi_1} - \frac{\eta_2}{\xi_2} \big| \sim M.
\end{equation*}
 We let to this end with $J_i$ intervals of $(\eta_i/\xi_i)$ of length $M$:
\begin{equation*}
\begin{split}
&\quad \int (f_{1,N_1,L_1} * f_{2,N_2,L_2}) f_{3,N_3,L_3} d\xi d\eta d\tau \\
 &= \sum_{J_1 \sim J_2} \int (f_{1,N_1,L_1}^{J_1} * f_{2,N_2,L_2}^{J_2} ) f_{3,N_3,L_3} d\xi d\eta d\tau \\
&\leq \sum_{J_1 \sim J_2} \| f_{1,N_1,L_1}^{J_1} * f_{2,N_2,L_2}^{J_2} \|_{L^2_{\tau,\xi,\eta}} \| f_{3,N_3,L_3} \|_{L^2_{\tau,\xi,\eta}}.
\end{split}
\end{equation*}
By Galilean invariance $\eta_i \to \eta_i - A \xi_i$ with $A \in \Z$, we can suppose that the $(\eta_i/\xi_i)$-intervals are of length $M$ around the origin. Hence, we can suppose that $|\eta_i| \lesssim MN$ and $|\eta| \lesssim MN$. In the following the $J$-superscript is omitted to lighten the notation.

\smallskip

We break the $\eta$-frequencies into intervals $[k N^2, (k+1) N^2]$, $k \in \Z$ with $|k| \lesssim M/N$. This yields a decomposition of the $\eta_i$-frequencies into intervals of length $\sim N^2$. There are $\mathcal{O}(M/N)$ intervals, which incurs a factor $\big( \frac{M}{N} \big)^{\frac{1}{2}}$ by Cauchy-Schwarz. Correspondingly, we split the $\xi$-frequencies into intervals of length $N^2 / M$, such that after rescaling the intervals are of length $N/M$, which matches the offset $|k| \lesssim \frac{MN}{N^2} = \frac{M}{N}$. This incurs another factor of $\big( \frac{M}{N} \big)^{\frac{1}{2}}$. We obtain from applying the $L^4_{t,x,y}$-Strichartz estimates Proposition \ref{prop:L4StrichartzShifted} with offset in $\eta$ to the functions $f'$ with additional frequency localization:
\begin{equation*}
\begin{split}
\| f'_{1,N_1,L_1} * f'_{2,N_2,L_2} \|_{L^2_{\tau,\xi,\eta}}
 &\leq \| \mathcal{F}^{-1}_{t,x,y}[f_{1,N_1,L_1}'] \|_{L^4_{t,x,y}} \| \mathcal{F}^{-1}_{t,x,y}[f_{2,N_2,L_2}'] \|_{L^4_{t,x,y}} \\
&\lesssim_\varepsilon N^{\frac{1}{4}+2\varepsilon} \big( \frac{M}{N} \big)^{\frac{1}{6}}  \prod_{i=1}^2 L_i^{\frac{1}{2}} \| f'_{i,N_i,L_i} \|_{L^2}.
\end{split}
\end{equation*}

Taking into account the frequency-localization we find
\begin{equation*}
\big| \int (f_{1,N_1,L_1} * f_{2,N_2,L_2}) f_{3,N_3,L_3} \big| \lesssim_\varepsilon N^{\frac{1}{4}+2 \varepsilon} \big( \frac{M}{N} \big) \big( \frac{M}{N} \big)^{\frac{1}{6}} (M^2 N)^{-\frac{1}{2}} \prod_{i=1}^3 L_i^{\frac{1}{2}} \| f_{i,N_i,L_i} \|_{L^2}.
\end{equation*}

For $M \leq N^{1+a_1}$ we find
\begin{equation}
\label{eq:SmallMHighHighHigh}
\big| \int (f_{1,N_1,L_1} * f_{2,N_2,L_2} ) f_{3,N_3,L_3} \big| \lesssim_\varepsilon N_1^{\frac{1}{4}+2 \varepsilon} N_1^{-\frac{3}{2}} N^{\frac{a_1}{6}} \prod_{i=1}^3 L_i^{\frac{1}{2}} \| f_{i,N_i,L_i} \|_{L^2}.
\end{equation}
For $M \geq N^{1+a_1}$ we use instead a bilinear Strichartz estimate provided by Proposition \ref{prop:SimplifiedBilinearStrichartz}:
\begin{equation}
\label{eq:LargeMHighHighHigh}
\big| \int (f_{1,N_1,L_1} * f_{2,N_2,L_2} ) f_{3,N_3,L_3} \big| \lesssim (M^2 N)^{-\frac{1}{2}} N^{\frac{1}{2}} N^{-\frac{\alpha}{2}} \prod_{i=1}^3 L_i^{\frac{1}{2}} \| f_{i,N_i,L_i} \|_{L^2}.
\end{equation}
Collecting estimates \eqref{eq:StronglyNonResonantCaseHighHighHighA}, \eqref{eq:ResonantTrilinearEstimateHighHighHigh}, and \eqref{eq:SmallMHighHighHigh}, \eqref{eq:LargeMHighHighHigh}, the latter estimates being dominant, we complete the proof.
\end{proof}

Now we turn to the slightly more involved case of a trilinear interaction involving one lower, but not too low frequency.
\begin{proposition}
\label{prop:TrilinearEstimateHighLowHigh}
Let $a_2 \in (0,1/8)$, $\alpha \in (0,1/8)$, $N_i \in 2^{\N_0}$, $N_1 \sim N_3 \gg N_2 $, $N_2 \gtrsim N_1^{1-a_2}$, and $L_i \in 2^{\N_0}$ with $L_i \gtrsim N_i^\alpha$. Let $f_{i,N_i,L_i} \in L^2_{\tau,\xi,\eta}$ with $\operatorname{supp}(f_{i,N_i,L_i}) \subseteq D_{N_i,L_i}$.
\begin{itemize}
\item Resonant case: Suppose that $L_{\max} \sim N_1^2 N_2$.
 Then the following estimate holds:
\begin{equation*} 
\begin{split}
&\quad \int (f_{1,N_1,L_1} * f_{2,N_2,L_2} ) f_{3,N_3,L_3} d\xi d\eta d\tau \\
&\lesssim_\varepsilon N_2^{-1} N_1^{-\frac{1}{4}+\frac{11 a_2}{24}+\varepsilon} \prod_{i=1}^3 L_i^{\frac{1}{2}} (1+L_i/N_i^3)^{\frac{1}{4}} \| f_{i,N_i,L_i} \|_2.
\end{split}
\end{equation*}
\item Non-resonant case: Suppose that $L_{\max} \sim M^2 N_2$ for $M \gg N_1$. Then the following estimate holds:
\begin{equation*}
\begin{split}
&\quad \int (f_{1,N_1,L_1} * f_{2,N_2,L_2} ) f_{3,N_3,L_3} d\xi d\eta d\tau \\
&\lesssim \big( N_2^{-1} \frac{M^{\frac{1}{6}} (N_1 N_2)^{\frac{1}{24}+\varepsilon}}{N_2^{\frac{1}{2}}} \wedge M^{-1} N_2^{-\frac{1}{2}} (1+N_2^{\frac{1}{2}} N_1^{-\frac{\alpha}{2}})) \prod_{i=1}^3 L_i^{\frac{1}{2}} (1+L_i/N_i^3)^{\frac{1}{4}} \| f_{i,N_i,L_i} \|_2.
\end{split}
\end{equation*}
\end{itemize}
\end{proposition}

When we distinguish in the non-resonant case between $M \lesssim N_1^{1+a_4}$ and $M \gg N_1^{1+a_4}$ we obtain the following:
\begin{corollary}
\label{cor:NonResonantTrilinearEstimateHighLowHigh}
With the notation as above, in the non-resonant case the following estimate holds for any $a_4 \in (0,1/8)$:
\begin{equation*}
\begin{split}
&\quad \int (f_{1,N_1,L_1} * f_{2,N_2,L_2} ) f_{3,N_3,L_3} d\xi d\eta d\tau \\
&\lesssim_\varepsilon N_2^{-1} ( N_1^{\frac{a_4}{6}-\frac{1}{4}+\frac{11 a_2}{24}+\varepsilon} \vee N_1^{-a_4- \frac{\alpha}{2}+\varepsilon}) \prod_{i=1}^3 L_i^{\frac{1}{2}} (1+L_i/N_i^3)^{\frac{1}{4}} \| f_{i,N_i,L_i} \|_2.
\end{split}
\end{equation*}
\end{corollary}

\begin{proof}[Proof~of~Proposition~\ref{prop:TrilinearEstimateHighLowHigh}]
By convolution constraint and the resonance relation we always have $L_{\max} \gtrsim N_1^2 N_2$. The case $L_{\max} \sim L_{\text{med}} \gtrsim N_1^2 N_2$ is readily estimated by bilinear Strichartz estimates from Proposition \ref{prop:SecondOrderBilinearEstimate}. Indeed, by symmetry we can suppose that $L_1 \sim L_{\max}$. Then we find
\begin{equation}
\label{eq:StronglyNonResonantCaseHighHighHigh}
\begin{split}
&\quad \| (f_{1,N_1,L_1} * f_{2,N_2,L_2}) f_{3,N_3,L_3} \|_{L^2_{\tau,\xi,\eta}} \\
 &\leq \| f_{1,N_1,L_1} \|_{L^2} \| f_{2,N_2,L_2} * f_{3,N_3,L_3} \|_{L^2} \\
&\lesssim (N_1^2 N_2)^{-\frac{1}{2}} L_1^{\frac{1}{2}} N_2^{\frac{3}{4}} (L_2 L_3)^{\frac{1}{2}} (N_1^2 N_2)^{-\frac{1}{4}} \prod_{i=1}^3 \| f_{i,N_i,L_i} \|_2 \\
&\lesssim N_1^{-1} N_2^{-\frac{1}{2}} \prod_{i=1}^3 L_i^{\frac{1}{2}} \| f_{i,N_i,L_i} \|_2.
\end{split}
\end{equation}

\smallskip

So, in the following we suppose that $L_{\max} \gg L_{\text{med}}$. 

\smallskip

$\bullet$ \emph{Resonant case:} $L_{\max} \sim N_1^2 N_2$. First, we suppose that the high frequency carries the high modulation $L_3 = L_{\max}$. In this case we obtain
\begin{equation*}
D \sim \big| \frac{\eta_1}{\xi_1} - \frac{\eta_2}{\xi_2} \big| \lesssim N_1. 
\end{equation*}
 By the usual almost orthogonality and Galilean invariance argument we can suppose  that $|\eta_i| \lesssim N_1 N_i$. We aim to obtain a bound via two $L^4$-Strichartz estimates. To $f_{1,N_1,L_1}$ we can apply Theorem \ref{thm:L4Strichartz}:
\begin{equation*}
\| \mathcal{F}^{-1}_{t,x,y} [f_{1,N_1,L_1}] \|_{L^4_{t,x,y}} \lesssim_\varepsilon N_1^{\frac{1}{8}+\varepsilon} L_1^{\frac{1}{2}} \| f_{1,N_1,L_1} \|_2.
\end{equation*}
For $f_{2,N_2,L_2}$ we aim to apply Proposition \ref{prop:L4StrichartzShifted}.
To this end, we require an additional decomposition of the support in the $\eta$-frequency into intervals of length $N_2^2$ and of the $\xi$-support into intervals of length $N_2^2 / N_1$, which incurs a factor of $(N_1 / N_2)$ by two applications of the Cauchy-Schwarz inequality. This allows us to apply Proposition \ref{prop:L4StrichartzShifted} with $k \sim N_1 / N_2$. We denote the decomposition by $f_{2,N_2,L_2} = \sum_J f^{J}_{2,N_2,L_2}$. Proposition \ref{prop:L4StrichartzShifted} yields
\begin{equation*}
\| \mathcal{F}^{-1}_{t,x,y}[f^J_{2,N_2,L_2}] \|_{L^4_{t,x,y}} \lesssim L_2^{\frac{1}{2}} N_2^{\frac{1}{8}+\varepsilon} \big( \frac{N_1}{N_2} \big)^{\frac{1}{12}} \| f_{2,N_2,L_2} \|_2.
\end{equation*}
Consequently, by applying H\"older's inequality and Cauchy-Schwarz inequality the above yields
\begin{equation}
\label{eq:TrilinearResonantHighLow}
\begin{split}
&\quad \int (f_{1,N_1,L_1} * f_{2,N_2,L_2}) f_{3,N_3,L_3} d \xi d\eta d\tau \\
&\leq \| f_{1,N_1,L_1} * f_{2,N_2,L_2} \|_{L^2_{\tau,\xi,\eta}} \| f_{3,N_3,L_3} \|_{L^2_{\tau,\xi,\eta}} \\
&\lesssim (N_1^2 N_2)^{-\frac{1}{2}} N_1^{\frac{1}{8}+\varepsilon} \big( \frac{N_1}{N_2} \big) N_2^{\frac{1}{8}+\varepsilon} \big( \frac{N_1}{N_2} \big)^{\frac{1}{12}} \prod_{i=1}^3 L_i^{\frac{1}{2}} \| f_{i,N_i,L_i} \|_2 \\
&\lesssim N_2^{-1} N_1^{-\frac{1}{4}+\frac{11a_2}{24}} \prod_{i=1}^3 L_i^{\frac{1}{2}} \| f_{i,N_i,L_i} \|_2.
\end{split}
\end{equation}

The case $L_2 = L_{\max} \sim N_1^2 N_2$ gives a more favorable bound: In this case we have the transversality bound
\begin{equation*}
\big| \frac{\eta_1}{\xi_1} - \frac{\eta_3}{\xi_3} \big| \lesssim N_2.
\end{equation*}
Hence, we can suppose by almost orthogonality and Galilean invariance that $|\eta_i| \lesssim N_1 N_2 \ll N_1^2$ for $i=1,3$.

We find from two $L^4_{t,x,y}$-Strichartz estimates provided by Theorem \ref{thm:ImprovedL4Strichartz}
\begin{equation*}
\begin{split}
&\quad \int (f_{1,N_1,L_1} * f_{2,N_2,L_2}) f_{3,N_3,L_3} d\xi d\eta d\tau \\
 &\leq \| f_{2,N_2,L_2} \|_{L^2_{\tau,\xi,\eta}} \prod_{i=1,3} \| \mathcal{F}^{-1}_{t,x,y} [f_{i,N_i,L_i}] \|_{L^4_{t,x,y}} \\
&\lesssim (N_1^{\frac{1}{8}-\frac{\alpha}{8}+\varepsilon})^2 N_1^{-\frac{3}{2}} \prod_{i=1}^3 L_i^{\frac{1}{2}} (1+L_i/N_i^3)^{\frac{1}{4}} \| f_{i,N_i,L_i} \|_2.
\end{split}
\end{equation*}
This is clearly more favorable than \eqref{eq:TrilinearResonantHighLow}. The estimate of the resonant case is complete.

\medskip

\emph{Non-resonant case:} $L_{\max} \sim M^2 N_2$ for $M \gg N_1$. 

Firstly, suppose that $L_3 = L_{\max}$ and we can suppose that $D \sim M$ which is a consequence of the resonance relation. By Galilean invariance we can suppose  that $|\eta_i| \lesssim M N_i$. We aim to obtain a bound via two $L^4$-Strichartz estimates.

In this case $\mathcal{F}^{-1}_{t,x,y}[f_{i,N_i,L_i}]$, $i=1,2$ are estimated via Proposition \ref{prop:L4StrichartzShifted}. We decompose the $\xi$-support into intervals of length $N_2^2 / M$ and the $\eta$-support into intervals of length $N_2^2$. This incurs a factor of $(M/N_2)$. Then we can apply H\"older's inequality and Proposition \ref{prop:L4StrichartzShifted} to find
\begin{equation*}
\begin{split}
&\quad \| f_{1,N_1,L_1} * f_{2,N_2,L_2} \|_{L^2_{\tau,\xi,\eta}} \\
 &\leq \| \mathcal{F}^{-1}_{t,x,y}[f_{1,N_1,L_1}] \|_{L^4_{t,x,y}} \| \mathcal{F}^{-1}_{t,x,y} [f_{2,N_2,L_2}] \|_{L^4_{t,x,y}} \\
&\lesssim \big( \frac{M}{N_2} \big) N_1^{\frac{1}{8}+\varepsilon} \big( \frac{M}{N_1} \big)^{\frac{1}{12}} L_1^{\frac{1}{2}} \| f_{1,N_1,L_1} \|_{L^2_{\tau,\xi,\eta}} N_2^{\frac{1}{8}+\varepsilon} \big( \frac{M}{N_2} \big)^{\frac{1}{12}} L_2^{\frac{1}{2}} \| f_{2,N_2,L_2} \|_{L^2}.
\end{split}
\end{equation*}
By an application of H\"older's inequality and taking into account $L_3 \sim M^2 N_2$ we find
\begin{equation*}
\begin{split}
&\quad \int (f_{1,N_1,L_1} * f_{2,N_2,L_2}) f_{3,N_3,L_3} d\xi d\eta d\tau \\
&\lesssim \frac{1}{N_2} \frac{M^{\frac{1}{6}} (N_1 N_2)^{\frac{1}{24}+\varepsilon}}{N_2^{\frac{1}{2}}} \prod_{i=1}^3 L_i^{\frac{1}{2}} \| f_{i,N_i,L_i} \|_{L^2}.
\end{split}
\end{equation*}

We impose a bound $M \lesssim N_1^{1+a_4}$ to find (recalling that $N_2 \gtrsim N_1^{1-a_2}$):
\begin{equation}
\label{eq:NonresonantTrilinearHighLowHighI}
\int (f_{1,N_1,L_1} * f_{2,N_2,L_2}) f_{3,N_3,L_3} d\xi d\eta d\tau \lesssim \frac{1}{N_2} N_1^{\frac{1+a_4}{6}-\frac{10}{24}+\frac{11 a_2}{24}} \prod_{i=1}^3 L_i^{\frac{1}{2}} \| f_{i,N_i,L_i} \|_2.
\end{equation}
If $M \gg N_1^{1+a_4}$, we can apply a bilinear Strichartz estimate from Proposition \ref{prop:SimplifiedBilinearStrichartz}:
\begin{equation}
\label{eq:NonresonantTrilineraHighLowHighII}
\begin{split}
\int (f_{1,N_1,L_1} * f_{2,N_2,L_2}) f_{3,N_3,L_3} d\xi d\eta d\tau &\lesssim (M^2 N_2)^{-\frac{1}{2}} N_2^{\frac{1}{2}} N_1^{-\frac{\alpha}{2}} \prod_{i=1}^3 L_i^{\frac{1}{2}} \| f_{i,N_i,L_i} \|_2 \\
&\lesssim N_2^{-1} N_1^{-a_4- \frac{\alpha}{2}} \prod_{i=1}^3 L_i^{\frac{1}{2}} \| f_{i,N_i,L_i} \|_2.
\end{split}
\end{equation}

Next, we handle the case $L_2 = L_{\max} \sim M^2 N_2$. Here we have the transversality bound
\begin{equation*}
\big| \frac{\eta_1}{\xi_1} - \frac{\eta_3}{\xi_3} \big| \sim (M/N_1) N_2.
\end{equation*}
This implies by the usual Galilean invariance and almost orthogonality $|\eta_i| \lesssim M N_2 $. For $M \lesssim N_1^{2} / N_2$ we can apply Theorem \ref{thm:ImprovedL4Strichartz} directly:
\begin{equation*}
\begin{split}
&\quad \int (f_{1,N_1,L_1} * f_{2,N_2,L_2}) f_{3,N_3,L_3} d\xi d\eta d\tau \\
 &\leq \| f_{2,N_2,L_2} \|_{L^2_{\tau,\xi,\eta}} \| f_{1,N_1,L_1} * f_{3,N_3,L_3} \|_{L^2_{\tau,\xi,\eta}} \\
&\lesssim (M^2 N_2)^{-\frac{1}{2}} (N_1^{\frac{1}{8}-\frac{\alpha}{8}+\frac{\alpha}{12}+\varepsilon})^2 \prod_{i=1}^3 L_i^{\frac{1}{2}} \| f_{i,N_i,L_i} \|_2 \\
&\lesssim M^{-\frac{3}{2}} (N_1^{\frac{1}{8}-\frac{\alpha}{8}+\frac{\alpha}{12}+\varepsilon})^2 \prod_{i=1}^3 L_i^{\frac{1}{2}} (1+L_i/N_i^3)^{\frac{1}{4}} \| f_{i,N_i,L_i} \|_2.
\end{split}
\end{equation*}

For $M \gg N_1^{2}/N_2$ we need to decompose the $\eta$-frequencies into intervals of length $N_1^2$ with the number of sub-intervals given by $k \sim M N_2 /N_1^2$. Correspondingly, we decompose the $\xi$-frequencies into intervals of length $N_1/k \sim \frac{N_1^3}{M N_2} $. Since we can suppose by convolution constraint that the $\xi$-frequencies are localized to an interval of length $N_2$, the number of subintervals is given by $\sim \frac{M N_2^2}{N_1^4}$. After additional localization of the frequency support, we can use Proposition \ref{prop:L4StrichartzShifted} to find
\begin{equation*}
\| f_{1,N_1,L_1}' * f_{3,N_3,L_3}' \|_{L^2_{\tau,\xi,\eta}} \lesssim_\varepsilon (N_1^{\frac{1}{8} +\varepsilon})^2 \big( \frac{M N_2}{N_1} \big)^{\frac{1}{6}} \prod_{i=1,3} \| f'_{i,N_i,L_i} \|_2.
\end{equation*}
Taking into account the additional frequency localization and the modulation size $L_2 \sim M^2 N_2$ we obtain
\begin{equation*}
\begin{split}
&\quad \int (f_{1,N_1,L_1} * f_{2,N_2,L_2}) f_{3,N_3,L_3} d\xi d\eta d\tau \\
&\leq \| f_{2,N_2,L_2} \|_{L^2_{\tau,\xi,\eta}} \| f_{1,N_1,L_1} * f_{3,N_3,L_3} \|_{L^2_{\tau,\xi,\eta}} \\
&\lesssim (M^2 N_2)^{-\frac{1}{2}} \big( \frac{M N_2}{N_1^2} \big)^{\frac{1}{2}} \big( \frac{M N_2^2}{N_1^4} \big)^{\frac{1}{2}} (N_1^{\frac{1}{8} +\varepsilon})^2 \big( \frac{M N_2}{N_1} \big)^{\frac{1}{6}} \prod_{i=1}^3 L_i^{\frac{1}{2}} \| f_{i,N_i,L_i} \|_2 \\
&\lesssim M^{-\frac{3}{2}} \frac{M N_2^{\frac{3}{2}}}{N_1^3} N_1^{\frac{1}{4} +\varepsilon} \big( \frac{M N_2}{N_1} \big)^{\frac{1}{6}} \prod_{i=1}^3 L_i^{\frac{1}{2}} (1+L_i/N_i^3)^{\frac{1}{4}} \| f_{i,N_i,L_i} \|_2 \\
&\lesssim M^{-\frac{1}{3}} N_2^{-1} N_1^{-\frac{1}{4}} \big( \frac{N_2}{N_1} \big)^{\frac{1}{6}} \prod_{i=1}^3 L_i^{\frac{1}{2}} (1+L_i/N_i^3)^{\frac{1}{4}} \| f_{i,N_i,L_i} \|_2.
\end{split}
\end{equation*}
This estimate improves on \eqref{eq:NonresonantTrilinearHighLowHighI}, for which reason we conclude the proof taking \eqref{eq:NonresonantTrilinearHighLowHighI} and \eqref{eq:NonresonantTrilineraHighLowHighII}.
\end{proof}

\begin{proof}[Proof~of~Corollary~\ref{cor:NonResonantTrilinearEstimateHighLowHigh}]
This is a consequence of \eqref{eq:NonresonantTrilinearHighLowHighI} and \eqref{eq:NonresonantTrilineraHighLowHighII}, since the estimates recorded below the aforementioned ones are more favorable.
\end{proof}

Lastly, we record the estimate for the case where the small frequency is quantitatively smaller than the highest frequency. In this case the main estimate is provided by the bilinear Strichartz estimate from Proposition \ref{prop:bilinear-sp}.

\begin{proposition}
\label{prop:HighVeryLowHighTrilinearEstimate}
Let $\alpha \in (0,1/8)$, $N_i \in 2^{\N_0}$, $N_1 \sim N_3 \gg N_2$, and $L_i \in 2^{\N_0}$ with $L_i \gtrsim N_i^\alpha$. Let $f_{i,N_i,L_i} \in L^2_{\tau,\xi,\eta}$ with $\operatorname{supp}(f_{i,N_i,L_i}) \subseteq D_{N_i,L_i}$. Then the following estimate holds:
\begin{equation*}
\begin{split}
\int (f_{1,N_1,L_1} * f_{2,N_2,L_2}) f_{3,N_3,L_3} d\xi d\eta d\tau &\lesssim N_1^{-1} (N_2^{-\frac{1}{4}} N_1^{-\frac{\alpha}{4}} + N_1^{-\frac{\alpha}{2}}) \log(N_1) \\
&\quad \prod_{i=1}^3 L_i^{\frac{1}{2}} (1+L_i/N_i^3)^{\frac{1}{4}} \| f_{i,N_i,L_i} \|_2.
\end{split}
\end{equation*}
\end{proposition}
\begin{proof}
This is straightforward from the preceding arguments and the bilinear Stri\-chartz estimate from Proposition \ref{prop:bilinear-sp}. Suppose that $L_{\max} \gg L_{\text{med}}$. First, we suppose that $L_3 = L_{\max}$, which is interchangeable with $L_3 = L_{\max}$. In the resonant case, $L_{\max} \sim N_1^2 N_2$ we can apply Proposition \ref{prop:bilinear-sp} with $D \lesssim N_1$ to find:
\begin{equation*}
\begin{split}
&\quad \int (f_{1,N_1,L_1} * f_{2,N_2,L_2} ) f_{3,N_3,L_3} d\xi d\eta d\tau \\ &\leq \| f_{3,N_3,L_3} \|_{L^2_{\tau,\xi,\eta}} \| f_{1,N_1,L_1} * f_{2,N_2,L_2} \|_{L^2_{\tau,\xi,\eta}} \\
&\lesssim (N_1^2 N_2)^{-\frac{1}{2}} (N_2^{\frac{1}{4}} N_1^{-\frac{\alpha}{4}} + N_2^{\frac{1}{2}} N_1^{-\frac{\alpha}{2}} ) \log(N_1) \prod_{i=1}^3 L_i^{\frac{1}{2}} (1+L_i/N_i^3)^{\frac{1}{4}} \| f_{i,N_i,L_i} \|_2.
\end{split}
\end{equation*}
In case $L_2 = L_{\max} \sim N_1^2 N_2$ we can use the analysis from Proposition \ref{prop:TrilinearEstimateHighLowHigh}, which gives a more favorable estimate.

\smallskip

In the non-resonant case $L_{\max} \sim M^2 N_2$, $M \gg N_1$, we again use the bilinear estimate in case $L_3 = L_{\max}$ with $D \sim M$. This yields the same estimate like above.
If $L_2 = L_{\max}$, we note that we have the transversality
\begin{equation*}
\big| \frac{\eta_1}{\xi_1} - \frac{\eta_3}{\xi_3} \big| \sim \frac{M N_2}{N_1} \gg N_2.
\end{equation*}
Then we can use the H\"older's inequality and the bilinear Strichartz estimate from Proposition \ref{prop:SimplifiedBilinearStrichartz} to find
\begin{equation*}
\begin{split}
&\quad \int (f_{1,N_1,L_1} * f_{2,N_2,L_2}) f_{3,N_3,L_3} d\xi d\eta d\tau \\
 &\leq \| f_{2,N_2,L_2} \|_{L^2_{\tau,\xi,\eta}} \| f_{1,N_1,L_1} * f_{3,N_3,L_3} \|_{L^2_{\tau,\xi,\eta}} \\
&\lesssim (M^3)^{-\frac{1}{2}} L_{2}^{\frac{1}{2}} (L_2 / N_2^3)^{\frac{1}{4}} N_2^{\frac{1}{2}} L_{\min}^{\frac{1}{2}} \langle L_{\text{med}} / N_2 \rangle^{\frac{1}{2}} \log(M) \prod_{i=1}^3  \| f_{i,N_i,L_i} \|_2 \\
&\lesssim M^{-\frac{3}{2}} \log(M) (1+N_2^{\frac{1}{2}} N_1^{-\frac{\alpha}{2}}) \prod_{i=1}^3 L_i^{\frac{1}{2}} (1+L_i/N_i^3)^{\frac{1}{4}} \| f_{i,N_i,L_i} \|_2
\end{split}
\end{equation*}
Summing over $M \gtrsim N$ finishes the proof.
\end{proof}

%
%

\section{Short-time energy estimates}
\label{section:ShorttimeEnergy}

\subsection{A priori estimates}

This section is devoted to prove energy estimates for solutions. We can take advantage of the conservative derivative nonlinearity and symmetry of the expression.
\begin{proposition}
\label{prop:APrioriEnergyEstimates}
Let $T \in (0,1]$, $0<\alpha<\frac{1}{10}$, and $s \geq s' \geq s_3(\alpha,\varepsilon) = - \frac{\alpha}{2} + \varepsilon$. Then the following estimate holds:
\begin{equation}
\label{eq:APrioriEnergyEstimates}
\| u \|^2_{E^s(T)} \lesssim \| u_0 \|_{H^{s,0}}^2 + T^\delta \| u \|^2_{F^s(T)} \| u \|_{F^{s'}(T)}.
\end{equation}
\end{proposition}
\begin{proof}
Since the low frequencies $\| P_{\leq 4} u \|_{E^s(T)} = \| P_{\leq 4} u_0 \|_{H^{s,0}}$ are estimated by the definition of the function spaces, it suffices to estimate $\| P_N u \|_{E^s(T)}$ for $N \geq 4$ and carry out the dyadic square summation.
We obtain an estimate for $\| P_N u(t) \|^2_{L^2}$, $t \in [0,T]$ by the fundamental theorem of calculus:
\begin{equation*}
\| P_N u(t) \|^2_{L^2} = \|P_N u(0) \|^2_{L^2} + 2 \int_0^t \int_{\T^2} P_N u(s) \partial_x P_N (u^2)(s) dx dy ds.
\end{equation*}
For technical reasons we suppose that $P_N$ denotes in the following a smooth frequency projection:
\begin{equation*}
\widehat{(P_N f)} (\xi,\eta) = \chi_1(\xi/N) \hat{f}(\xi,\eta)
\end{equation*}
with $\chi_1 \in C^\infty_c$ denoting a suitable bump funciton.
Before decomposing the time integral into intervals of size $T=T(N_{\max}) = N_{\max}^{-\alpha}$, we integrate by parts to let the derivative always act on the lowest frequency.

\medskip

We use the standard paraproduct decomposition:
\begin{equation}
\label{eq:ParaproductEnergyEstimates}
P_N (u^2) = 2 P_N (u P_{\ll N} u) + P_N( P_{\gtrsim N} u P_{\gtrsim N} u).
\end{equation}
Clearly, for the second term it is immediate how the derivative acts on low frequencies. We turn to the first term. Write for $K \ll N$
\begin{equation}
\label{eq:CommutatorArgumentI}
P_N ( u P_{K} u) = P_N u P_K u + [P_N(u P_K u) - P_N u P_K u].
\end{equation}
For the first term in \eqref{eq:CommutatorArgumentI} we obtain
\begin{equation*}
\begin{split}
\int_{\T^2} P_N u \partial_x (P_N u P_K u) dx dy &= \int_{\T^2} (\partial_x P_K u) (P_N u)^2 dx dy + \int_{\T^2} \frac{1}{2} (\partial_x (P_N u)^2) P_K u dx dy \\
&= \frac{1}{2} \int_{\T^2} (\partial_x P_K) u (P_N u)^2 dx dy.
\end{split}
\end{equation*}
For the second term we use Parseval's identity:
\begin{equation*}
\begin{split}
&\quad \int_{\T^2} P_N u \partial_x [P_N (u P_K u) - P_N u P_K u] \\
&= \sum_{\substack{(\xi_1,\eta_1) \in \Z^2, \\ (\xi_2,\eta_2) \in \Z^2}} \widehat{P_N u}(-\xi_1-\xi_2,-\eta_1-\eta_2) i(\xi_1+ \xi_2) m(\xi_1,\xi_2) \hat{u}(\xi_1,\eta_1) \hat{u}(\xi_2,\eta_2).
\end{split}
\end{equation*}
We denote the symbol by
\begin{equation*}
m(\xi_1,\xi_2) = \chi_N(\xi_1+\xi_2) \chi_K(\xi_2) - \chi_N(\xi_1) \chi_K(\xi_2) = [\chi_N(\xi_1+\xi_2) - \chi_N(\xi_1)] \chi_K(\xi_2).
\end{equation*}

As a consequence of the mean-value theorem, we have the size estimate
\begin{equation*}
|m(\xi_1,\xi_2)| \lesssim \frac{K}{N} \tilde{\chi}_N(\xi_1) \chi_K(\xi_2) =: \tilde{m}(\xi_1,\xi_2)
\end{equation*}
with $\tilde{\chi}_N$ denoting a mildly enlarged version of $\chi_N$. With the regularity estimates
\begin{equation*}
|\partial_{\xi_1}^{\alpha_1} \partial_{\xi_2}^{\alpha_2} m(\xi_1,\xi_2) | \lesssim_{\alpha_1,\alpha_2} N^{-\alpha_1} K^{-\alpha_2} \tilde{m}(\xi_1,\xi_2)
\end{equation*}
immediate, we can expand $m(\xi_1,\xi_2)$ as a Fourier series and effectively reduce to
\begin{equation*}
\int_{\T^2} P_N u \partial_x (P_N (u P_K u) - P_N u P_K u) dx dy \sim \int_{\T^2} P_N u \tilde{P}_N u (\partial_x P_K u) dx dy.
\end{equation*}

It remains to obtain estimates for
\begin{equation*}
K \int_0^T \int_{\T^2} \tilde{P}_N u \tilde{P}_N u P_K u \, dx dy ds, \quad K \lesssim N.
\end{equation*}
To this end, we consider a smooth partition of unity of $[0,T]$ into intervals of length $N^{-\alpha}$ to account for the definition of the function spaces:
\begin{equation*}
1_{[0,T]}(s) \sum_{n \in \Z} \gamma^3(N^\alpha s - n ) = 1_{[0,T]}(s).
\end{equation*}
Clearly,
\begin{equation*}
\# \{ n \in \Z : 1_{[0,T]}(s) \gamma^3(N^\alpha s - n ) \not \equiv 0 \} \lesssim T N^\alpha.
\end{equation*}
Let
\begin{equation*}
\mathcal{N} = \{ n \in \Z : 1_{[0,T]}(s) \gamma^3(N^\alpha s- n) \not \equiv 0 \}.
\end{equation*}
Suppose in the following for $n \in \mathcal{N}$ that
\begin{equation}
\label{eq:BulkCases}
1_{[0,T]}(t) \gamma^3(N^\alpha t-n ) = \gamma^3(N^\alpha t -n).
\end{equation}

 With the estimates not depending on the location of the time interval $I$, $|I| \lesssim N^{-\alpha}$, in the following we aim to prove estimates
\begin{equation}
\label{eq:ShorttimeFrequencyLocalizedEnergyEstimate}
\big| \int_{\R \times \T^2} \gamma^3(N^\alpha (t-t_I)) \tilde{P}_N u \tilde{P}_N u P_K u dx dy ds \big| \lesssim C(N,K) \| P_N u \|^2_{F_N} \| P_K u \|_{F_K}
\end{equation}
for $K \lesssim N$.

\smallskip

We introduce the notation
\begin{equation*}
\begin{split}
f_{1,N_1} &= \mathcal{F}_{t,x,y}[\gamma(N^\alpha(t-t_I)) \tilde{P}_N u], \quad f_{2,N_2} = \mathcal{F}_{t,x,y}[\gamma(N^\alpha(t-t_I)) P_K u], \\
f_{3,N_3} &= \mathcal{F}_{t,x,y}[\gamma(N^\alpha(t-t_I)) \tilde{P}_N u].
\end{split}
\end{equation*}
Next, we localize in modulation reflecting the time-localization
\begin{equation}
\label{eq:ModulationLocalization}
\begin{split}
f_{i,N_i,L_i} &= 1_{D_{N_i,L_i}} f_{i,N_i} \; \, \text{ for } L_i > N_1^\alpha, \\
f_{i,N_i,L_i} &= 1_{D_{N_i,\leq L_i}} f_{i,N_i} \text{ for } L_i = N_1^\alpha.
\end{split}
\end{equation}

We shall prove estimates
\begin{equation}
\label{eq:ModulationLocalizedEnergyEstimates}
\begin{split}
&\quad \big| \int (f_{1,N_1,L_1} * f_{2,N_2,L_2}) f_{3,N_3,L_3} d\xi d\eta d\tau \big| \\
&\lesssim C(N_1,N_2,N_3) \prod_{i=1}^3 L_i^{\frac{1}{2}} (1+L_i/N_i^3)^{\frac{1}{4}} \| f_{i,N_i,L_i} \|_{L^2},
\end{split}
\end{equation}
from which \eqref{eq:ShorttimeFrequencyLocalizedEstimate} follows by summation in $L_i$ and properties of the function spaces.

\medskip

Now we comment on the cases, in which the cut-off is not smooth: 
\begin{equation*}
\# \{n \in \N : 1_{[0,T]}(t) \gamma^3(N^\alpha t-n) \not \equiv \gamma^3(N^\alpha t -n) \} \leq 4.
\end{equation*}
For these finitely many cases we can use the following lemma:
\begin{lemma}[{\cite[p.~291]{IonescuKenigTataru2008}}]
If $I \subseteq \R$ is an interval, $k \in \Z$, $f_k \in X_K$, and $f_k^I = \mathcal{F}( 1_I(t) \mathcal{F}^{-1}(f_k))$, then 
\begin{equation*}
\sup_{L \in 2^{\N}} L^{\frac{1}{2}} (1+L/N^3)^{\frac{1}{4}} \| \eta_L(\tau-\omega(\xi,\eta)) f_k^I \|_{L^2} \lesssim \| f_k \|_{X_K}.
\end{equation*}
\end{lemma}
Then, we can interpolate suitable estimates in \eqref{eq:ModulationLocalizedEnergyEstimates} with the bilinear Strichartz estimate provided by Proposition \ref{prop:SecondOrderBilinearEstimate}:
\begin{equation*}
\big| \int (f_{1,N_1,L_1} * f_{2,N_2,L_2} ) f_{3,N_3,L_3} d\xi d\eta d\tau \big| \lesssim L_{\min}^{\frac{1}{2}} N_{\min}^{\frac{3}{4}} L_{\text{max}}^{\frac{1}{4}} \prod_{i=1}^3 \| f_{i,N_i,L_i} \|_2.
\end{equation*}
This gives the following variant of \eqref{eq:ModulationLocalizedEnergyEstimates}:
\begin{equation*}
\begin{split}
&\quad \big| \int (f_{1,N_1,L_1} * f_{2,N_2,L_2}) f_{3,N_3,L_3} d\xi d\eta d\tau \big| \\
&\lesssim C(N_1,N_2,N_3) N_{\max}^{0+} L_{\max}^{0-} \prod_{i=1}^3 L_i^{\frac{1}{2}} (1+L_i/N_i^3)^{\frac{1}{4}} \| f_{i,N_i,L_i} \|_{L^2},
\end{split}
\end{equation*}
which allows us to replace the $\ell^1$-summation over modulation with an $\ell^\infty$-norm. In addition we can combine the estimate with Lemma \ref{lem:ModulationSlack} to squeeze out the factor $T^\delta$ also for the contribution affected by the sharp time cutoff.

\bigskip

$\bullet$ High$-$High$-$High-estimates: $N_1 \sim N_2 \sim N_3$. Let $N = N_1$ for brevity. We can apply Proposition \ref{prop:TrilinearEstimateHighHighHigh} to find for $a_1 \in (0,1/8)$:
\begin{equation*}
\int (f_{1,N_1,L_1} * f_{2,N_2,L_2} ) f_{3,N_3,L_3} d\xi d\eta d\tau \lesssim N^{-1} ( N^{-\frac{1}{4}+\frac{a_1}{6}+\varepsilon} \vee N^{-a_1 - \frac{\alpha}{2}}) \prod_{i=1}^3 L_i^{\frac{1}{2}} \| f_{i,N_i,L_i} \|_2.
\end{equation*}

The first bound in the above estimate yields taking into account powers of $N$ from derivative and time localization, i.e., a factor of $N^{1+\alpha}$:
\begin{equation}
\label{eq:APrioriRegularityThresholdI}
s_3(\alpha,\varepsilon) \geq \alpha + \frac{a_1}{6} - \frac{1}{4} + \varepsilon.
\end{equation}
The second estimate gives the lower bound:
\begin{equation}
\label{eq:APrioriRegularityThresholdII}
s_3(\alpha,\varepsilon) \geq \frac{\alpha}{2} - a_1 + \varepsilon.
\end{equation}

$\bullet$ High$-$Low$-$High-interaction: In this case we estimate
\begin{equation*}
N_2 \iint \gamma^3(N_1^\alpha(t-t_I)) P_{N_1} u P_{N_2} u P_{N_3} u dx dy dt \text{ with } N_1 \sim N_3 \gg N_2.
\end{equation*}
To prove estimates \eqref{eq:APrioriEnergyEstimates}, we need to take into account the input regularity $N_1^{ 2s_3}$ or $N_2^{2s_3}$. In the following we focus on the latter as the first case is better behaved.

\smallskip

Frequency-dependent time localization incurs a factor of $N_1^\alpha$. After introducing notation for the space-time Fourier transform like above, we are again reduced to show trilinear convolution estimates as handled in Proposition \ref{prop:TrilinearEstimateHighLowHigh}:
\begin{equation*}
\int \big( f_{1,N_1,L_1} * f_{2,N_2,L_2} \big) f_{3,N_3,L_3} d\xi d\eta d\tau \lesssim C(N_1,N_2) \prod_{i=1}^3 L_i^{\frac{1}{2}} \| f_{i,N_i,L_i} \|_{L^2}.
\end{equation*}

First, we consider the case $N_2 \lesssim N_1^{1-a_2}$. In this case we apply Proposition \ref{prop:HighVeryLowHighTrilinearEstimate} to find
\begin{equation*}
\begin{split}
\int \big( f_{1,N_1,L_1} * f_{2,N_2,L_2} \big) f_{3,N_3,L_3} d\xi d\eta d\tau &\lesssim N_1^{-1} (N_2^{-\frac{1}{4}} N_1^{-\frac{\alpha}{4}} + N_1^{-\frac{\alpha}{2}}) \log(N_1) \\
&\quad \times \prod_{i=1}^3 L_i^{\frac{1}{2}} (1+L_i/N_i^3)^{\frac{1}{4}} \| f_{i,N_i,L_i} \|_2.
\end{split}
\end{equation*}
Taking into account time localization and derivative loss, we find the summability condition
\begin{equation*}
N_1^{\alpha-1} N_2^{1+s_3} (N_2^{-\frac{1}{4}} N_1^{-\frac{\alpha}{4}} + N_1^{-\frac{\alpha}{2}}) \log(N_1) \lesssim N_1^{2s_3}.
\end{equation*}
This yields
\begin{equation}
\label{eq:APrioriRegularityThresholdIII}
s_3(\alpha,\varepsilon) \geq \frac{\frac{\alpha}{2}-a_2}{1+a_2} + \varepsilon.
\end{equation}

 For $N_2 \gtrsim N_1^{1-a_2}$ we apply Proposition \ref{prop:TrilinearEstimateHighLowHigh}. In the resonant case $L_{\max} \sim N_1^2 N_2$ we find
 \begin{equation*}
 \begin{split}
&\quad \int (f_{1,N_1,L_1} * f_{2,N_2,L_2} ) f_{3,N_3,L_3} d\xi d\eta d\tau \\
&\lesssim_\varepsilon N_2^{-1} N_1^{-\frac{1}{4}+\frac{11 a_2}{24}+\varepsilon} \prod_{i=1}^3 L_i^{\frac{1}{2}} (1+L_i/N_i^3)^{\frac{1}{4}} \| f_{i,N_i,L_i} \|_2.
 \end{split}
 \end{equation*}
Taking into account the regularity of the functions and the factor $N_1^\alpha$ from time localization, we obtain the condition
\begin{equation*}
N_2^{s_3} N_1^\alpha N_1^{-\frac{1}{4}+\frac{11 a_2}{24}} \lesssim N_1^{2s_3}.
\end{equation*}
This gives
\begin{equation}
\label{eq:APrioriRegularityThresholdIV}
s_3(\alpha,\varepsilon) \geq - (\frac{1}{4} + \alpha + \frac{11 a_2}{48})/(1+a_2) + \varepsilon.
\end{equation} 
 
 In the non-resonant case we find
 \begin{equation*}
\begin{split}
&\quad \int (f_{1,N_1,L_1} * f_{2,N_2,L_2} ) f_{3,N_3,L_3} d\xi d\eta d\tau \\
&\lesssim_\varepsilon N_2^{-1} ( N_1^{\frac{a_4}{6}-\frac{1}{4}+\frac{11 a_2}{24}+\varepsilon} \vee N_1^{-a_4- \frac{\alpha}{2}+\varepsilon}) \prod_{i=1}^3 L_i^{\frac{1}{2}} (1+L_i/N_i^3)^{\frac{1}{4}} \| f_{i,N_i,L_i} \|_2.
\end{split}
\end{equation*}
The first estimate gives the condition taking into account time localization and derivative loss:
\begin{equation*}
N_2^{s_3} N_1^{\alpha + \frac{a_4}{6}-\frac{1}{4}+\frac{11 a_2}{24}} \lesssim N_1^{2s_3}.
\end{equation*}
This yields the bound
\begin{equation}
\label{eq:APrioriRegularityThresholdV}
s_3(\alpha,\varepsilon) \geq ( - \frac{1}{4} + \alpha + \frac{a_4}{6} + \frac{11 a_2}{24} ) / (1+a_2) + \varepsilon.
\end{equation}
The second estimate gives the condition
\begin{equation*}
N_2^{s_3} N_1^{\frac{\alpha}{2}-a_4} \lesssim N_1^{2s_3}.
\end{equation*}
This gives the condition
\begin{equation}
\label{eq:APrioriRegularityThresholdVI}
s_3(\alpha,\varepsilon) \geq \frac{\frac{\alpha}{2}-a_4}{1+a_2} + \varepsilon.
\end{equation}
 
We choose $a_2 = 3\alpha/2$ to simplify the conditions \eqref{eq:APrioriRegularityThresholdIII} and \eqref{eq:APrioriRegularityThresholdIV}:
\begin{equation}
\label{eq:APrioriRegularityThresholdIVSumm}
s_3(\alpha,\varepsilon) \geq - \frac{\alpha}{1 + 3\alpha/2} + \varepsilon, \quad s_3(\alpha,\varepsilon) \geq - \frac{3}{8} + \frac{44 \alpha}{24} + \varepsilon.
\end{equation}
\bigskip

For $0<a_1,\ldots,a_4 < \frac{1}{8}$ parameters, which are chosen like above to quantify ratios of frequencies, we collect the regularity thresholds:
\begin{itemize}
\item High$-$High$-$High-interaction  \eqref{eq:APrioriRegularityThresholdI}:
\begin{equation*}
s_3(\alpha,\varepsilon) \geq \alpha + \frac{a_1}{6} - \frac{1}{4} + \varepsilon.
\end{equation*}
\item High$-$High$-$High-interaction \eqref{eq:APrioriRegularityThresholdII}:
\begin{equation*}
s_3(\alpha,\varepsilon) \geq \frac{\alpha}{2} - a_1 + \varepsilon.
\end{equation*}
\item High$-$Low$-$High-interaction, $N_2 \lesssim N_1^{1-a_2}$ \eqref{eq:APrioriRegularityThresholdIII}:
\begin{equation*}
s_3(\alpha,\varepsilon) \geq \frac{\frac{\alpha}{2}-a_2}{1+a_2} + \varepsilon.
\end{equation*}
\item High$-$Low$-$High-interaction, $N_2 \gtrsim N_1^{1-a_2}$, resonant case \eqref{eq:APrioriRegularityThresholdIV}:
\begin{equation*}
s_3(\alpha,\varepsilon) \geq - (\frac{1}{4} + \alpha + \frac{11 a_2}{48})/(1+a_2) + \varepsilon.
\end{equation*}
\item High$-$Low$-$High-interaction, $N_2 \gtrsim N_1^{1-a_2}$, non-resonant case \eqref{eq:APrioriRegularityThresholdV}:
\begin{equation*}
s_3(\alpha,\varepsilon) \geq ( - \frac{1}{4} + \alpha + \frac{a_4}{6} + \frac{11 a_2}{24} ) / (1+a_2) + \varepsilon.
\end{equation*}
\item High$-$Low$-$High-interaction, non-resonant case  \eqref{eq:APrioriRegularityThresholdVI}:
\begin{equation*}
s_3(\alpha,\varepsilon) \geq \frac{\frac{\alpha}{2}-a_4}{1+a_2} + \varepsilon.
\end{equation*}
\end{itemize}
We choose $a_1 = \alpha$, $a_2 = a_4 = \frac{3 \alpha}{2}$ to find by \eqref{eq:APrioriRegularityThresholdII}, \eqref{eq:APrioriRegularityThresholdIII}, \eqref{eq:APrioriRegularityThresholdVI}, for the latter two we use $\alpha < \frac{1}{8}$,
\begin{equation}
\label{eq:FirstConditionEnergy}
s_3(\alpha,\varepsilon) \geq - \frac{\alpha}{2} + \varepsilon,
\end{equation}
and from \eqref{eq:APrioriRegularityThresholdI}, \eqref{eq:APrioriRegularityThresholdIV}, \eqref{eq:APrioriRegularityThresholdV}:
\begin{equation}
\label{eq:SecConditionsEnergy}
s_3(\alpha,\varepsilon) \geq \max(- \frac{1}{4} + \frac{7 \alpha}{6} + \varepsilon , - (\frac{1}{4} + \alpha + \frac{33 \alpha}{96})/(1+\frac{3\alpha}{2}) + \varepsilon , ( - \frac{1}{4} + \alpha + \frac{3 \alpha}{12} + \frac{33 \alpha}{96} ) / (1+\frac{3 \alpha}{2}) + \varepsilon).
\end{equation}

Choosing $0<\alpha<\frac{1}{10}$ we see that \eqref{eq:SecConditionsEnergy} does not give an additional constraint to \eqref{eq:FirstConditionEnergy}. This finishes the proof.

\end{proof}

\subsection{Estimates for differences of solutions}

To prove continuous dependence, it remains to establish energy estimates for differences of solutions $v = u_1 - u_2$:
\begin{equation}
\label{eq:EnergyEstimatesDifferences}
\| v \|^2_{E^{s'}(T)} \lesssim \| v_0 \|_{H^{s',0}(\T^2)}^2 + T^\delta \| v \|^2_{F^{s'}(T)} (\| u_1 \|_{F^s(T)} + \| u_2 \|_{F^s(T)} )
\end{equation}

The estimate is carried out in lower regularity $s' = 20s$. The main difference to the a priori estimates established above is that we can no longer always assume the derivative to act on the lowest frequency. We have to estimate expressions
\begin{equation*}
\iint_{[0,T] \times \T^2} P_{N_1} \partial_x v P_{N_2} v P_{N_3} u_i dx dy dt \text{ with } N_1 \sim N_3 \gg N_2,
\end{equation*}
where due to reduced symmetry we cannot integrate by parts to assign the derivative to $P_{N_2} v$. Here the lowered regularity comes to rescue.

\medskip

We show the following:
\begin{proposition}
\label{prop:EnergyEstimatesDifferences} Let $T \in (0,1]$, $0<\alpha<\frac{1}{36}$, and $s'=20s$, $s \geq s_4(\alpha,\varepsilon) = -\frac{\alpha}{2} + \varepsilon$, and $s \leq - \frac{5 \alpha}{12}$. Then we find \eqref{eq:EnergyEstimatesDifferences} to hold for $v=u_1 - u_2$ with $u_i \in C([0,T],L^2(\T^2))$ solutions to \eqref{eq:KPII}.
\end{proposition}
\begin{proof}
Like in the previous section we obtain an estimate for $\|P_N v(t) \|^2_{L^2}$, $t \in [0,T]$ by the fundamental theorem of calculus:
\begin{equation*}
\| P_N v(t) \|^2_{L^2} = \| P_N v(0) \|^2_{L^2} + \int_0^t \int_{\T^2} P_N v(s) \partial_x P_N (v(u_1+u_2))(s) dx dy ds.
\end{equation*}
The estimates for
\begin{equation*}
\int_0^t \int_{\T^2} P_N v(s) \partial_x P_N (v \, u_i)(s) dx dy ds
\end{equation*}
are carried out separately. In the following we denote $u_i$ with $u$ to lighten the notation. We use a paraproduct decomposition like above:
\begin{equation}
\label{eq:ParaproductDifferences}
P_N (v u)  = P_N( v P_{\ll N} u) + P_N (P_{\ll N} v u) + P_N (P_{\gtrsim N} u P_{\gtrsim N} v).
\end{equation}
For the first term the high frequency is again on $v$, for which reason we can still integrate by parts and reduce to estimates for
\begin{equation*}
\int_{0}^t \int_{\T^2} \tilde{P}_N v(s) \tilde{P}_N v(s) \partial_x P_K u dx dy ds, \quad K \ll N. 
\end{equation*}
$\tilde{P}_N$ denotes frequency projection to frequencies comparable to $N$ with a slightly larger projection.

Here we have to take into account time localization and regularity of the functions, which leads to the same conditions like above. This case gives no additional constraint on the summability properties. Our starting point based on Proposition \ref{prop:APrioriEnergyEstimates} is to assume
\begin{equation*}
s \geq s_4(\alpha,\varepsilon) = - \frac{\alpha}{2} + \varepsilon
\end{equation*}
and $0 < \alpha < \frac{1}{10}$.

\medskip

We turn to the second case, in which we need to establish estimates
\begin{equation*}
\begin{split}
&\quad \big| N_1 \int_0^t \int_{\T^2} P_{N_1} v P_{N_2} v P_{N_3} u \, dx dy ds \big| \\
&\lesssim C(N_1,N_2,N_3) \| P_{N_1} v \|_{F_{N_1}} \| P_{N_2} v \|_{F_{N_2}} \| P_{N_3} u \|_{F_{N_3}}, \quad N_1 \sim N_3 \gtrsim N_2.
\end{split}
\end{equation*}
This had also been essentially carried out in the previous section. What is left is to check the different summation properties now that the derivative acts on the high frequency.

To obtain estimates in short-time function spaces, we decompose $[0,t]$ into intervals of length $N_1^{-\alpha}$. So, we reduce to estimates established above
\begin{equation*}
\begin{split}
&\quad \big| \int_{\R \times \T^2} \gamma^3(N_1^\alpha(t-t_I)) P_{N_1} v P_{N_2} v P_{N_3} u dx dy dt \big| \\
&\lesssim C(N_1,N_2) \| P_{N_1} v \|_{F_{N_1}} \| P_{N_2} v \|_{F_{N_2}} \| P_{N_3} u \|_{F_{N_3}}.
\end{split}
\end{equation*}
The time localization incurs a factor of $N_1^\alpha$. We introduce the notation
\begin{equation*}
\begin{split}
f_{1,N_1} &= \mathcal{F}_{t,x,y}[ \gamma(N_1^\alpha(t-t_I)) P_{N_1} v], \quad f_{2,N_2} = \mathcal{F}_{t,x,y}[ \gamma(N_1^\alpha (t-t_I)) P_{N_2} v], \\
f_{3,N_3} &= \mathcal{F}_{t,x,y}[\gamma(N_1^\alpha(t-t_I)) P_{N_3} u ],
\end{split}
\end{equation*}
and localize in modulation like in \eqref{eq:ModulationLocalization}.

\medskip

With the estimates proved in \eqref{eq:ModulationLocalizedEnergyEstimates}, 
\begin{equation*}
\begin{split}
&\quad \big| \int (f_{1,N_1,L_1} * f_{2,N_2,L_2}) f_{3,N_3,L_3} d\xi d\eta d\tau \big| \\
&\lesssim C(N_1,N_2,N_3) \prod_{i=1}^3 L_i^{\frac{1}{2}} (1+L_i/N_i^3)^{\frac{1}{4}} \| f_{i,N_i,L_i} \|_{L^2},
\end{split}
\end{equation*}
we check the different summability properties. Clearly, nothing changes for the High$-$High$-$High-interaction. 

\medskip

$\bullet$ High$\times$Low $\rightarrow$ High-interaction: $N_2 \ll N_1 \sim N_3$. First, we suppose that $N_2 \lesssim N_1^{1-a_1}$. Applying Proposition \ref{prop:HighVeryLowHighTrilinearEstimate} yields
\begin{equation*}
\begin{split}
\int (f_{1,N_1,L_1} * f_{2,N_2,L_2}) f_{3,N_3,L_3} d\xi d\eta d\tau &\lesssim N_1^{-1} (N_2^{-\frac{1}{4}} N_1^{-\frac{\alpha}{4}} + N_1^{-\frac{\alpha}{2}}) \log(N_1) \\
&\quad \prod_{i=1}^3 L_i^{\frac{1}{2}} (1+L_i/N_i^3)^{\frac{1}{4}} \| f_{i,N_i,L_i} \|_2.
\end{split}
\end{equation*}
Taking into account the regularity and time localization we find the condition
\begin{equation*}
N_1^{1+\alpha} (N_1^2 N_2)^{-\frac{1}{2}} \langle  N_2^{\frac{1}{2}} N_1^{-\frac{\alpha}{2}} \rangle N_1^{40s} \lesssim N_1^{20s} N_2^{20s} N_1^s \Rightarrow N_1^{\frac{\alpha}{2}} N_1^{19s} \lesssim N_2^{20s}.
\end{equation*}
For $a_1 = \frac{1}{8}$ we obtain
\begin{equation*}
N_1^{\frac{\alpha}{2}+\frac{3}{2}s} \lesssim 1.
\end{equation*}
This gives the condition
\begin{equation}
\label{eq:DiffRegConditionI}
s \leq - \frac{\alpha}{3}.
\end{equation}
We have explained our choice of $s'=20s$ in Section \ref{section:Outline}. The compatibility of \eqref{eq:DiffRegConditionI} with $s \geq - \frac{\alpha}{2}+ \varepsilon$ is one of the reasons for our choice.

\medskip

Next, we suppose that $N_2 \gtrsim N_1^{1-a_1}$, $a_1 = 1/8$. First suppose that we are in the resonant case $L_{\max} \sim N_1^2 N_2$. In this case we apply Proposition \ref{prop:TrilinearEstimateHighLowHigh} to find
\begin{equation*}
\begin{split}
\big| \int (f_{1,N_1,L_1} * f_{2,N_2,L_2} ) f_{3,N_3,L_3} d\xi d\eta d\tau \big| &\lesssim N_2^{-1} N_1^{-\frac{1}{4}+\frac{11}{192}} \prod_{i=1}^3 L_i^{\frac{1}{2}} (1+L_i/N_i^3)^{\frac{1}{4}} \| f_{i,N_i,L_i} \|_{L^2} \\
&\lesssim N_1^{-1-\frac{13}{192}} \prod_{i=1}^3 L_i^{\frac{1}{2}} (1+L_i/N_i^3)^{\frac{1}{4}} \| f_{i,N_i,L_i} \|_{L^2}.
\end{split}
\end{equation*}

This gives the condition
\begin{equation*}
N_1^{1+\alpha} N_1^{40s} N_1^{-1-\frac{13}{192}} \lesssim N_2^{20s} N_1^{20s} N_1^s \Leftarrow N_1^{\alpha - \frac{13}{192}} \lesssim N_1^s.
\end{equation*}
This gives the condition on the regularity:
\begin{equation}
\label{eq:DiffRegConditionII}
s_4(\alpha,\varepsilon) \geq \alpha -\frac{13}{192} + \varepsilon.
\end{equation}

$\cdot$ Non-resonant case: $L_{\max} \sim M^2 N_2$ for some $M \gg N_1$. We apply Proposition \ref{prop:TrilinearEstimateHighLowHigh} in the non-resonant case
\begin{equation*}
\big| \int (f_{1,N_1,L_1} * f_{2,N_2,L_2} ) f_{3,N_3,L_3} d\xi d\eta d\tau \big| \lesssim \frac{M^{\frac{1}{6}} (N_1 N_2)^{\frac{1}{8}+\varepsilon}}{N_2 (N_1 N_2)^{\frac{1}{12}}} \frac{1}{N_2^{\frac{1}{2}}} \prod_{i=1}^3 L_i^{\frac{1}{2}} \| f_{i,N_i,L_i} \|_{L^2}.
\end{equation*}
We choose $M \leq N_1^{1+a_4}$ and let $a_4 = \frac{1}{8}$. Then we compute for the above constant remembering $N_2 \gtrsim N_1^{1-a_3}$, $a_3 = \frac{1}{8}$ taking into account derivative loss and time localization
\begin{equation*}
M^{\frac{1}{6}} (N_1 N_2)^{\frac{1}{24}} \frac{1}{N_2^{\frac{3}{2}}} N_1^{1+\alpha} \leq N_1^{\frac{9}{48}} N_1^{\frac{1}{12}} N_1^{- \frac{21}{16}} N_1^{1+\alpha} = N_1^{\frac{13}{48}+\alpha} N_1^{-\frac{15}{48}} = N_1^{\alpha - \frac{1}{24}}.
\end{equation*}
Taking into account the regularity we find the condition
\begin{equation*}
N_1^{40s} N_1^{\alpha - \frac{1}{24}} \lesssim N_1^{20s} N_2^{20s} N_1^s,
\end{equation*}
which leads to the condition on the regularity
\begin{equation}
\label{eq:DiffRegConditionIII}
s_4(\alpha,\varepsilon) \geq \alpha - \frac{1}{24} + \varepsilon.
\end{equation}

$\cdot$ $M \geq N_1^{1+a_4}$ can be estimated by the second bound provided by Proposition \ref{prop:TrilinearEstimateHighLowHigh} in the non-resonant case:
\begin{equation*}
\big| \int (f_{1,N_1,L_1} * f_{2,N_2,L_2} ) f_{3,N_3,L_3} d\xi d\eta d\tau \big| \lesssim (M^2 N_2)^{-\frac{1}{2}} \langle N_2^{\frac{1}{2}} N_1^{-\frac{\alpha}{2}} \rangle \prod_{i=1}^3 L_i^{\frac{1}{2}} \| f_{i,N_i,L_i} \|_2.
\end{equation*}
For $a_4 = \frac{1}{8}$ we find taking into account derivative loss and time localization $N_1^{1+\alpha}$ the summability condition:
\begin{equation*}
N_1^{40 s} N_1^{-\frac{1}{8}+\frac{\alpha}{2}} \lesssim N_1^{20s} N_2^{20s} N_1^s.
\end{equation*}
This yields the regularity condition
\begin{equation}
\label{eq:DiffRegConditionIV}
s_4(\alpha,\varepsilon) \geq - \frac{1}{8}+\frac{\alpha}{2} + \varepsilon.
\end{equation}

\medskip

$\bullet$ High$\times$High $\rightarrow$ Low-interaction: In this case we need to establish estimates
\begin{equation*}
\big| \int_0^t \int_{\T^2} P_{N_2} v P_{N_1} u P_{N_3} v dx dy dt \big| \lesssim C(N_1,N_2,N_3) \| P_{N_2} v \|_{F_{N_2}} \| P_{N_1} u \|_{F_{N_1}} \| P_{N_3} v \|_{F_{N_3}}
\end{equation*}
for $N_2 \ll N_1 \sim N_3$, which corresponds to the third case in the paraproduct decomposition \eqref{eq:ParaproductDifferences}. The summability condition taking into account regularity of the functions, time localization and derivative loss is given by
\begin{equation}
\label{eq:SumConditionDiffII}
N_2^{40s} N_2 N_1^\alpha C(N_1,N_2,N_3) \lesssim N_1^{40s} N_2^s.
\end{equation} 
In the High$\times$Low $\rightarrow$ High-case we have checked the summability condition
\begin{equation*}
N_1^{40s} N_1^{1+\alpha} C(N_1,N_2,N_3) \lesssim N_1^{21s} N_2^{20s}.
\end{equation*}
We use this to validate \eqref{eq:SumConditionDiffII} writing
\begin{equation*}
\frac{N_2^{40s} N_2}{N_1^{40s} N_1} N_1^{40s} N_1^{1+\alpha} C(N_1,N_2,N_3) \lesssim \frac{N_2^{40s} N_2}{N_1^{40s} N_1} N_1^{21s} N_2^{20s}.
\end{equation*}
To conclude, we need to check that
\begin{equation*}
\frac{N_2^{40s} N_2}{N_1^{40s} N_1} N_1^{21s} N_2^{20s} \lesssim N_1^{40s} N_2^s.
\end{equation*}
This is the case if
\begin{equation*}
N_2^{1+59s} \lesssim N_1^{1+59s}
\end{equation*}
and gives the requirement
\begin{equation}
\label{eq:DiffRegConditionVII}
 s_4(\alpha,\varepsilon) \geq - \frac{1}{59} + \varepsilon.
\end{equation}

We collect the conditions on the regularities obtained above:
\begin{itemize}
\item High$\times$High $\rightarrow$ High-interaction: The analysis from Proposition \ref{prop:APrioriEnergyEstimates} applies and for $0<\alpha<\frac{1}{10}$ we have the only requirement:
\begin{equation*}
s_4(\alpha,\varepsilon) \geq - \frac{\alpha}{2} + \varepsilon.
\end{equation*}
\item High$\times$Low $\rightarrow$ High-interaction, $N_2 \lesssim N_1^{1-a_1}$ \eqref{eq:DiffRegConditionI}:
\begin{equation*}
s \leq - \frac{\alpha}{3}.
\end{equation*}
\item High$\times$Low $\rightarrow$ High-interaction, $N_2 \gtrsim N_1^{\frac{7}{8}}$, Resonant case \eqref{eq:DiffRegConditionII}:
\begin{equation*}
s_4(\alpha,\varepsilon) \geq \alpha -\frac{13}{192} + \varepsilon.
\end{equation*}
\item High$\times$Low $\rightarrow$ High-interaction, $N_2 \gtrsim N_1^{\frac{7}{8}}$, Non-resonant case \eqref{eq:DiffRegConditionIII}, \eqref{eq:DiffRegConditionIV}:
\begin{equation*}
s_4(\alpha,\varepsilon) \geq \alpha - \frac{1}{24} + \varepsilon \wedge s_4(\alpha,\varepsilon) \geq - \frac{1}{8}+\frac{\alpha}{2} + \varepsilon.
\end{equation*}
\item High$\times$High $\rightarrow$ Low-interaction: To ensure summability in this case with the estimates from the dual High$\times$Low $\rightarrow$ High-interaction at hand, we need to choose
\begin{equation*}
s_4(\alpha,\varepsilon) \geq - \frac{1}{59} + \varepsilon.
\end{equation*}

\end{itemize}

To fit these conditions on the regularities under the umbrella
\begin{equation*}
s_4(\alpha,\varepsilon) \geq - \frac{\alpha}{2} + \varepsilon \wedge s \leq - \frac{\alpha}{3},
\end{equation*}
we impose a lower bound on the time localization. 
 The conditions obtained in the non-resonant case with high modulation on the high frequency lead us to the upper bound $0 < \alpha < \frac{1}{36}$. We see that choosing $s=-\frac{5 \alpha}{12}$ satisfies all conditions on the regularity for this range of $\alpha$. The proof is complete.
\end{proof}

\section{Improved local well-posedness}
\label{section:ImprovedLWP}

In this section we conclude the local well-posedness below $L^2$ based on the short-time analysis from previous sections. Starting point of our analysis is Bourgain's global $L^2$-well-posedness result for real-valued initial data with persistence of regularity:
\begin{theorem}[{Global well-posedness of KP-II, \cite[Theorem~1.4]{Bourgain1993}}]
\label{thm:GWPKPIIL2}
Let $u_0 \in L^2(\T^2)$ be a real-valued function with $\int_{\T} u_0(x,y) dx = 0$. Then \eqref{eq:KPII} is globally well-posed.
\end{theorem}

\medskip

We denote the data-to-solution mapping by $S_*^0 : L^2(\T^2) \rightarrow C(\R;L^2(\T^2))$, which assigns  real-valued initial data $u_0$ with $\int_{\T} u_0(x,y) dx = 0$ to global solutions $u$ to \eqref{eq:KPII}.

We show the following improved local well-posedness result, which has been summarized in the Introduction:
\begin{theorem}[Improved~local~well-posedness~for~KP-II]
Let $0>s>- \frac{1}{90}$, and $R>0$. Then $S^0_*$ admits a unique continuous extension
\begin{equation*}
S_T^s: H^{s,0}(\T^2) \to C([0,T],H^{s,0}(\T^2))
\end{equation*}
with $T=T(s,R)$ provided that $\| u_0 \|_{H^{s,0}(\T^2)} \leq R$.
\end{theorem}

The proof is carried out in three steps:
\begin{itemize}
\item[(i)] We establish a priori estimates for solutions to \eqref{eq:KPII} at negative Sobolev regularities on times $T=T(s,\|u_0\|_{H^{s,0}(\T^2)})$.
\item[(ii)] We prove a priori estimates for differences of solutions at lower regularities $s'=20s$ on times $T=T(s,\| u_{01} \|_{H^{s,0}(\T^2)}, \| u_{02} \|_{H^{s,0}(\T^2)})$.
\item[(iii)] The proof of continuous dependence in the $H^{s,0}(\T^2)$-topology is concluded using frequency envelopes.
\end{itemize}

We begin with a priori estimates:
\begin{theorem}[A~priori~estimates~at~Sobolev~regularity~of~negative~order]
\label{thm:APrioriEstimates}
Let $s>-\frac{1}{20}$ and $u_0 \in L^2(\T^2)$ a  real-valued function with vanishing mean $\int_{\T} u_0(x,y) dx = 0$. Then there is $T=T(s,\| u_0 \|_{H^{s,0}})$ such that we have the following a priori estimates:
\begin{equation}
\label{eq:APriori}
\| u \|_{F^{\bar{s}}(T)} \lesssim \| u_0 \|_{H^{\bar{s},0}(\T^2)}
\end{equation}
for any $\bar{s} \in [s,0)$. $T=T(s,\| u_0 \|_{H^{s,0}})$ can be chosen to be decreasing in $\| u_0 \|_{H^{s,0}}$.
\end{theorem}
\begin{proof}
We choose $0<\alpha<\frac{1}{10}$ and $\varepsilon>0$ such that $s=-\frac{\alpha}{2}+\varepsilon$.
Note that by Theorem \ref{thm:GWPKPIIL2} we have global solutions emanating from $u_0 \in L^2(\T^2)$. 

We invoke Lemma \ref{lem:LinearEnergyEstimate}, Proposition \ref{prop:ShorttimeBilinearEstimates}, and Proposition \ref{prop:APrioriEnergyEstimates} to find $\delta>0$ such that the following set of estimates holds with frequency-dependent time localization given by $T=T(N)=N^{-\alpha}$:
\begin{equation*}
\left\{ \begin{array}{cl}
\| u \|_{F^{s}(T)} &\lesssim \| u \|_{E^{s}(T)} + \| \partial_x(u^2) \|_{\mathcal{N}^{s}(T)}, \\
\| \partial_x(u^2) \|_{\mathcal{N}^{s}(T)} &\lesssim T^\delta \| u \|^2_{F^{s}(T)}, \\
\| u \|^2_{E^{s}(T)} &\lesssim \| u_0 \|_{H^{s,0}(\T^2)}^2 + T^\delta \| u \|^3_{F^{s}(T)}.
\end{array} \right.
\end{equation*} 
Taking the estimates together gives
\begin{equation}
\label{eq:PropagationSolution}
\| u \|_{F^{s}(T)} \lesssim \| u_0 \|_{H^{s,0}(\T^2)} + T^\delta \| u \|_{F^{s}(T)}^2 + T^\delta \| u \|_{F^{s}(T)}^{3/2}.
\end{equation}
By the limiting properties (see \cite{IonescuKenigTataru2008}) of the short-time function spaces it holds
\begin{equation*}
\limsup_{T \downarrow 0} \| u \|_{E^{s}(T)} \lesssim \| u_0 \|_{H^{s,0}(\T^2)} \text{ and } \lim_{T \downarrow 0} \| \partial_x (u^2) \|_{\mathcal{N}^{s}(T)} = 0,
\end{equation*}
therefore we find
 \begin{equation*}
 \limsup_{T \downarrow 0} \| u \|_{F^s(T)} \lesssim \| u_0 \|_{H^{s,0}(\T^2)}.
 \end{equation*}
With this limiting property at hand, it is standard to argue by a continuity argument that \eqref{eq:PropagationSolution} implies \eqref{eq:APriori} for $\bar{s} =s$. To obtain the persistence of regularity, we consider the same time localization and invoke Lemma \ref{lem:LinearEnergyEstimate}, Proposition \ref{prop:ShorttimeBilinearEstimates}, and Proposition \ref{prop:APrioriEnergyEstimates} again to find the set of estimates:
\begin{equation*}
\left\{ \begin{array}{cl}
\| u \|_{F^{\bar{s}}(T)} &\lesssim \| u \|_{E^{\bar{s}}(T)} + \| \partial_x(u^2) \|_{\mathcal{N}^{\bar{s}}(T)}, \\
\| \partial_x(u^2) \|_{\mathcal{N}^{\bar{s}}(T)} &\lesssim T^\delta \| u \|_{F^{s}(T)} \| u \|_{F^{\bar{s}}(T)}, \\
\| u \|^2_{E^{\bar{s}}(T)} &\lesssim \| u_0 \|_{H^{\bar{s},0}(\T^2)}^2 + T^\delta \| u \|^2_{F^{\bar{s}}(T)} \| u \|_{F^s(T)}.
\end{array} \right.
\end{equation*}
This can be subsumed as
\begin{equation*}
\| u \|_{F^{\bar{s}}(T)} \lesssim \| u_0 \|_{H^{\bar{s},0}(\T^2)} + T^\delta \| u \|_{F^{\bar{s}}(T)} (\| u \|_{F^s(T)} + \| u \|_{F^s(T)}^{\frac{1}{2}}).
\end{equation*}
By the a priori estimates
\begin{equation*}
\| u \|_{F^s(T)} \lesssim \| u_0 \|_{H^{s,0}(\T^2)}
\end{equation*}
already established, we find upon choosing $T=T(s,\| u_0 \|_{H^{s,0}(\T^2)})$ smaller if necessary
\begin{equation*}
\| u \|_{F^{\bar{s}(T)}} \lesssim \| u_0 \|_{H^{\bar{s},0}(\T^2)}.
\end{equation*}
The proof is complete.
\end{proof}

In the next step we show Lipschitz continuous dependence in a low regularity $s'=20s$ for some $s<0$, which depends only on the $H^{s,0}(\T^2)$-norm of the initial data.

\begin{theorem}[Lipschitz continuous dependence in rough topology] 
\label{thm:DifferenceEstimates}
Let $s>-\frac{1}{90}$ and $u_{01}, u_{02} \in L^2(\T^2)$ real-valued initial data with vanishing mean $\int_{\T} u_{0i}(x,y) dx = 0$. Then there is $T=T(s,\| u_{0i} \|_{H^{s,0}})$ such that we have the following estimate on the difference of solutions $v=u_1-u_2$ with $u_i$ denoting the global solution emanating from $u_{0i}$:
\begin{equation}
\label{eq:LipDependence}
\| v \|_{F^{20s}(T)} \lesssim \| v(0) \|_{H^{20s,0}(\T^2)}.
\end{equation}
$T=T(s, \| u_{0i} \|_{H^{s,0}})$ can be chosen to be decreasing in $\| u_{0i} \|_{H^{s,0}}$, $i=1,2$.
\end{theorem}
\begin{proof}
Choose $0<\alpha<\frac{1}{45}$ and $\varepsilon>0$ such that there $s \geq - \frac{\alpha}{2} + \varepsilon$ and $s \leq - \frac{5 \alpha}{12}$.

We invoke Lemma \ref{lem:LinearEnergyEstimate}, Proposition \ref{prop:ShorttimeBilinearEstimates}, and Proposition \ref{prop:EnergyEstimatesDifferences} to find the following set of estimates for some $\delta > 0$ with the time localization $T=T(N)=N^{-\alpha}$:
\small
\begin{equation*}
\left\{ \begin{array}{cl}
\| v \|_{F^{20s}(T)} &\lesssim \| v \|_{E^{20s}(T)} + \| \partial_x(v (u_1+u_2)) \|_{\mathcal{N}^{20s}(T)}, \\
\| \partial_x(v (u_1+u_2)) \|_{\mathcal{N}^{20s}(T)} &\lesssim T^\delta \| v \|_{F^{20s}(T)} (\| u_1 \|_{F^s(T)} + \| u_2 \|_{F^s(T)}), \\
\| v \|^2_{E^{20s}(T)} &\lesssim \| v(0) \|_{H^{20s,0}}^2 + T^\delta \| v \|^2_{F^{20s}(T)} (\| u_1 \|_{F^s(T)} + \| u_2 \|_{F^s(T)}).
\end{array} \right.
\end{equation*}
\normalsize

\medskip

Taking the estimates together we find
\begin{equation*}
\begin{split}
\| v \|_{F^{20s}(T)} &\lesssim \| v_0 \|_{H^{20s,0}(\T^2)} + T^\delta \| v \|_{F^{20s}(T)} ( \| u_1 \|_{F^s(T)} + \| u_2 \|_{F^s(T)}) \\
&\quad + T^\delta \| v \|_{F^{20s}(T)} (\|u_1 \|_{F^s(T)} + \| u_2 \|_{F^s(T)})^{\frac{1}{2}}.
\end{split}
\end{equation*}
Choosing $T=T(s,\|u_{0i}\|_{H^{s,0}(\T^2)})$ we can invoke Theorem \ref{thm:APrioriEstimates} to obtain
\begin{equation*}
\| u_i \|_{F^s(T)} \lesssim \| u_{i0} \|_{H^{s,0}(\T^2)}.
\end{equation*}
Hence, the previous estimate reduces to
\begin{equation*}
\begin{split}
\| v \|_{F^{20s}(T)} &\lesssim \| v(0) \|_{H^{20s,0}(\T^2)} + T^\delta \| v \|_{F^{20s}(T)} ( \| u_{10} \|_{H^{s,0}(\T^2)} + \| u_{20} \|_{H^{s,0}(\T^2)} ) \\
&\quad + T^\delta \| v \|_{F^{20s}(T)} ( \| u_{10} \|_{H^{s,0}(\T^2)} + \| u_{20} \|_{H^{s,0}(\T^2)} )^{\frac{1}{2}}.
\end{split}
\end{equation*}
Choosing $T=T(\|u_{i0} \|_{H^{s,0}(\T^2)})$ smaller if necessary, we obtain the estimate \eqref{eq:LipDependence}.
\end{proof}

\medskip

We turn to the final part in the argument: the conclusion of local well-posedness via frequency envelopes. We follow the presentation of Ifrim--Tataru \cite{IfrimTataru2023} and shall be brief to avoid repetition.

We introduce the notation
\begin{equation*}
\big[ \frac{N}{M} \big] = \max\big( \frac{N}{M}, \frac{M}{N} \big).
\end{equation*}

\begin{definition}[Frequency~envelopes~in~$H^{s,0}(\T^2)$]
Let $u_0 \in H^{s,0}(\T^2)$. A sequence $(c_N)_{N \in 2^{\N}}$ is referred to as (sharp) frequency envelope for $u_0$, if it satisfies, for some $\delta>0$
\begin{itemize}
\item[(a)] energy estimates $\| P_N u_0 \|_{H^{s,0}(\T^2)} \leq c_N$,
\item[(b)] the slowly varying property: $\frac{c_N}{c_M} \lesssim \big[ \frac{N}{M} \big]^{\delta}$, 
\item[(c)] $\| u_0 \|^2_{H^{s,0}(\T^2)} \approx \sum_{N \geq 1} c_N^2$.
\end{itemize}
\end{definition}
Before we construct sharp frequency envelopes for given initial data, we derive a frequency envelope bound for solutions to truncated initial data. Suppose that $u_0^H=P_{\leq H} u_0$ and that $(c_N)$ is a sharp frequency envelope of $u_0^H$. Then, the following bounds hold:
\begin{itemize}
\item[(i)] Uniform bounds: $\| P_N u_0^H \|_{H^{s,0}(\T^2)} \lesssim c_N$,
\item[(ii)] High frequency bounds: $\| u_0^H \|_{H^{s+j,0}(\T^2)} \lesssim H^j c_H$,
\item[(iii)] Difference bounds: $\| u_0^{2H} - u_0^H \|_{H^{20s,0}(\T^2)} \lesssim H^{19s} c_H$.
\end{itemize}

As a consequence of the a priori estimates from Theorem \ref{thm:APrioriEstimates} and the estimates for solutions to the difference equation Theorem \ref{thm:DifferenceEstimates}, we have the following estimates for the  global solutions $u^H$ emanating from $u_0^H$ on the time scale $T=T(s,\| u_0 \|_{H^{s,0}(\T^2)})$:
\begin{itemize}
\item[(i)] High frequency bounds:
\begin{equation}
\label{eq:FrequencyEnvelopeSolutions}
\| u^H \|_{C_TH^{s+j,0}} \lesssim_j H^j c_H, \quad j > 0,
\end{equation}
\item[(ii)] Difference bounds:
\begin{equation}
\label{eq:FrequencyEnvelopeDifferences}
\| u^{2H} - u^H \|_{C_TH^{20s,0}} \lesssim H^{19s} c_H.
\end{equation}
\end{itemize}

We obtain as a consequence
\begin{equation}\label{eq:env}
\| P_K u^H \|_{C_TH^{s,0}} \lesssim_N c_K (K/H \vee 1)^{-N}.
\end{equation}
By using the difference bounds \eqref{eq:FrequencyEnvelopeDifferences} we show first that $u^H$ is a Cauchy sequence in $C([0,T],H^{20s,0}(\T^2))$. Indeed, for $H'>H$, we have
\begin{equation*}
\| u^{H'} - u^H \|_{C_TH^{20s,0}} \lesssim \sum_{H\leq M <H'} \| u^{2M} - u^M \|_{C_TH^{20s,0}}\lesssim H^{19s}\big(\sum_{M\geq H} c_M^2\big)^{\frac12},
\end{equation*}
which implies convergence to a limiting solution $u \in C([0,T], H^{20s,0}(\T^2))$ satisfying
\begin{equation*}
\| u - u^H \|_{C_TH^{20s,0}} \lesssim  H^{19s}\big(\sum_{M\geq H} c_M^2\big)^{\frac12}.
\end{equation*}
Next, we upgrade this convergence to  $C([0,T], H^{s,0}(\T^2))$:
\begin{align*}
& \| u - u^H \|_{C_TH^{s,0}} \leq \| P_{\leq H}(u - u^H) \|_{C_T H^{s,0}}+\|P_{> H}( u - u^H)\|_{C_TH^{s,0}}
\end{align*}
Concerning the first term, we simply use the above bound to obtain
\[\| P_{\leq H}(u - u^H) \|_{C_TH^{s,0}}\lesssim H^{-19s}\|u - u^H \|_{C_T H^{20s,0}}\lesssim \big(\sum_{M\geq H} c_M^2\big)^{\frac12}.\]
We decompose the second term
\begin{align*}
    \|P_{> H}( u - u^H)\|_{C_TH^{s,0}}\lesssim \Big( \sum_{K>H}   \|P_{K}( u - u^H)\|_{C_T H^{s,0}}^2\Big)^{\frac12}.
\end{align*}
Then, $\|P_{K}( u - u^H)\|_{C_T H^{s,0}}\leq I_{\leq K}+I_{>K}$ for
\begin{align*}
   I_{\leq K} =&\sum_{M \leq K }\|P_K(u^{2M}-u^M)\|_{C_TH^{s,0}},\\
I_{> K} =&\sum_{M >K}\|P_K(u^{2M}-u^M)\|_{C_TH^{s,0}}
\end{align*}
From \eqref{eq:env} (with $N=1$) and the slowly varying property we obtain
\[
I_{\leq K} \lesssim  \sum_{M\leq K } \|P_K(u^M)\|_{C_TH^{s,0}}\lesssim \sum_{M\leq K}
\frac{M}{K} c_M\lesssim c_K.
\]
From \eqref{eq:FrequencyEnvelopeDifferences}  and the slowly varying property  we infer that
\begin{align*}
I_{>K} \lesssim & \sum_{ M>K } K^{-19s}\|P_K(u^{2M}-u^M)\|_{C_TH^{20s,0}}\\
\lesssim & \sum_{M>K} \Big(\frac{K}{M}\Big)^{-19s} c_M\lesssim c_K.
\end{align*}
In summary, we obtain
\begin{equation}\label{eq:H-conv}
 \| u - u^H \|_{C_T H^{s,0}} \lesssim  \big(\sum_{M \geq H} c_M^2\big)^{\frac12}, 
\end{equation}
which tends to $0$ as $H \to \infty$.
\medskip

Consequently, we obtain an approximation of $u$ with  global solutions $u^H$. This is the key ingredient in the proof of continuous dependence, for which we consider next a sequence $u_{j,0} \to u_0$ in $H^{s,0}(\T^2)$. The emanating solutions are denoted by $u_j$, $u$, respectively. The frequency truncations are denoted by $u_{j,0}^H$, $u_0^H$, respectively, and the emanating solutions of the frequency truncations by $u^H_j$ and $u^H$.

We sketch the construction of a sharp 
  frequency envelope $(c_N)$ for $u_0$.  As initial guess we consider $d_N = \| P_N u_0 \|_{H^{s,0}(\T^2)}$. A frequency envelope is obtained from mollifying: $c_N = \sum_{M \in 2^{\N_0}} d_M \big[ \frac{M}{N} \big]^{-\delta}$, for some $0<\delta<19|s|$. By Young's inequality it holds
\begin{equation*}
    \| (c_N) \|_{\ell^2_N} \lesssim \| (d_N) \|_{\ell^2} \sum_{M \in 2^{\N_0}} M^{-\delta} \lesssim \| u_0 \|_{H^{s,0}},
\end{equation*}
therefore property c) is satisfied.
To check the slowly varying property, we note that
\begin{equation*}
    c_N = \sum_{J \in 2^{\N_0}} d_{J} \Big[ \frac{J}{N}\Big]^{-\delta} \leq \sum_{J \in 2^{\N_0}} d_{J} \Big[ \frac{J}{M}\Big]^{-\delta}\Big[ \frac{M}{N}\Big]^{\delta}\lesssim c_M \big[\frac{M}{N}]^\delta.
\end{equation*}
Similarly, we construct a sharp frequency envelope $(c^{(j)}_N)$ for $u_{0,j}$ such that we have $c^{(j)}\to c$ in $\ell^2$.

We use the triangle inequality and \eqref{eq:H-conv} to write
\begin{equation*}
\begin{split}
\| u_j - u \|_{C_T H^{s,0}} &\leq \| u_j - u_j^H \|_{C_T H^{s,0}} + \| u_j^H - u^H \|_{C_T H^{s,0}} \\
&\quad + \| u^H - u \|_{C_T H^{s,0}} \\
&\lesssim \| c^{(j)}_{\geq H} \|_{\ell^2}+ \| u_j^H - u^H \|_{C_T H^{s,0}} + \| c_{\geq H} \|_{\ell^2} .
\end{split}
\end{equation*}
By Theorem \ref{thm:GWPKPIIL2}, we have as $j \to \infty$
\begin{equation*}
\| u_j^H - u^H \|_{C_T H^{s,0}} \leq \| u_j^H - u^H \|_{C_T L^2} \lesssim \| u_{j,0}^H - u_0^H \|_{L^2} \to 0.
\end{equation*}
Consequently,
\begin{equation*}
\limsup_{j \to \infty} \| u_j -u \|_{C_T H^{s,0}} \lesssim \| c_{\geq H} \|_{\ell^2}.
\end{equation*}
Finally, letting $H \to \infty$, we infer
\begin{equation*}
\lim_{j \to \infty} \| u_j -u \|_{C_T H^{s,0}} = 0
\end{equation*}
and the proof is complete. $\hfill \Box$

\section{Sharpness of Bourgain's bilinear Strichartz estimate}\label{sect:ex}
In this section, we prove that Bourgain's estimate provided in Proposition \ref{prop:BilinearStrichartzBourgain} is sharp in general. To this end, we construct $f_{i,N_i,L_i}$ such that
\begin{equation*}
\| f_{1,N_1,L_1} * f_{2,N_2,L_2} \|_{L^2_{\tau,\xi,\eta}} \sim N_2^{\frac{1}{4}} \min( L_1 , L_2)^{\frac{1}{2}} \max(L_1, L_2)^{\frac{1}{4}} \prod_{i=1}^2 \| f_{i,N_i,L_i} \|_{L^2_{\tau,\xi,\eta}}.
\end{equation*}
We use the same example as in \cite{Gruen09} (which in turn goes back an unpublished manuscript of the third author).
Let $f_{i,N_i,L_i}$ be a characteristic function,  supported on a single $\xi$-node with $\xi_i = N_i$, $1 \leq N_2 \leq N_1$, and the $\eta$-support be intervals of length $N_2^{\frac{1}{2}}$. Then we can arrange that for $(\tau_i,\xi_i,\eta_i) \in \operatorname{supp}(f_{i,N_i,L_i})$:
\begin{equation*}
\big| \frac{\eta_1}{\xi_1} - \frac{\eta_2}{\xi_2} \big| \lesssim N_2^{-\frac{1}{2}}.
\end{equation*}
Moreover, we let $\tau_i$ range in intervals of length $\sim 1$, such that
\begin{equation*}
\big| \tau_i - (\xi_i^3 - \frac{\eta_i^2}{\xi_i}) \big| \lesssim 1.
\end{equation*}
Clearly,
\begin{equation*}
\| f_{i,N_i,L_i} \|_{L^2} \sim N_2^{\frac{1}{4}}.
\end{equation*}
And moreover, on a $\tau$-interval of length $1$, for $\xi = \xi_1 + \xi_2$ and for $\eta$ in an interval of length $\sim N_2^{\frac{1}{2}}$ we have
\begin{equation*}
f_{1,N_1,L_1} * f_{2,N_2,L_2}(\tau,\xi,\eta) \sim N_2^{\frac{1}{2}},
\end{equation*}
and the support contains an  $\eta$-interval of length $\sim N_2^{\frac{1}{2}}$.
Consequently,
\begin{equation*}
\| f_{1,N_1,L_1} * f_{2,N_2,L_2} \|_{L^2} \gtrsim N_2^{\frac{1}{2}} N_2^{\frac{1}{4}}.
\end{equation*}
We obtain
\begin{equation*}
\| f_{1,N_1,L_1} * f_{2,N_2,L_2} \|_{L^2} \gtrsim N_2^{\frac{1}{4}} \min(L_1, L_2)^{\frac{1}{2}} \max(L_1,  L_2)^{\frac{1}{4}} \prod_{i=1}^2 \| f_{i,N_i,L_i} \|_{L^2}.
\end{equation*}

Based on this example, the same computation as in \cite{Gruen09} for the cylinder $\T\times\R$ can be done in our setting, which shows $C^3$-ill-posedness of the Cauchy problem in $H^{s,0}(\T^2)$ for $s<- \frac{1}{4}$. Semilinear ill-posedness for $s \in (-\frac{1}{4},0)$ remains open. As mentioned in Remark \ref{rem:quasi}, semilinear local well-posedness hinges on an improvement of Bourgain's bilinear $L^2$-estimate.

\section*{Acknowledgements}
\small
\begin{enumerate}
  \item
Funded by the Deutsche Forschungsgemeinschaft (DFG, German Research Foundation) -- Project-ID 317210226 -- SFB 1283.
\item
S.H.\ thanks the \emph{Laboratoire de Math\'ematique d'Orsay -- Universit\'e Paris-Saclay} for its hospitality and funding via \emph{CNRS Poste Rouge} is gratefully acknowledged.
\item R.S.\ gratefully acknowledges financial support from the Humboldt foundation (Feodor-Lynen fellowship) and partial support by the NSF grant DMS-2054975. 
\item N.T. was partially supported by the ANR project Smooth ANR-22-CE40-0017 and by the European research Council (ERC) under the European Union Horizon 2020 research and innovation programme (Grant agreement 101097172 - GEOEDP).
\end{enumerate}
We thank the referees for helpful remarks, which significantly improved the presentation.

\normalsize

\bibliographystyle{plain}

\end{document}